\documentclass[10pt]{amsart}

\makeatletter
\def\@settitle{\begin{center}%
		\baselineskip14\p@\relax
		\normalfont\Large
		\@title
	\end{center}%
}
\makeatletter
\def\@settitle{\begin{center}%
		\baselineskip14\p@\relax
		\normalfont\Large
		\@title
	\end{center}%
}
\makeatother

\usepackage{graphicx,epsfig}
\usepackage{mathabx,epsfig}
\usepackage{bbm}
\usepackage[title]{appendix}
\usepackage{esvect}
\usepackage{tikz-cd}
\usepackage{fancyhdr}
\usepackage{amsfonts}
\usepackage{amsmath}
\usepackage{amssymb}
\usepackage{mathabx}
\usepackage{amsthm}
\usepackage{tikz-cd}
\usepackage{parskip}
\usepackage{times} 

\usepackage{mathrsfs}
\usepackage{scalerel,amssymb}

\usepackage{amsmath,amssymb,latexsym, amsfonts, amscd, amsthm, xy}
\input{xy}
\xyoption{all}

\input{xy}
\xyoption{all}

\makeindex 

\usepackage{fancyhdr}
\usepackage{amsfonts}
\usepackage{amsmath}
\usepackage{amssymb}

\usepackage{amsthm}
\usepackage{tikz}
\usepackage{tikz-cd}
\usepackage{parskip}
\usepackage{times} 

\usepackage[style=alphabetic,backend=biber]{biblatex}

\usepackage{hyperref}
\hypersetup{pdftoolbar=true, pdftitle={GKform}, pdffitwindow=true, colorlinks=true, citecolor=green, filecolor=black, linkcolor=blue, urlcolor=blue, hypertexnames=false}
\usepackage{cleveref}
\makeindex 

\newtheorem{thm}{Theorem}[section]
\newtheorem{prop}[thm]{Proposition}
\newtheorem{lem}[thm]{Lemma}
\newtheorem{cor}[thm]{Corollary}

\newtheorem{defn}[thm]{Definition}
\theoremstyle{definition}
\newtheorem{rem}[thm]{Remark}
\newtheorem{ex}[thm]{Example}
\newcommand{\Z}{\mathbb{Z}}
\newcommand{\cZ}{\mathcal{Z}}
\newcommand{\Q}{\mathbb{Q}}
\newcommand{\C}{\mathbb{C}}
\newcommand{\R}{\mathbb{R}}

\newcommand{\F}{\mathbb{F}}
\renewcommand{\c}{\mathfrak{c}}
\renewcommand{\d}{\mathfrak{d}}
\renewcommand{\b}{\mathfrak{b}}

\newcommand{\D}{\mathbf{D}}

\newcommand{\sgn}{\mathrm{sgn}}
\newcommand{\cS}{\mathcal{S}}
\newcommand{\St}{\mathrm{St}}
\newcommand{\OS}{\mathrm{OS}}
\newcommand{\Chains}{\mathrm{Chains}}

\newcommand{\G}{\mathbb{G}}
\renewcommand{\P}{\mathbb{P}}

\newcommand{\SL}{\mathrm{SL}}
\newcommand{\GL}{\mathrm{GL}}
\numberwithin{equation}{section}
\title{Symbols for toric Eisenstein cocycles and arithmetic applications}
\author{Peter Xu}

\makeatletter
\let\runauthor\@author
\let\runtitle\@title
\makeatother
\usepackage{setspace}
\begin{document}

        \makeatletter
\@setabstract
\makeatother
\maketitle

\begin{abstract}
    Using a complex parameterizing rational spherical chains, we construct explicit cocycles for $\mathrm{GL}_n(\Q)$ valued in the motivic cohomology of (open subsets of) the algebraic $n$-torus $\mathbb{G}_m^n$. The resulting cocycles directly generalize the work of Sharifi and Venkatesh from the case $n=2$ \cite{SV}. Even in this special case, our systematic use of pushforwards allows us to avoid the use of their ``connecting sequences,'' and allows us to refine the construction and Hecke properties of the Sharifi map $\varpi$ to the maximal expected statements, while inverting only the prime $2$. For general $n$, the $d\log$ regulator of our cocycle is related by convex conical duality to cocycles constructed from Shintani cones. This affords a systematic approach to $p$-adic $L$-functions for totally real fields without need for auxiliary data or logarithm sheaf coefficients, including a distribution-valued $\GL_n(\Z)$-cocycle specializing in a simple way to all such $p$-adic $L$-functions. It moreover provides a direct conceptual link between polylogarithmic constructions of Eisenstein classes (e.g., in \cite{BKL}), and those constructed using Shintani cones (e.g., in \cite{CDG}). We also show how our formalism gives an alternate proof of the exceptional divisibilities of the Deligne-Ribet $2$-adic $L$-function in almost all cases.
\end{abstract}

\tableofcontents

    \section{Introduction}

    The degree-$n$ Milnor $K$-theory of a field $F$, denoted $K_n^M(F)$, is a quotient of $F^\times \otimes \ldots \otimes F^\times$ by certain (``Steinberg'') relations, with a general element written $\{f_1,\ldots, f_n\}$ for $f_1,\ldots, f_n\in F^\times$. When $F=\Q(z_1,\ldots, z_n)$, which we think of as the fraction field of the $n$-fold power of the multiplicative group $\G_m^n$, the simple-looking elements given by $\GL_n$-translates of 
    \[
        \{1-z_1,\ldots, 1-z_n\},
    \]
    contain surprisingly interesting arithmetic information: for $n=1$, specializing them at $N$th-roots of unity leads to $N$-cyclotomic units, and when $n=2$, similar specializations are also important in cyclotomic Iwasawa theory, and in particular the Sharifi conjectures. For $n>2$, the specializations in $K$-theory are not as interesting, but one can formally take the coordinate-wise logarithmic derivative of these Milnor $K$-theory elements to get generating series which look like
    \[
    \frac{e^{2\pi i (t_1+\ldots + t_n)}}{(1-e^{2\pi i t_1})\ldots(1-e^{2\pi i t_n})}
    \]
    where $z_i=e^{2\pi i t_i}$, $1\le i \le n$, which are generating series of generalized Bernoulli numbers. For $n=1$, the Mellin transform of this series is used classically in the integral representation of the Riemann zeta function
    \[
    \xi(s):=\Gamma(s/2)\zeta(s) = \int_0^\infty \frac{x^{s}}{e^x-1}\,\frac{dx}{x}
    \]
    and thereby to prove the functional equation and find rational formulas for its values at negative integers; in the general case, this method was generalized by Shintani to to $L$-values of totally real fields of degree $n$ \cite{Shin}, using a formalism which viewed these series as generating series associated to rational polyhedral cones in $\R^n$ (which is often called the ``Shintani method'').
    
    In \cite{SV}, with applications to the Sharifi conjectures in mind, the authors construct $1$-cocycles for $\mathrm{GL}_2(\mathbb{Z})$ and its congruence subgroups, taking values in elements like $\{1-z_1,1-z_2\}$ in the second Milnor $K$-group $K_2^M(k(\mathbb{G}_m^2))$. In this article, we generalize their approach by constructing analogous cocycles for (subgroups of) $\mathrm{GL}_n(\mathbb{Q})$ valued in $K_n^M(k(\mathbb{G}_m^n))$, or various refinements/specializations thereof. When $n=2$, our formalism enables us to to improve on their results and obtain new integrality and Hecke equivariance properties for one of the maps in the Sharifi conjectures, while at the same time being more geometrically-flavored and requiring fewer technical computations.
    
    For $n>2$, the $d\log$ regulator of our cocycle yields a new construction of an ``Eisenstein cocycle'' containing a large amount of interesting arithmetic information, closely related to many previous cocycles appearing in the literature. In particular, our cocycles bear a ``conical dual'' relation to cocycles constructed based on Shintani's techniques, which we will exploit to give a systematic treatment of the integrality properties of $L$-values for totally real fields, including a conceptual proof of ``exceptional`` $2$-adic divisibilities previously only known by explicit automorphic computations.

    \subsection{Relation to existing work}
    
    The primary inspiration for the methods of this article was the motivic $\GL_2(\Z)$-Eisenstein cocycle defined by \cite{SV}: our main symbol construction can in large part be seen as an extension of the method of \cite[\S5]{SV} from $n=2$ to general $n$, though they use toric geometry in place of Bloch cycles to prove relations, and do not work with pushforwards. In the arithmetically interesting case $n=2$, we are able to refine their result by computing that the Hecke action is the expected one on a certain modular symbol, rather than just a cohomology class. In particular, we use a new method to obtain a Hecke-equivariant map on homology of the modular curve relative to cusps, rather than on the closed modular curve. Using our tools, we are also able to directly compute the Hecke actions at primes dividing the level while only inverting the prime $2$, improving on both Sharifi-Venkatesh's results as well as the work of Lecouturier-Wang \cite{LcW} on operators dividing the level.
    
    Also closely related to the present work is the approach of Shintani \cite{Shin} using rational generating functions associated to rational polyhedral cones, originally formulated in a non-cohomological way to compute $L$-values of totally real fields. Eisenstein cocycles built on this approach have previously appeared in the literature, e.g. in \cite{SH}, \cite{CD} and \cite{CDG}, \cite{Hill}, and \cite{LP}, but our approach is more systematic: besides being equally applicable to the motivic setting, we are able to quotient by very few relations in the coefficients, and do not need auxiliary data like lexicographic orderings (as in Richard Hill's approach \cite{Hill}) or perturbations (as in the cocycle of Charollois-Dasgupta-Greenberg \cite{CDG}, after an idea of Colmez); this extra data implicitly appears only in the ``degenerate'' values of a choice of cocycle representative. 
    
    We achieve this by working with a symbol complex of \emph{spherical chains}, whose realizations are related to \emph{pushforwards} of generating functions of the \emph{dual} cones to those considered by Shintani. In fact, our symbol complex construction is more directly tied to ``polylogarithmic'' constructions of Eisenstein classes, as in the work of Beilinson-Kings-Levin \cite{BKL} as well as Bergeron-Charollois-Garcia \cite{BCG} \cite{BCGV}, a term we use to refer to Eisenstein classes constructed by specifying residues along torsion cycles in some algebraic or topological group. We make this relationship precise in Appendix \ref{appendix:a}. The terminology ``polylogarithmic'' reflects that the more general versions of these cocycles (as in \cite{BKL}) use coefficients in a large ``logarithm'' sheaf in order to obtain a direct relation to higher-weight specializations of Eisenstein series, which is necessary to prove the higher-weight interpolation for the corresponding constructions of $p$-adic $L$-functions. (We also note that \cite{BHYY} relatedly constructs a equivariant polylogarithm class in \emph{Deligne} cohomology of the torus corresponding to a totally real field, which as noted in Remark \ref{rem:eqpoly}, should be the restriction of the Deligne regulator of our motivic class $\Theta(n)$ upon projection to constant coefficients.)
    
    These polylogarithmic-type constructions generally are closely linked to actual analytic Eisenstein series representatives, while the connection between Shintani-style constructions and actual Eisenstein series is largely more indirect. For example, Cassou-Noguès' \cite{CN} approach to $p$-adic $L$-functions using Shintani's method produces the same objects as the approach of Deligne-Ribet \cite{DR} using Hilbert-Eisenstein families, since they interpolate the same $L$-values, but the conceptual link between the two approaches is not clarified by this. By producing a cocycle which is both polylogarithmic (from Appendix \ref{appendix:a}) and related by conical duality to Shintani generating series, we bridge this gap. The relation to Shintani's method also allows us to prove the interpolation property of our $p$-adic $L$-functions at higher weights using only trivial coefficients, rather than logarithm sheaf.

    We also bridge a gap between the $p$-adic congruences obtained by \cite{DR} and \cite{CN}: the former obtains extra $2$-adic divisibilities coming from a computation of the constant term of Hilbert-Eisenstein series. The issue of whether these could be ``seen'' by a Shintani cone approach was already raised in \cite{DR}; this was also considered by Gross \cite{Gro}, who put it into an equivariant algebraic framework, and very recently by Colmez \cite{Col}, who obtained partial progress (see Théorème 3.22, Remarque 3.23 in loc. cit). Our algebraic setup leads to a proof of the full $2$-adic divisibility, except in the particularly delicate case of unramified totally odd characters at weight $s=0$; see Theorem \ref{thm:2adic}.
    
    Finally, to review the relation of our cocycle with a few others not yet mentioned: in the motivic setting, Lim and Park \cite{LP} based off an idea of Stevens \cite{Stevens2}, also construct a cocycle valued in Milnor $K$-theory of a ring of ``trigonometric functions''. Their construction is a complex-analytic and ``infinite level'', with their ring of trigonometric functions corresponding to the functions on a pro-tower of algebraic tori over all finite isogenies; however, like other previous Shintani-style approaches \cite{CDG} and \cite{Hill}, their construction is ``conical dual'' to ours, and involves addditional quotients by relations. The duality relation they exhibit in \cite[\S3]{LP}, between the regulator of their motivic cocycle and a naive version of the Shintani cocycle of \cite{Hill}, inspired our understanding of the conical duality linking our construction to previous ones in the literature. 

    \subsection{Structure of the paper}

        We briefly the structure of the paper: the preliminaries consist of a review of needed facts on motivic cohomology in Section \ref{section:motivic}, followed by the construction of the symbol complex and resulting cocycles in Sections \ref{section:symbols} and \ref{section:comb}. Section \ref{section:motivictheta}, on the motivic cocycle and applications, and Section \ref{section:derhamtheta}, on the de Rham cocycle and applications, are mostly independent; readers interested in one or the other may safely skip around and refer back to the preliminary sections as needed. The only major exception is that existence of the de Rham cocycle is proven using Theorem \ref{thm:gerssym} in the motivic section, since we construct the de Rham cocycle as the regulator of the motivic one.

    \subsection{Summary of methods and results}
    We now give more precise formulations of our results and arguments, to the extent possible without introducing excessive technical background. 
    
     Our main initial construction is:
    \begin{thm} \label{thm:thm1}
    There exists a Milnor $K$-theory-valued cohomology class
    \[
    \Theta(n) \in H^{n-1}(\GL_n(\Q), K_n^M(\Q(z_1^{\pm},\ldots, z_n^{\pm}))/\{-z_1,\ldots, -z_n\})
    \]
    with a homogeneous cocycle representative given by
    \[
        (\gamma_1,\ldots, \gamma_n) \mapsto \begin{pmatrix}\ell_1 & \ldots & \ell_n\end{pmatrix}_* \{1-z_1,\ldots, 1-z_n\}
    \]
    for any $\gamma_1,\ldots, \gamma_n$ whose respective first columns $(\ell_1,\ldots, \ell_n)$ are linearly independent.
    \end{thm}

    This theorem was previously proven for $n=2$ in \cite{SV}. Our method, which is a variant of the approach in \cite[\S5]{SV}, is to construct an explicit $\GL_n(\Q)$-module of \emph{symbols} $C(n)$ parameterizing elements in $K_n^M(\Q(z_1^{\pm},\ldots, z_n^{\pm}))/\{-z_1,\ldots, -z_n\}$. We refer to these as symbols, as they are in an explicit combinatorial way and are easy to write down.

    The module $C(n)$ is constructed as the top-degree chains of a chain complex computing the homology of the $(n-1)$-sphere, modulo the fundamental class, which we refer to as the symbol complex. The $\GL_n(\Q)$-action on the resulting length-$(n+1)$ exact complex yields explicit $(n-1)$-cocycles via a lifting process, all representing the cohomology class. Readers interested only in arithmetic applications may safely skip the details, referring only when needed to the characterization of the symbol complex $\Chains(n)$ in Proposition \ref{prop:simplesub}. 
    
    Most of the work of constructing $\Theta(n)$ is in proving that the relations are satisfied in Milnor $K$-theory, which we prove by constructing explicit corresponding boundaries in a cohomology complex (the \emph{Bloch cycle complex}) computing Milnor $K$-theory. The details of this realization are in section \ref{section:realization}. Notably, our systematic use of \emph{pushforwards} in the realization maps allows us both to extend our cocycle to $\GL_n(\Q)$ (rather than just $\GL_n(\Z)$), and to avoid the technicalities of the ``connecting sequences'' used in \cite{SV}, as we can work directly with non-unimodular symbols instead of needing to decompose them into unimodular ones. This also enables our simpler treatment of Hecke actions (see below).
    
    As in \cite[\S5]{SV}, we also consider variants of the construction in this section, there are various methods by which one can remove the necessity of quotienting by $\{-z_1,\ldots, -z_n\}$, (though usually one gives up something else; e.g. inverting primes). Conversely, by quotienting out by \emph{more} relations, one can obtain an \emph{parabolic} version of $\Theta(n)$. Our use of symbol complexes makes it possible to understand systematically what modifications are required for each of these variants.

    \subsubsection{Application to Sharifi maps}
    In particular, a modular symbol variant of the cocycle has the following application: in the case $n=2$, Sharifi and Venkatesh use their parabolic cocycle to construct and prove the Eisenstein property of one of the maps which figures in the eponymous Sharifi conjectures, relating the homology of modular curves and algebraic $K$-theory of rings of integers in cyclotomic fields. Their methods produce the desired map in the form
    \[
    \Pi_N: H_1(X_1(N),\Z) \to K_2(\Z[1/N,\zeta_N]) \otimes \Z[1/2]
    \]
    which is Eisenstein for the anemic Hecke operators only, i.e. factors through the quotient by the operators $T_\ell - 1-\ell\langle \ell\rangle$ for $\ell \nmid N$. The Sharifi conjectures concern an explicitly constructed map 
    \[
    \Pi_N^\circ: H_1(X_1(N),C_1^\circ(N),\Z) \to  K_2(\Z[1/N,\zeta_N]) \otimes \Z[1/2], \;\;\;[u:v]\mapsto \{1-\zeta_N^u, 1-\zeta_N^v\}
    \]
    on the larger cohomology group relative to the set of cusps $C_1^\circ(N)$ not in the $\Gamma_0(N)$-orbit of $\infty$, and work of Fukaya-Kato \cite{FK} showed that for $p|N$, the complex conjugation-fixed (``plus'') version of this map, denoted by subscript $+$,
    \[
    (\Pi_N^\circ)_+\otimes \Z_p: H_1(X_1(N),C_1^\circ(N),\Z_p)_+ \to  (K_2(\Z[1/N,\zeta_N]) \otimes \Z_p)_+
    \]
    given by the formula
    \[
    [u:v]\mapsto \{1-\zeta_N^u, 1-\zeta_N^v\}_+ := \frac{1}{2}(\{1-\zeta_N^u, 1-\zeta_N^v\}+ \{1-\zeta_N^{-u}, 1-\zeta_N^{-v}\})
    \]
    factored through not only the anemic Eisenstein operators $T_\ell - 1-\ell\langle \ell\rangle$, but also $U_p - 1$ whenever $p | N$. Thus, Sharifi-Venkatesh's work improved on this by constructing an anemically Eisenstein map with $\Z[1/2]$-coefficients, and not only on the plus part. However, they were unable to compute the action of the Hecke operators dividing the level, or to define it on the larger relative cohomology group as in $\Pi_N^\circ$. Later, \cite{LcW} combined the methods of Sharifi-Venkatesh and Fukaya-Kato to show that the map $\Pi_N\otimes \Z[1/6]$ factors through the non-anemic Hecke operators $U_\ell - 1$ for $\ell|N$, again only on the plus part.

    After ``spreading out'' the pullback of the modular symbol variant of our cocycle $\Theta^{MS}(2)$ in order to pull it back by an $N$th root of unity, we are able to obtain the improved result:

    \begin{thm} \label{thm:sharifi}
        The map
        \[
        (\Pi_N^\circ)_+: H_1(X_1(N),C_1^\circ(N),\Z) \to K_2(\Z[1/N,\zeta_N]) \otimes \Z[1/2]
        \]
        factors through both $T_p^* - p-\langle p\rangle$ for $p \nmid N$ and $U_p-1$ for $p|N$. Further, the diamond operator $\langle d \rangle$ acts as pullback by the Galois automorphism $[d]:\zeta_N\mapsto \zeta_N^d$ for all $d\in (\Z/N)^\times$. 

        Additionally, if $N$ is a power of a prime, or if we further restrict the set of cusps to exclude $\Gamma_0(p)$-orbit of $\infty$ for \emph{all} $p|N$, then all of this holds even for the map $\Pi_N^\circ$, \emph{without} projecting to the $+$ part. 
    \end{thm}

    The point is that these relative homology groups are naturally $\Gamma_1(N)$-subquotients of the Steinberg module $\St(2)$, consisting of mass-zero finite functions on $\P^1(\Q)$; this, in turn, is a $\GL_2(\Q)$-equivariant quotient of our spherical chains module $C(n)$. Our realization map allows us to understand precisely what elements in $K$-theory we need to quotient by to have these extra Steinberg relations, resulting in a $\GL_n(\Q)$-equivariant modular symbol 
    \[
    \Theta^{MS}(2):\St(2)\to K_2(k(\G_m^2))/\text{extra relations}
    \]
    and which pulls back, upon restriction to $\Gamma_1(N)$-level, to a modular symbol
    \[
    \St(2)_{\Gamma_1(N)} \to K_2(\Z[1/N,\zeta_N])/\text{extra relations} \otimes \Z[\tfrac{1}{2}]
    \]
    where these extra relations are trivial with $\Z[\frac{1}{2}]$-coefficients upon projection to the plus part (or are already zero, in the special case noted in the theorem).
    
    Moreover, the Hecke action can be computed directly on the modular symbol $\Theta^{MS}(2)$ (without ever passing to a cohomology class), in terms of a simple geometric operation on the torsion sections used to pullback. This enables us to compute the Hecke action at all primes; we find the Eisenstein property holds if one restricts the set of cusps as indicated in the theorem. We further give a geometric interpretation and proof of the $N$-integrality of a restriction of $\Pi_N^\circ$ to the homology of the closed curve $X_1(N)$, which previously had been proven computationally \cite[Lemma 4.2.7]{SV} \cite[Lemma 3.3.11]{FK}. 

    \subsubsection{Application to $L$-values}
    We now turn to the \emph{regulator} of $\Theta(n)$, i.e. its composition with the map
    \[
    \{f_1,\ldots, f_n\} \mapsto d\log f_1 \wedge \ldots \wedge d\log f_n \in \Omega^n_{k(\G_m^n)}.
    \]
    After contracting with the $\SL_n(\Q)$-invariant vector $\partial z_1 \wedge \ldots \wedge \partial z_n$, the resulting de Rham cocycle $\Theta^{dR}(n)$ is valued in precisely the same kind of generating functions associated by Shintani \cite{Shin} to arrangements of cones, closely related to $L$-functions of totally real fields. However, the generating function associated to the class of a rational simplex $\Delta\subset S^{n-1}$ on the $(n-1)$-sphere is \emph{not} the generating function of the associated (open) cone, but to its \emph{dual} cone: that is, formally, to the sum
    \[
    (t_1,\ldots, t_n)\mapsto \sum_{\substack{x\in \Z^n\\ \langle x,\R_{\ge 0}\Delta\rangle>0}} e^{2\pi i \langle x, t\rangle }
    \]
    rather than 
    \[
    (t_1,\ldots, t_n)\mapsto \sum_{x\in \Z^n\cap \R_{>0}\Delta} e^{2\pi i \langle x, t\rangle},
    \]
    where $\langle-,-\rangle$ is the usual inner product, and we identify a region of $S^{n-1}$ with the space of rays in its $\R^+$-span. In \cite{CDG}, for example, the Shintani generating function associated to the anisotropic torus associated to a totally real field $F$ of degree $n$ is of the latter form (and also involves delicate choices of lower-dimensional cones), while the specialization of our $\Theta^{dR}(n)|_{\mathcal{O}_F^\times}$ is of the former. However, we are able to show that $L$-value specializations of the latter type of series are the same as those coming from ours $\Theta^{dR}|_{\mathcal{O}_F^\times}$, if one takes the \emph{inverse transpose} of the embedding $\mathcal{O}_F^\times\hookrightarrow \GL_n(\Z)$: this reflects the underlying conical duality, and involves a technical argument using stabilized versions of our symbol complex at an auxiliary prime, as well as a geometric incarnation of conical duality in the Tits building of $\GL_n$ (see Appendix \ref{appendix:b}). 
    
    As a result of the smoothing, we are able to pull back $\Theta^{dR}(n)$ along \emph{all} torsion sections, and hence obtain a single class valued in distributions which we can show specializes to all different $p$-adic $L$-functions associated to totally real fields of degree $n$: 

    \begin{thm}[Theorem \ref{thm:interp}] \label{thm:distros}
        For each integer $c>1$ prime to a prime $p$, there exists a cohomology class valued in a certain space of $p$-adic distributions (see Section \ref{section:fourier} and the formula \eqref{eq:getzeta} for notational details)
        \[
        {}_c\Theta^{dR}_{\D_p}(n)\in H^{n-1}(\GL_n(\Z),\hat{\D}_0(\Q_p^n, \Z_p)(\sgn)^{(0)}),
        \]
        such that if $F$ is a totally real field with $[F:\Q]=n$ and $\alpha:I\xrightarrow{\sim} (\Z^n)^\vee$ is a choice of dual basis for a fractional ideal $I\subset F$, then the restriction of
        ${}_c\Theta^{dR}_{\D_p}(n)$ to the resulting embedded copy of $U^+:=(\mathcal{O}_F^\times)^+\hookrightarrow \GL_n(\Z)$ satisfies
        \[
            \int_{(\Q_p^n)^\vee} \psi(\alpha^{-1}(t))\mathbf{N}^F_\Q(\alpha^{-1}(t))^k\, d({}_c\Theta^{dR}_{\D_p}(n)|_{U^+} \frown c_{U^+}) = \zeta_{I}(\psi, -k)-c^{n}\zeta_{(c)I}(\psi, -k) 
        \]
        for any function $\psi:I^+/U^+\to \overline{\Q_p}$ with $\psi(0)=0$, locally constant in the $p$-adic topology.\footnote{For reasons of brevity, we do not consider adding tame level - i.e., prime-to-$p$ level - to our $p$-adic $L$-functions; however, the same constructions all work in this generality as well.} Here $c_{U^+}$ is a fundamental class of $U^+\cong \Z^{n-1}$, positively oriented in an appropriate sense, and 
        \[
        \zeta_{I}(\psi, -k) := \left(\sum_{\lambda \in I^+/U^+} \psi(\lambda)\mathbf{N}\lambda^{-s}\right)_{s=-k}
        \]
        is a partial zeta value for $I$.
    \end{thm}
    The technical condition $\psi(0)=0$ means that unlike \cite{DR} or \cite{CN}, but like Kubota-Leopoldt's original construction of their $p$-adic zeta function for $\Q$, our method does not interpolate $L$-values of narrow Hilbert class characters at $s=0$ only; by the Stark conjectures, the information in these values is essentially the size of the $\psi$-class group for totally odd narrow Hilbert class characters $\psi$ (as for non-totally odd characters, the $L$-value vanishes). This appears to be an inherent limitation of our method of construction, though we would be interested to find otherwise. 
    
    \begin{rem}\label{rem:totallyodd}
        When $n$ is odd, there are no totally odd characters at all, and so all $L$-values vanish at $s=0$; hence, in this case we are not ``losing any information''. Interestingly, in this case, there is also a very natural way to canonically lift over the orientation obstruction in the construction of $\Theta^{S^{n-1}}(n)$ by symmetrizing it to be $[-1]$-invariant, since $[-1]$ is orientation-reversing. 
    \end{rem}

    The fact that we use Shintani's method (in Section \ref{section:cones} onward) to prove the $p$-adic interpolation is of note, when juxtaposed with our ``polylogarithmic''-style construction (see Appendix \ref{appendix:a}) of our classes. Previously, \cite{BKL} constructed $p$-adic $L$-functions for totally real fields using Eisenstein classes constructed polylogarithmically on topological tori $(S^1)^n$, but needed coefficients in the logarithm sheaf to prove the interpolation property at higher weights; the direct relationship with Shintani's method allows us to avoid this.

    Finally, we remark that we can also obtain $S$-arithmetic or even $\GL_n(\Q)$ versions of this class, and representing cocycles (maybe after inverting primes in the coefficients). We do not go into $S$-arithmetic applications in this article, so this is only covered in passing; however, for example, we expect it to give an alternative approach to the rigid analytic cocycles of \cite{RX2}.

    The construction of $p$-adic $L$-functions by Deligne-Ribet \cite{DR} also proved extra divisibilities at the prime $2$ for certain $L$-values ``with parity''; it remained an open question whether approaches using Shintani cones (like \cite{CN}, contemporaneously) could ``see'' these extra congruences; for example, this was considered by Gross \cite[Proposition 5.4]{Gro} (who formulated it in algebraic terms quite related to our distributions) and Colmez \cite[Théorème 3.22]{Col} (who proved a slightly weaker statement). We prove the full expected congruence for the values we are able to interpolate:
    
    \begin{thm}[Theorem \ref{thm:2adicbody}, analogue of (8.11--12) in \cite{DR}] \label{thm:2adic}
        Let $\psi:G_{p^\infty}\to \overline{\Q_p}$ be a totally odd continuous function on the $p^\infty$-ray class group of $F$. Then, for our $2$-adic zeta element ${}_\c\zeta_2^F$ constructed in Section \ref{section:padic}, we have the congruence
        \[
        \int_{G_{p^\infty}} \psi(J)\, d{}_\c\zeta_2^F(J) \equiv 0 \pmod{2^n}.
        \]
    \end{thm}
    
    This theorem follows in a very formal way from our setup; as such, by comparison with the method of \cite{DR}, one could view it as a ``cohomological explanation'' for the appearance of $2^n$ in the constant term of Hilbert-Eisenstein series. As in \cite{DR}, we in fact prove the theorem for more general functions on a monoid of ideals strictly larger than the $p^\infty$-ray class group. Our  ``vanishing at zero'' restriction on interpolated values means that the only cases we miss are when $\psi$ is a linear combination of totally odd narrow Hilbert class character (which is not a case included in the simplified statement of the theorem above.)
    
    \subsection{Future work}

    The basic methods (parameterization by linear algebraic complexes of symbols) of this article apply, though with some substantial changes, to the setting of elliptic polylogarithms: for example, the action of $\mathrm{GL}_n/\mathbb{Q}$ on the $n$th power of the universal elliptic curve, or $\mathrm{GL}_n/K$ on the $n$th power of an elliptic curve with CM by $K$. This elliptic case will be the subject of a sequel to this article. Working analytically over $\mathbb{C}$, certain specializations of the former setting yield (after pullback by torsion sections) the Siegel unit-valued modular symbols constructed in the recent work of \cite{BPPS}. This elliptic case has analogous applications to the arithmetic of totally real fields, as well as to the Sharifi conjectures for the pairs ($\GL_3(\Q),\GL_2(\Q)$) and ($\GL_2(K),\GL_1(K)$) in the case of a CM field $K$, just as the present article deals with the case ($\GL_2(\Q),\GL_1(\Q)$). In the setting of Drinfeld modules for a function field $F/\F_q$, we also use symbol complex methods to attack the case $(\GL_n(F), \GL_{n-1}(F))$, in forthcoming joint work with the the first author of \cite{SV}.

    Another question posed by \cite{Katz} is the relation between Shintani cones and toroidal compactifications of Hilbert-Blumenthal varieties; via Hirzebruch's conjecture (now a theorem), this also is closely related to the same $L$-values studied by Deligne-Ribet (as well as the present article). We would be interested in a group theoretic/geometric answer to this question, which seems related to our cohomological methods.

    Finally, we do not go into the relation of our cocycles to Gross-Stark units in this article, but it may be of interest to write down the relationship carefully. As noted in Remark \ref{rem:grossstark}, we believe that our ``de Rham cocycle'' results in the same (smoothed) distributions for totally real fields as one obtains from \cite{CDG}, for example, and thus should yield the same now-proven formulas for Gross-Stark units of \cite{DS} (thanks to the work of \cite{DKSW}). These other existing cocycles are already perfectly sufficient for explicit class field theory, but the very natural algebraic formulation of integrality properties in our approach could possibly offer benefits of exposition, at least. 
    
    We \emph{do} write down the relationship with Gross-Stark units of a closely-related cocycle (from the polylogarithm of a \emph{topological} torus) in the joint work \cite{RX}, as well as its relation to the rigid analytic cocycles of \cite{DV}; in joint work in preparation, we will relate that polylogarithmic cocycle to the one in this article.

    \subsection{Acknowledgements}

    Many thanks to Nicolas Bergeron, Marti Roset Julia, Romyar Sharifi, Timothy Smits, Jeehoon Park, and Pierre Charollois for open ears, discussions, and suggestions.

    \section{Preliminaries} \label{section:prelim}

    \subsection{Motivic cohomology}\label{section:motivic}
	In this section, we recall some needed facts about motivic cohomology and Milnor $K$-theory. For a smooth equidimensional scheme $X$ over a discrete valuation ring (DVR) $R$, we will compute (or, for the purposes of this article, \emph{define}) the motivic cohomology $H^i(X, \mathbb{Z}(n))$ as the cohomology $H^i(X,\mathbb{Z}(n)_X)$ of Bloch's weight-$n$ \emph{cubical} cycle complex $\mathbb{Z}(n)_X$; this is the approach followed in \cite{Tot} and \cite{GL}, for example, over a field.

    The Bloch cycle complex is defined as follows: let
	\[
	\tilde{z}^n(X, i) := Z^n(U\times \Box^{i})
	\]
	be the group of codimension-$n$ cycles on $X\times \Box^{i}$ meeting all faces properly. Here, $\Box^i$ is the algebraic $i$-cube which we identify with $(\mathbb{P}^1-\{1\})^n$ and the $j$th face map is given by the difference of the pullbacks to the subvarieties cut out by $t_j=0$, respectively $t_j=\infty$; the alternating sum of face maps gives, as usual, a differential from $\tilde{z}^n(X, i)$ to $\tilde{z}^n(X, i-1)$.\footnote{In some conventions, $\Box^i$ is identified with $\mathbb{A}^i$, and the face maps are given by $t_j=0,1$ instead of $0,\infty$. This differs from our convention by a Möbius transformation; ours is the convention used by \cite{Tot}, and we we find it to be more convenient for later applications.} It turns out that the resulting complex splits into a direct sum
	\[
	\tilde{z}^n(X, i) = d^n(U,i)\oplus z^n(X, i)
	\]
	where the former summand consists of \emph{degenerate} cycles which can be pulled back from one of the faces of $\Box^i$ given by $t_j=0$, and the latter summand consists of the \emph{reduced cycles} which are in the kernel of of the restriction to each face $t_j=0$. We define 
	\[
	(\mathbb{Z}(n)_X)^i:=z^n(X, 2n-i).
	\]
	This complex is suitably functorial for flat pullbacks and proper pushforwards.
    
    More generally, over a Dedekind scheme base, we can define motivic cohomology as the \emph{hypercohomology} of the Zariski sheafification of the Bloch complex: the fact that this coincides with our above (non-hypercohomology) definition for the Bloch cycle complex was proven for fields as \cite[Corollary 12.2]{FS}, and for DVRs in \cite{Lev2}; consult these sources for details. The only fact we need about this broader definition is the existence in this generality of a localization sequence:

    \begin{prop}[\cite{Geisser}] \label{prop:loc}
        If $Z\subset X$ is a pair of smooth schemes embedded with pure codimension $d$, over a Dedekind domain, then there is a localization sequence
        \[
        \ldots \to H^{i-2d}(Z,\Z(n-d))\to H^i(X,\Z(n)) \to H^i(X-Z, \Z(n)) \to H^{i+1-2d}(Z, \Z(n-d)) \to \ldots 
        \]
    \end{prop}

    This will be used to extend results about pullbacks of our motivic cocycles to bases larger than DVRs.
    
    \begin{rem}
        This article is not concerned with the minutiae of motivic constructions, so we only remark briefly on the technical details we are glossing over in the above definition, in order to provide justifying references: first, this construction should more properly be called Borel-Moore motivic homology; historically, Bloch called them ``higher Chow groups''. For smooth schemes over a field, higher Chow groups agree with the usual modern construction of motivic cohomology defined via Voevodsky-style motivic complexes, thanks to the results in \cite{V} and \cite{FS}. Second, simplicial language is more standard than cubical (as in, e.g., \cite{Geisser}), but the two approaches are equivalent for formal reasons, as proven in \cite{GL} over a field (though the proof works also over a DVR). 
    \end{rem}
    
    When $X=\mathrm{Spec }\,F$, then \cite{Tot} proves that there is a natural isomorphism
    \[
    \psi_{X}^n: H^n(X,\Z(n)) \xrightarrow{\sim} K_n^M(F)
    \]
    between the degree-$n$, weight-$n$ motivic cohomology and \emph{Milnor $K$-theory}, which we recall is given in degree $n$ by 
    \[
    K_n^M(F) := (F^\times \otimes \ldots \otimes F^\times )/ I
    \]
    where $I$ is the degree-$n$ part of the ideal in the free tensor algebra on $F^\times$ generated by $x\otimes (1-x)$ for non-identity $x\in F^\times$. This map is given on the level of the Bloch cycle complex as follows: for the class of an irreducible closed subvariety $[Z]\in z^n(X, n)$, write $p:Z\to X$ for the projection map; then we set 
		\begin{equation}\label{eq:psimap}
			\psi_X^n(z):= p_*([t_1\otimes \ldots \otimes t_{n}]) \in K_{n}(F)
		\end{equation}
	where $p_*$ is the finite pushforward (also called \emph{transfer map}) in Milnor $K$-theory (as defined in \cite[\S5]{BaTa}) for the map $p$, where we recall the $t_\bullet$ are the coordinate functions on the simplicial cube. 

    When $n=1$ above, the above discussion applies to more than fields: for any scheme $X$, the map $H^1(X,\Z(1)) \to \Gamma_X(\mathcal{O}_X^\times)$
    given by sending $[Z]\mapsto p_*(t_1)$ (and zero if $Z$ is not dominant over $X$) is an isomorphism identifying the degree-$1$ weight-$1$ motivic cohomology with the global units. This allows us to more generally consider cup products like 
    \[
        u_1 \smile \ldots \smile u_n \in H^n(X, \Z(n))
    \]
    for units $u_1,\ldots, u_n$ on $X$. Following notational conventions in Milnor $K$-theory, we will also denote such cup products, as well as the associated tensors in Milnor $K$-theory, by the curly braces $\{u_1,\ldots, u_n\}$.

    \subsubsection{Trace operators on cohomology of $\mathbb{G}_m$}

    We observe the element $1-z\in H^1(\mathbb{G}_m-\{1\},\mathbb{Z}(1))$ satisfies 
    \[
    [a]_*(1-z) = 1-z
    \]
    for any $a\in \mathbb{N}$, where $[a]_*$ denotes the restriction of the map $[a]:\mathbb{G}_m\to \mathbb{G}_m$ is the map induced by $z\mapsto z^a$, composed with pullback by the inclusion $\mathbb{G}_m-\{1\} \hookrightarrow \mathbb{G}_m$. In general, we use the same notation $[a]:\mathbb{G}_m^n\to \mathbb{G}_m^n$ for the same map coordinatewise $(z_1,\ldots, z_n)\mapsto (z_1^a,\ldots, z_n^a)$, and observe that also 
    \[
    [a]_*\{1-z_1,\ldots, 1-z_n\} = \{1-z_1,\ldots, 1-z_n\}.
    \]
    For the submodule of the cohomology $\mathbb{G}_m^n$, or various subspaces thereof, we will write superscript $(0)$ to mean the elements which are fixed in this way by all $[a]_*$; e.g.
    \[
    \{1-z_1,\ldots, 1-z_n\} \in H^n((\mathbb{G}_m-\{1\})^n,\mathbb{Z}(n))^{(0)}.
    \]
    Note that $H^n((\mathbb{G}_m-\{1\})^n,\mathbb{Z}(n))$ carries a pushforward action of $M_n(\mathbb{Z})$ commuting with all $[a]_*$; since the $[a]_*$ act invertibly, the $M_n(\mathbb{Z})$ actually induces a $\mathrm{GL}_n(\mathbb{Q})$-action on the trace-fixed submodule $H^n((\mathbb{G}_m-\{1\})^n,\mathbb{Z}(n))^{(0)}$. This applies to the trace-fixed cohomology of any subspace of $\mathbb{G}_m$ with a $M_n(\mathbb{Z})$-action.

    \subsection{Symbols from spherical chain complexes} \label{section:symbols}
    We turn now to constructing a chain complex $\widetilde{\Chains}(n)$ with action by $\mathrm{GL}_n(\mathbb{Q})$, and show that it computes the reduced homology of the sphere $S^{n-1}$ and is thus almost exact. In top degree, this complex will naturally parameterize special elements in the Milnor $K$-theory of the function field of $\mathbb{G}_m^n$, and we will abbreviate the top-degree module
    \[
    \tilde{C}(n) := \widetilde{\Chains}(n)_n
    \]
    The construction of this section directly generalizes \cite[\S5]{SV}, who consider the case $n=2$, though we use slightly different conventions for the group action: we are taking \emph{left} actions on spaces with a resulting pushforward left action on cohomology groups, while they consider right actions on spaces.
    
    We identify $S^{n-1} = (\R^{n}-\{0\})/\R_{>0}^\times$ with the set of \emph{rays} in an $n$-dimensional real vector space, and write $S^{n-1}(\Q)$ for the subset of rational rays. Consider the poset $\mathcal{P}$ of linear triangulations $T$ of $S^{n-1}$ with vertices contained in $S^{n-1}(\mathbb{Q})\subset S^{n-1}$, ordered by \emph{refinement}: that is to say, we say $T_1\ge T$ if every simplex of $T$ is a union of simplices of $T_1$. Here, by \emph{linear triangulation}, we mean every $k$-simplex is the intersection of a $S^{n-1}$ with a dimension-$(k+1)$ simplicial polyhedral cone in $\mathbb{R}^n$ emanating from the origin. Notice that in particular, the great circle containing any given simplex can be recovered as the intersection of $S^{n-1}$ with the span of the rays corresponding to the vertices of the simplex. As such, each of these corresponding subspaces are themselves rationally defined, since these rays are rational.
    
    Given any geodesic $k$-simplex $\Delta$ with endpoints in $S^{n-1}(\mathbb{Q})$ with $k\le n-1$, there exists a triangulation $T\in \mathcal{P}$ such that $\Delta$ is one of the triangles in $T$: indeed, for each face of $\Delta$, take the corresponding great circle on $S^{n-1}$; these divide the sphere into polyhedral regions, which then can be barycentrically subdivided into simplices by picking a rational point in each interior. We have the following ``upper bound'' theorem on refinements of trinagulations.

    \begin{prop} \label{prop:refine}
        Given any two triangulations $T_1,T_2\in \mathcal{P}$, there exists a common upper bound for them, i.e., a triangulation $T\in \mathcal{P}$ refining both $T_1$ and $T_2$.
    \end{prop}
    \begin{proof}
        Superimposing the triangulations $T_1$ and $T_2$, we obtain a decomposition of $S^{n-1}$ into convex polyhedral cells. Each vertex of this superimposition is still in $S^{n-1}(\mathbb{Q})$, because any new vertices are the intersection points of rationally defined subspaces in $\mathbb{R}^n$ (corresponding to the great circles forming the intersecting faces) and hence themselves correspond to rational rays. Pick a rational point in interior of each polyhedral cell and subdivide it into simplices by joining the point to each $(n-2)$-face of the cell with a $(n-1)$-simplex; we obtain a new triangulation which still has all vertices in $S^{n-1}(\mathbb{Q})$. 
    \end{proof}

    Write $C_T$ for the augmented cellular chain complex of the simplicial complex $T$ with
    \[
    C_{T,i} = \mathbb{Z}\{T_i\}
    \]
    for all $i\ge 0$, freely generated by the $i$-simplices endowed with their orientation coming from the globally oriented $S^{n-1}$. (We will consider the oppositely-oriented version $\overline{\Delta}$ of a simplex $\Delta$ to live in the chain complex, representing its additive inverse $[\overline{\Delta}] = -[\Delta].$) The boundaries $C_{T,i+1}\xrightarrow{\partial} C_{T,i}$ given by the alternating face maps when $i\ge 0$. By ``augmented'', we mean that we set $C_{T,-1}=\mathbb{Z}$, with its incoming boundary map given by the degree map on $0$-simplices. Then for any $T$,
    \[
    H_{n-1}(C_{T,\bullet}) \cong \mathbb{Z}
    \]
    is the only nontrivial homology group, since each simplicial complex is homotopy equivalent to $S^{n-1}$.

    For a refinement $T'$ of $T$, there is a pullback complex map
    \[
    C_T \to C_{T'}
    \]
    which is a quasi-isomorphism; these maps are functorial for compositions in the poset of triangulations. We can then define
    \begin{equation} \label{eq:ind}
    \widetilde{\Chains}(n)_i := \varinjlim C_{T,i}[1]
    \end{equation}
    as the colimit over all triangulations; by exactness of direct limits, $\widetilde{\Chains}(n)$ also only has homology $\mathbb{Z}$ in degree $n$; recall also that we abbreviate the top-degree term as $\tilde{C}(n)$.
    
    Via the left action of $\mathrm{GL}_n(\mathbb{Q})$ on $S^{n-1}(\mathbb{Q})$ viewed as parameterizing rays in $\mathbb{R}^n$, the complex $\widetilde{\Chains}(n)$ takes a left action of $\mathrm{GL}_n(\mathbb{Q})$. In particular, if $\gamma$ fixes a simplex $\Delta$ but $\det \gamma = -1$, then
    \[
    \gamma \cdot [\Delta] =[\overline{\Delta}]= -[\Delta].
    \]

    We now give generators and relations for $\widetilde{\Chains}(n)$. Given any tuple of independent rays $(m_1,m_2,\ldots, m_k)$ in $S^{n-1}(\mathbb{Q})$ lying in the same (open) hemisphere, which we will call an \emph{acyclic} tuple, they span a unique oriented geodesic simplex (with orientation corresponding to the order), whose $1$-frame is formed from the shorter segment of the great $1$-circle through each pair of points. (By assumption, there are no antipodes, so there is no ambiguity.)
    
    For any triangulation $T$ containing this simplex, we write
    \[
    [\Delta(m_1,m_2,\ldots, m_k)_T]\in C_{T,k-1}
    \]
    for the class in homology; under any triangulation $T'$ which refines $T$ with the same property, the class $[\Delta(m_1,m_2,\ldots, m_k)_T]$ maps to $[\Delta(m_1,m_2,\ldots, m_k)_{T'}]$. Thus, we can speak unambiguously of a class in the direct limit
    \[
    [\Delta(m_1,m_2,\ldots, m_k)]\in \widetilde{\Chains}(n)_k.
    \]
    These classes generate $\widetilde{\Chains}(n)_i$ for each $i$, since any rational simplex can be subdivided into rational acyclic simplices.

    We claim that the relations between the generators corresponding to these simplices are generated by subdivision of an acyclic simplex into subsimplices. Indeed, suppose we have an relation of $k$-simplex generators
    \begin{equation} \label{eq:relation}
    [\Delta_1]+\ldots + [\Delta_t]=0.
    \end{equation}
    Superimposing all the great circles corresponding to the faces of the $\Delta_i$ and subdividing the resulting polyhedral cells as in the proof of \ref{prop:refine}, we can find a set of geodesic simplices $\Delta_1',\ldots,\Delta_s'$, disjoint outside of their faces, such that each $\Delta_i$ is a union
    \[
    \Delta_i = \bigsqcup_{j\in S_i}\Delta_{j}'.
    \]
    Associated to this decomposition of simplices we have pushforward relation
    \begin{equation} \label{eq:decomprel}
    [\Delta_i] = \sum_{j\in S_i} \Delta_j'.
    \end{equation}
    Substituting these into \eqref{eq:relation}, the resulting relation is between disjoint simplices, on which the coefficient of each simplex therefore must be zero. Hence relations of the form \eqref{eq:decomprel}, coming from subdivision of an acyclic simplex, generate all relations as we claimed.

    In fact, we can go even further, and generate all acyclic simplicial subdivision relations in terms of some easily classified ones:
    \begin{prop} \label{prop:simplesub}
        Write $\mathbf{T}_k$ for the set of all acyclic positively oriented geodesic $k$-simplices with rational faces.\footnote{As before, we will also freely use negatively oriented simplices as generators, with the convention they are simply $-1$ times the positively oriented simplex.} The module $\widetilde{\Chains}(n)_i$ for $i>0$ is given by 
        \[
        \mathbb{Z}\{\mathbf{T}_{i-1}\}/\cS_{i-1}
        \]
        where $\cS_k$ is the module of acyclic $k$-simplicial subdivision relations, and is generated by \emph{stellar subdivisions} of rank $r$ for $2\le r\le k$ given as follows: fix an acyclic simplex $\Delta\in \Delta(m_1,\ldots, m_k)$ and a rational point $m$ lying in the great circle corresponding to the face with vertices $(m_1,\ldots, m_{r})$, but sharing some hemisphere of the great circle with all of them. Corresponding to this data we have the \emph{stellar} subdivision relation
        \begin{equation}\label{eq:simpgen}
        [\Delta(m_1,\ldots, m_k)]=\sum_{i=1}^r [\Delta(m_1,\ldots, \hat{m_i},m,m_{i+1},\ldots, m_k)].
        \end{equation}
    \end{prop}
    \begin{proof}
        From the discussion preceding the proposition, it remains only to prove that the stellar subdvision relations \eqref{eq:simpgen} generate all acyclic subdivision relations.

        First, we say that a subdivision of a simplex $\Delta$ is a \emph{sequentially stellar subdivision} if it can be obtained by a sequence of stellar subdivisions of $\Delta$ (and the resulting subsimplices at each sequential step). The relation corresponding to a sequentially stellar subdivision of an acyclic simplex is certainly generated by those of the implicated stellar subdivisions by repeated substitution.

        We observe, therefore, that it suffices to show that any acyclic subdivision         
        \[
        \Delta = \Delta_1\sqcup \ldots \sqcup \Delta_t
        \]
        can be refined to a sequentially stellar subdivision of $\Delta$, which is also a sequentially stellar subdivision of each $\Delta_i$ when restricted to it; this result in the field of combinatorial topology is due to \cite{New}.
    \end{proof}
    
    We also write
    \[
    \Chains(n) := \widetilde{\Chains}(n)/H_n(\widetilde{\Chains}(n))
    \]
    for the exact quotient of $\widetilde{\Chains}(n)$ given by quotienting by top homology. We use the analogous abbreviation
    \[
    C(n):= \Chains(n)_n
    \]
    for the top-degree module.
    
    \subsection{Combinatorial cohomology classes for $\mathrm{GL}_n$} \label{section:comb}

    We now construct ``abstract'' cohomology classes for $\mathrm{GL}_n(\mathbb{Q})$ valued in symbols in the symbol complex, whose realization in Milnor $K$-theory (and later, differential forms) will yield our arithmetic classes of interest. We use the following standard construction in group cohomology, whose proof we sketch:

    \begin{lem} \label{lem:groupcohom}
        If a group $G$ acts on an exact complex $C_\bullet$ supported in degrees $[0,n]$, then we have a natural map on cohomology
        \[
        C_0^G \to H^{n-1} (G, C_n)
        \]
        inhomogeneous cocycle representatives of which can be constructed as follows: associated to $e\in C_0^G$, pick a lift $\ell_1$ of $e$ to $C_1$, and consider the $1$-cochain
        \[
        \gamma\mapsto (\gamma-1)\ell_1 \in C^1(G, C_1).
        \]
        By exactness, this is the boundary of an element $\ell_2\in C^1(G,C_2)$; we take the chain coboundary $\partial \ell_2 \in C^2(G,C_2)$ which again lifts to $\ell_3\in C^2(G,C_3)$, etc. The lift $\ell_n\in C^{n-1}(G,C_n)$ is a cocycle representing the image of $e$ in $H^{n-1} (G, C_n)$.
    \end{lem}
    \begin{proof}
        The map $C_0^G \to H^{n-1} (G, C_n)$ is the composition of connecting homomorphisms
        \begin{align}
        C_0 & \to H^1(G, \ker \partial_1) = H^1(G, \mathrm{im } \partial_2)\\
            & \to H^2(G, \ker \partial_2) = H^2(G, \mathrm{im } \partial_3) \\
            & \to \ldots \\
            & \to H^{n-1}(G, \ker \partial_n)=H^{n-1}(G, C_n).
        \end{align}
        for the $G$-short exact sequences
        \[
        \ker \partial_i \xrightarrow{\partial_{i+1}} C_i \xrightarrow{\partial_i} \mathrm{im }\partial_i.
        \]
        Unwinding the definition of each of these maps yields precisely the process of iterated lifting described in the lemma.
    \end{proof}

    We apply the lemma to the action of $\mathrm{GL}_n(\mathbb{Q})$ on $\Chains(n)$, using the element $e\in \mathbb{Z}$ in degree zero of each respective complex. This affords us a cohomology class
    \begin{align}
    \Theta^{S^{n-1}}(n)\in H^{n-1}(\mathrm{GL}_n(\mathbb{Q}), C(n)).
    \end{align}
    
    \subsubsection{Explicit cocycle representatives}

    We can apply Lemma \ref{lem:groupcohom} to obtain cocycle representatives for $\Theta^{S^{n-1}}(n)$.
    
    To write down our explicit cocycle, it will be useful to specify an extension of the previous notation $[\Delta(m_1,\ldots, m_k)]$ to \emph{any} tuple of rays $(m_1,\ldots, m_k)$, even tuples failing to be acyclic or independent. Thus, we define a $\Delta$-\emph{extension} $E$ to be a collection of classes\footnote{This is somewhat abusive notation, since we are not necessarily saying that $[\Delta_E(m_1,\ldots, m_k)]$ is actually the class of a particular geodesic simplex when the tuple is not linearly independent.}
    \[
        [\Delta_E(m_1,\ldots, m_k)] \in \Chains(n)_k
    \]
    for arbitrary tuples of rays $(m_1,\ldots, m_k)$ satisfying the following properties:
    \begin{enumerate}
        \item If the tuple $(m_1,\ldots, m_k)$ is independent, then 
        \[
        [\Delta_E(m_1,\ldots, m_k)]=[\Delta(m_1,\ldots, m_k)]
        \]
        as we have previously defined it.
        \item $\gamma [\Delta_E(m_1,\ldots, m_k)]= [\Delta_E(\gamma m_1,\ldots, \gamma m_k)]$ for all $\gamma\in \mathrm{GL}_n(\mathbb{Q})$.
        \item The image of $[\Delta_E(m_1,\ldots, m_k)]$ under the face map is 
        \[
        \sum_i (-1)^{i-1} [\Delta_E(m_1,\ldots,\hat{m_i},\ldots m_k)].
        \]
        \item If $(m_1,\ldots m_k)$ is a dependent acyclic tuple, then $[\Delta_E(m_1,\ldots, m_k)]$ is zero.
    \end{enumerate}
    We will call the data of such an extension a \emph{$\Delta$-extension}; all $\Delta$-extensions by definition agree on linearly independent tuples.
    \begin{prop} \label{prop:degen}
        $\Delta$-extensions exist.    
    \end{prop}
    \begin{proof}
        We show how to construct such an $E$ inductively on the corank of $(m_1,\ldots, m_k)$: that is, on the difference $k-r$ between the number of rays and the rank of their span. When $k=r$, the definition of $[\Delta_E(m_1,\ldots, m_k)]$ is forced on us by (1); these definitions certainly satisfy (2) and (3), and (4) is not applicable. This furnishes the base case. 

        We will now first complete the inductive step so as to fulfill (2), (3), then return to analyze how one fulfills (4) in different inductive steps, as fulfilling (4) cannot be analyzed uniformly across all steps.

        Indeed, to satisfy (2), all we need to do is construct $\Delta_E(m_1,\ldots, m_k)$ for an arbitrary representative of each $\mathrm{GL}_n(\mathbb{Q})$-orbit of tuples and extend by group translation, since the conditions (3) and (4) are certainly translation-invariant. Thus, assume that we have constructed $\Delta_E$ for corank up to $i-1$, and we wish to construct it for tuples of corank $i$. Then the point is that 
        \begin{equation} \label{eq:alt}
        \sum_i (-1)^{i-1} [\Delta_E(m_1,\ldots,\hat{m_i},\ldots m_k)]
        \end{equation}
        has boundary zero, so by exactness of $\Chains(n)$ it is possible to pick some chain lifting it under the boundary map.
        
        We claim that $(m_1,\ldots, m_k)$ is acyclic, we can pick the extension so that \eqref{eq:alt} is identically zero. Indeed, in the first inductive step $i=1$, \eqref{eq:alt} is precisely a stellar subdivision relation, so it vanishes and we can pick the lift to be zero. When $i>1$, then inductively all the terms of \eqref{eq:alt} are acyclic dependent tuples, which by the inductive hypothesis are zero, and hence we can pick zero as a lift under the boundary map.
    \end{proof}

    To any $\Delta$-extension, there corresponds an explicit cocycle:

    \begin{thm} \label{thm:combrep}
        For any $(n-1)$-tuple of matrices $\underline{\gamma}=(\gamma_1,\ldots, \gamma_{n-1})$, write $c_i(\underline{\gamma})$ for 
        \[\gamma_1\gamma_{2}\ldots\gamma_i e_1\]
        for $i\ge 0$, where $e_1=(1,0,\ldots, 0)$; this is equivalently the first column of the product matrix written above. Now fix any $\Delta$-extension $E$; for such an extension, we define a $(n-1)$-cochain $\theta_E^{S^{n-1}}(n)$ by
        \begin{equation} \label{eq:thetacom}
         \underline{\gamma}\to [\Delta_E[\underline{\gamma}]]:=  [\Delta_E(c_{n-1}(\underline{\gamma}),c_{n-2}(\underline{\gamma}),\ldots, c_0(\underline{\gamma}))].
        \end{equation}
        (In particular, this is the zero class if the simplex in question is degenerate.) Then $\theta^{S^{n-1}}_E(n)$ is a cocycle representative for $\Theta^{S^{n-1}}(n)$.
    \end{thm}
    \begin{proof}
        This follows immediately from the properties of a $\Delta$-extension and the lifting process of \ref{lem:groupcohom}: we lift $1\in \mathbb{Z}$ to $e_1$, whose group coboundary is the $1$-cochain $\gamma \mapsto (\gamma-1)e_1$ which lifts to
        \[
        \gamma\mapsto [\Delta_E(c_1(\gamma),c_0(\gamma))].
        \]
        The group coboundary of this lifts to the $2$-cochain
        \[
        \gamma\mapsto [\Delta_E(c_2(\gamma_1,\gamma_2),c_1(\gamma_1,\gamma_2),c_0(\gamma_1,\gamma_2))]
        \]
        and so on.
    \end{proof}

    \begin{rem}
        The construction of Shintani cocycles (though not the motivic setting) by \cite{Hill} (and used by \cite{LP}) introduces a lexicographical order on lines in order to make the cocycle property work. From our perspective, the data of this choice of order is the same data one needs to specify a certain $\Delta$-extension, though one must be careful: their construction is actually ``dual'' to ours in the sense specified later in Section \ref{section:cones}, so this is not quite precise. To be more exact, making free use of the language and concepts of that (later) section, their lexicographic order is used to specify choices of lower-dimensional conical faces to include, which \emph{does} correspond under conical duality corresponds to the degenerate wedge classes chosen in a $\Delta$-extension.
        
        From the point of view of the cohomology class, the particular values of the cocycle (on non-acyclic tuples of group elements only) resulting from these auxiliary choices are thus not of independent significance. We thank Jeehoon Park for pointing this out. 
    \end{rem}

    \begin{rem} \label{rem:inddelta}
    Observe that the ambiguities in the choices of $\Delta$-extensions only matter for tuples such that the lines associated to $c_0,\ldots, c_{n-1}$ are dependent. In particular, if $\gamma_1,\ldots, \gamma_{n-1}$ are generators of an anisotropic torus of rank $n-1$ inside $\GL_n(\Z)$, then the value of our cocycle on $(\gamma_1,\ldots, \gamma_{n-1})$ is independent of the $\Delta$-extension.
    \end{rem}

    \subsubsection{Lifting the symbol-valued cocycles} \label{section:lift}

    The cohomology class $\Theta^{S^{n-1}}(n)$ is valued in $C(n)$; in particular, we have quotiented by the fundamental class of the sphere. As we will see in section \ref{section:realization}, the realizations of these modules will live in a \emph{quotient} of motivic cohomology groups of $\mathbb{G}_m^n$ by a rank-one submodule coming from the ``orientation obstruction'' of the fundamental class. In \cite[\S5]{SV}, the authors show that their cocycle (which agrees with ours for the case $\SL_2(\Z)$) can be lifted over this obstruction after inverting $6$; in this section, we indicate how this generalizes. In particular, we will to lift our cocycles to be valued in $\tilde{C}(n)$, i.e lift over the copy of the rank-$1$ free module $\mathbb{Z}$ corresponding to the fundamental class of $S^{n-1}$. 
    
    To begin, note that applying the lifting process of Lemma \ref{lem:groupcohom} to the length-$(n+2)$ exact complex of $\mathrm{GL}_n(\mathbb{Q})$-modules
    \begin{equation} \label{eq:topaugchains}
    \mathbb{Z}(\sgn)\to \widetilde{\Chains}(n)_\bullet
    \end{equation}
    in the same way as we did with $\Chains(n)$ yields a cocycle $\varepsilon_n^E\in C^{n}(\mathrm{GL}_n(\mathbb{Q}), \mathbb{Z}(\sgn))$. 
    \begin{prop}
        The cocycle $\varepsilon_n^E$ represents the Euler class for the standard representation of $\mathrm{GL}_n(\mathbb{Q})$.
    \end{prop}
    \begin{proof}
        The complex $\widetilde{\Chains}(n)$ is the reduced homology complex of an ind-simplicial model of the $(n-1)$-sphere; since this is a closed manifold, we can view it also by Poincaré duality, as computing cohomology in the complementary degree. Then the double complex $C^\bullet(\mathrm{GL}_n(\mathbb{Q}), \widetilde{\Chains}(n)_\bullet)$ computes the equivariant cohomology of $S^{n-1}$. By Lemma \ref{lem:groupcohom}, the sum of the lifts associated to $E$ 
        \[\ell_1+\ldots+\ell_{n+1}\]
        (using the notation of the lemma) for the lifts of $1\in\mathbb{Z}$ in the augmentation for this complex then is a representative for the Thom class of the associated sphere bundle, since the Thom class is dual to the zero section. Pulling back to the base (i.e. a point with the trivial $\mathrm{GL}_n(\mathbb{Q})$-action; equivalently, $B\mathrm{GL}_n(\mathbb{Q})$) by the zero section kills all terms except $\varepsilon_n^E=\ell_{n+1}$, which is therefore a representative of the Euler class.
    \end{proof}

    Suppose now that $\Gamma$ is an $S$-arithmetic subgroup of $\mathrm{GL}_n(\mathbb{Q})$, where $S$ is any subset of primes which has nonempty complement. The following vanishing is a theorem of Sullivan \cite{Sul}:

    \begin{thm} \label{cor:sullivan}
        The cocycle $\varepsilon_n$, restricted to $\Gamma$, is a coboundary after multiplying by the greatest common denominator $d_{n,S}$ of $m^n(m^n-1)$, as $m$ ranges over all integers divisible only by primes not in $S$. For $S$ empty, $d_{n}:=d_{n,\emptyset}$ is twice the denominator of the $n$th Bernoulli number.
    \end{thm}

    The short exact sequence
    \[
    0\to\mathbb{Z}(\sgn)\to \tilde{C}(n)\to C(n)\to 0 
    \]
    yields a long exact sequence in cohomology
    \[
    \ldots \to H^{n-1}(\Gamma, \mathbb{Z}(\sgn))\to H^{n-1}(\Gamma,\tilde{C}(n))\to H^{n-1}(\Gamma, C(n)) \to H^n(\Gamma, \mathbb{Z}(\sgn)) \to \ldots
    \]
    After inverting $d_{n,S}$, the image of $\Theta^{S^{n-1}}(n)$ (i.e. the Euler class) in the rightmost term vanishes, and hence it lifts non-uniquely to be valued in $\tilde{C}(n)$, and the set of lifts of is a torsor under the image of $d_{n,S}^{-1} \cdot H^{n-1}(\Gamma, \mathbb{Z}(\sgn))$. We thus find that:

    \begin{cor} \label{cor:lifts}
        After inverting $d_{n,S}$ and restricting to $\Gamma$, representatives for these lifts are given by
        \[
        \underline{\gamma}\mapsto [\Delta_E(\underline{\gamma})]-\phi(\underline{\gamma})
        \]
        where $\phi(\underline{\gamma})$ ranges over primitives of $\varepsilon_n^E$.
    \end{cor}

    Write $\theta_{E,\phi}^{S^{n-1}}(n)$ for the cocycle corresponding to the $(n-1)$-cochain $\phi$ transgressing the Euler cocycle; the corresponding class is
    \begin{equation} \label{eq:eulerlift}
    \Theta_\phi^{S^{n-1}}(n) = [\theta_{E,\phi}^{S^{n-1}}(n)] \in d_{n,S}^{-1} \cdot H^{n-1}(\Gamma, \tilde{C}(n)).
    \end{equation}
    Notice that the difference of two such $\phi$ is a cocycle, so the space of possible lifts over the Euler class is a torsor under $d_{n,S}^{-1} \cdot H^{n-1}(\GL_n(\Q), \Z(\sgn))$.

    \begin{rem}
        As mentioned earlier, this generalizes the construction in \cite[\S5]{SV} of a canonical lift of the analogue to the cocycle $\theta^{S^{n-1}}(2)$ for $\Gamma= \mathrm{SL}_2(\mathbb{Z})$. (When $n=2$, the choice of $E$ is immaterial, so we omit it from the notation.) In this case, $H^1(\Gamma)=H^2(\Gamma)=0$ after inverting $6$, so there is no ambiguity of lift, and the authors consequently find an explicit distinguished primitive of the Euler class $\varepsilon_2$ to give them a canonical lift. This raises the question whether there is some more canonical choice of lift in general even when $H^{n-1}(\Gamma)$ does not vanish; we do not know the answer.
    \end{rem}

    \subsubsection{The Steinberg quotient} \label{section:steinberg}
    We end this section by explaining how $\theta_E^{S^{n-1}}(n)$ can be modified into a parabolic cocycle if one quotients out by extra relations; this will also remove the need for a choice of $\Delta$-extension $E$ in the lifting process. After defining our realizations of these symbols, this will correspond to quotienting by certain elements in $K$-theory; see subsection \ref{section:steinmod} for details. This quotient will also be relevant for the application to Sharifi's conjectures.

    Let $\St(n)$ be the \emph{Orlik-Solomon complex} \cite{OS} defined as follows: it is the graded-commutative algebra generated in degree $1$ by symbols $[\ell]$ for $\ell\in \mathbb{P}^{n-1}(\mathbb{Q})$, with relations generated multiplicatively by the dependence relations
        \[
        \partial ([\ell_1]\wedge \ldots \wedge [\ell_k]) =0
        \]
    for any lines $\ell_1,\ldots, \ell_k$ spanning a space of rank strictly less than $k$. Here, the differential-graded structure is defined via $\partial [\ell] = 1$ and extended by the graded Leibniz rule. The relations generated as above are closed under the differential, and the resulting differential-graded algebra is exact. (More generally, for any configuration of lines, not just the set of \emph{all} rational ones, one can define an Orlik-Solomon complex with the same properties; this will be used later in Section \ref{section:smooth}.)
    
    Then we have a $\mathrm{GL}_n(\mathbb{Q})$-equivariant map of complexes
    \begin{equation} \label{eq:osquo}
    \Chains(n)_i \twoheadrightarrow \OS(n)_i
    \end{equation}
    sending 
    \[
    \Delta(r_1,\ldots, r_k) \mapsto [\mathbb{Q}r_1] \wedge \ldots \wedge [\mathbb{Q}r_k].
    \]
    This map \eqref{eq:osquo} is well-defined, since the Orlik-Solomon complex obeys alternation in the vertices, and stellar subdivision relations simply become dependence relations in the image.

    In top degree, the map \eqref{eq:osquo} corresponds to the $\mathrm{GL}_n(\mathbb{Q})$-equivariant quotient of the top spherical chains
    \begin{equation} \label{eq:stquo}
    R_{St}:C(n) \twoheadrightarrow \St(n):=\OS(n)_n
    \end{equation}
    where $\St(n)$ is our notation for a $\GL_n(\Q)$-module often called the \emph{Steinberg representation}. From the definition of the Orlik-Solomon algebra, we see that it can be described as generated by symbols $[\ell_1]\wedge \ldots \wedge[\ell_{n}]$ where $\ell_1,\ldots, \ell_n\in \mathbb{P}^{n-1}(\mathbb{Q})$, quotiented by the relations
    \begin{enumerate}
        \item $[\ell_1]\wedge \ldots \wedge [\ell_{n}]=0$ if the lines do not span $\mathbb{Q}^n$, and
        \item For any $\ell_0,\ldots, \ell_n$, the dependence relation \[
        \sum_{i=0}^n (-1)^i [\ell_0]\wedge \ldots \wedge  [\hat{\ell_i}] \wedge \ldots \wedge[\ell_n] = 0.
        \]
    \end{enumerate}
    From this description, one sees that the pushforward of $\theta_E^{S^{n-1}}(n)$ along \eqref{eq:stquo} is therefore independent of the choice of $E$, since all dependent tuples simply are sent to zero (as all independent tuples bound some acyclic simplex). Hence we get a cocycle 
    \[
    \theta^{St}(n):=(R_{St})_*\theta_E^{S^{n-1}}(n), (\gamma_1,\ldots, \gamma_{n-1})\mapsto [c_{n-1}(\underline{\gamma})] \wedge \ldots \wedge [c_0(\underline{\gamma})]
    \]
    independent of $E$, representing a class
    \[
    \Theta^{St}(n)\in H^{n-1}(\mathrm{GL}_n(\mathbb{Q}), \St(n))
    \]
    which one can see from relation (1) is parabolic.\footnote{After restricting to $\mathrm{SL}_n(\mathbb{Z})$, this cocycle is in fact the universal parabolic cocycle coming from Bieri-Eckman duality, as the Steinberg module is the dualizing module for $\mathrm{SL}_n(\mathbb{Z})$.} 
    
    We conclude this section with the following description of the kernel of \eqref{eq:stquo}, which will be useful when considering realizations:
    \begin{lem} \label{lem:stquo}
        The kernel of the map \eqref{eq:stquo} is generated by ``wedge'' classes of the form
        \[
        \Delta(r_1,\ldots, r_n)-\Delta(-r_1,r_2,\ldots,r_n)
        \]
        for independent tuples $(r_1,\ldots, r_n)$.
    \end{lem}
    \begin{proof}
        The top-dimensional spherical chains are generated by acyclic simplices together with stellar subdivision relations, while the Steinberg module is generated by independent tuples modulo dependence relations. If one imposes all the identification of spherical simplices
        \[
        \Delta(r_1,\ldots, r_n)\sim \Delta(\pm r_1,\ldots, \pm r_n),
        \]
        for any combination of signs, then the relations resulting from the stellar subdivision relations are precisely the dependence relations. These identifications can be deduced, by (anti)symmetry, from the identifications
        \[
        \Delta(r_1,\ldots, r_n)\sim \Delta(-r_1,r_2,\ldots,r_n)
        \]
        since we can then change one sign at a time and bootstrap to the general case.
    \end{proof}

    \section{Construction and application of the motivic cocycles} \label{section:motivictheta}

    \subsection{Realization map for top-dimensional chains} \label{section:realization}

    We now turn to constructing the realization map associating classes in motivic cohomology/Milnor $K$-theory to symbols. This section will be devoted to proving the following theorem:
    
    \begin{thm} \label{thm:gerssym}
    There exists a map of $\mathrm{GL}_n(\mathbb{Q})$-modules
    \[\tilde{\rho}:\widetilde{\Chains}(n)_{n} \to K_M^n(k(\mathbb{G}_m^n))^{(0)}.\]
    defined on generators by\footnote{Here, the matrix columns can be any non-zero vector in the ray; see discussion following the theorem statemnt for why this is well-defined.}
    \begin{equation} \label{eq:gerssymdef}
    [\Delta(m_1,\ldots, m_n)] \mapsto \begin{pmatrix}m_1&\ldots & m_n\end{pmatrix}_* \{1-z_1,\ldots, 1-z_n\} \in K_n^M(k(\mathbb{G}_m^n))^{(0)}
    \end{equation}
    The image of the fundamental class in $\Z\cong H_{n}(\widetilde{\Chains}(n))\subset \widetilde{\Chains}(n)$ 
    is the class $\{-z_1,\ldots, -z_n\}$, so $\tilde{\rho}$ descends to a map
    \[\rho:\Chains(n)_{n} \to K_M^n(k(\mathbb{G}_m^n))^{(0)}/\{-z_1,\ldots, -z_n\}.\]
    \end{thm}
    
    The hard part of the theorem is proving the relations, so we first view $f_k$ as a map from $\mathbb{Z}\{\mathbf{T}_{n+1}\}$. 
    
    We note that replacing $m_i$ by a scalar multiple is immaterial because each $1-z_i$ is invariant under $[a]_*:\mathbb{G}_m\to \mathbb{G}_m$ in the corresponding coordinate, so pre-composing $\begin{pmatrix}m_1&\ldots & m_i\end{pmatrix}_*$ with these isogenies does not change the definition of $\rho$. The $\mathrm{GL}_n(\mathbb{Q})$-equivariance of the definition is then completely formal, from functoriality of pushforwards.
    
    It remains to check that the relations in $\cS_\bullet$ between the classes of simplices hold. By Proposition \ref{prop:simplesub}, it suffices to check the acyclic stellar subdvision relations.
    
    \begin{prop} \label{prop:gersrel}
    For any integer $2\le r \le n$, each acyclic independent tuple $\underline{m}=(m_1,\ldots, m_n)$ of rays and ray $m$ lying on the great circle corresponding to the face spanned by $(m_1,\ldots, m_n)$ and sharing some hemisphere with all of them, the relation
        \begin{equation}
        \rho[\Delta(m_1,\ldots, m_k)]=\sum_{i=1}^r \rho[\Delta(m_1,\ldots, \hat{m_i},m,m_{i+1},\ldots, m_n)]
        \end{equation}
    coming from \eqref{eq:simpgen} holds. 
    \end{prop}
    \begin{proof}
        We may reduce to the case $r=n$ as follows: the claimed relation can be written as 
        \[
        \begin{pmatrix}m_1&\ldots & m_n\end{pmatrix}_* \{1-z_1,\ldots, 1-z_n\} = \sum_{i=1}^r \begin{pmatrix}m_1&\ldots& \hat{m_i} & m & \ldots   & m_k\end{pmatrix}_* \{1-z_1,\ldots, 1-z_n\}.
        \]
        The left-hand side factors as the cup product
        \[
        \begin{pmatrix}m_1&\ldots & m_r\end{pmatrix}_* \{1-z_1,\ldots, 1-z_r\} \smile \begin{pmatrix}m_{r+1}&\ldots & m_n\end{pmatrix}_* \{1-z_{r+1},\ldots, 1-z_n\} 
        \]
        and the right-hand side as 
        \[
        \left(\sum_{i=1}^r \begin{pmatrix}m_1&\ldots& \hat{m_i} & m & \ldots   & m_r\end{pmatrix}_* \{1-z_1,\ldots, 1-z_r\}\right)
        \smile \begin{pmatrix}m_{r+1}&\ldots & m_n\end{pmatrix}_* \{1-z_{r+1},\ldots, 1-z_n\} 
        \]
        so it suffices to prove that 
        \[
        \begin{pmatrix}m_1&\ldots & m_r\end{pmatrix}_* \{1-z_1,\ldots, 1-z_r\} = \sum_{i=1}^r \begin{pmatrix}m_1&\ldots& \hat{m_i} & m & \ldots   & m_r\end{pmatrix}_* \{1-z_1,\ldots, 1-z_r\}
        \]
        But this is the pullback of a top-rank stellar relation from the quotient $\mathbb{G}_m^n\twoheadrightarrow \mathbb{G}_m^n/G$, for $G$ the image of 
        \[
        \begin{pmatrix}m_{r+1}&\ldots & m_n\end{pmatrix}: \mathbb{G}_m^r \to \mathbb{G}_m^n.
        \]
        We therefore henceforth assume that $r=n$, and need to prove the relation 
        \[
        \begin{pmatrix}m_1&\ldots & m_n\end{pmatrix}_* \{1-z_1,\ldots, 1-z_n\} - \sum_{i=1}^n \begin{pmatrix}m_1&\ldots& \hat{m_i} & m & \ldots   & m_n\end{pmatrix}_* \{1-z_1,\ldots, 1-z_n\} =0.
        \]
        Taking the Bloch cycle description
        \[
        K_n^M(k(\mathbb{G}_m^n)) \hookrightarrow z^n(k(\mathbb{G}_m^n)\times \Box^n)/\partial z^{n+1}(k(\mathbb{G}_m^n)\times \Box^{n+1})
        \]
        we see that it suffices to show that 
        \begin{equation} \label{eq:finalstell}
        \sum_{i=1}^{n+1} (-1)^i \begin{pmatrix}m_1&\ldots& \hat{m_i} & \ldots   & m_{n+1}\end{pmatrix}_* \Gamma(1-z_1,\ldots, 1-z_n) \in \partial z^{n+1}(k(\mathbb{G}_m^n)\times \Box^{n+1}).
        \end{equation}
        for all acyclic tuples of rays $(m_1,\ldots, m_{n+1})$ such that each sub-$n$-tuple is full rank. Formally, \eqref{eq:finalstell} is the image under the cubical face maps of
        \begin{equation} \label{eq:finalcycle}
        \begin{pmatrix}m_1&\ldots & m_{n+1}\end{pmatrix}_* \Gamma(1-z_1,\ldots, 1-z_{n+1})\in \partial z^{n+1}(k(\mathbb{G}_m^n)\times \Box^{n+1})
        \end{equation}
        where the matrix denotes the map
        \begin{equation}\label{eq:n+1}
        \begin{pmatrix}m_1&\ldots & m_{n+1}\end{pmatrix}: \mathbb{G}_m^{n+1}\to \mathbb{G}_m^n.
        \end{equation}
        However, this seems not to use the acyclicity condition; what gives? The problem is that \eqref{eq:n+1} is not a finite map, so the pushed forward cycle may have improper intersection with the faces. Indeed, we claim that \eqref{eq:finalcycle} intersects all faces properly exactly when $(m_1,\ldots, m_{n+1})$ is acyclic. We will prove this on the level of cycles in 
        \[
        \mathbb{G}_m^n \times \Box^{n+1}
        \]
        since this certainly implies it for the restriction to the generic point. Write $M$ for the matrix $(m_1,\ldots, m_{n+1})$, and define the $(n+1)$-variable monomials
        \[
        p_j(z_1,\ldots, z_{n+1}) = \prod_{i=1}^{n+1} z_i^{M_{j,i}}
        \]
        whose exponents are the $j$th row of $M$. Then in $\mathbb{G}_m^n \times \Box^{n+1}$, the cycle \eqref{eq:finalcycle} can be described as the closure of the locus
        \[
        (p_1, p_2, \ldots, p_n, 1-z_1,\ldots, 1-z_{n+1}). 
        \]
        Its intersection with a codimension-$d$ cubical face is then indexed by a labelled subset $I\in \{1,\ldots, n+1\}$ of cardinality $d$, with each element $i$ of the subset labelled by $\ell_I(i)\in \{0,\infty\}$; this face is given by the intersection of the cycle with the locus
        \[
        \bigcap_{i\in I} \{ 1-z_i = \ell_I(i)\} 
        \]
        i.e. fixing each $z_i$ corresponding to the indexing set to be either $1$ or $\infty$. We wish to check, for each $I$, whether or not this intersection has the correct codimension, i.e. codimension $n+|I|$; the cycle \eqref{eq:finalcycle} meets all faces properly if and only if all codimensions are correct. 

        First consider the case where $I$ has at least one label of $0$; without loss of generality, we assume that it corresponds to $1\in I$, i.e. the relation $1-z_1=1\Leftrightarrow z_1=0$. The intersection of \eqref{eq:finalcycle} with $\{z_1=0\}$ is then the closure of the locus
        \begin{equation} \label{eq:minus1loc}
        (p_1(z_1=0),\ldots, p_n(z_1=0), 0, 1-z_2,\ldots, 1-z_n).
        \end{equation}
        where the notation indicates that we plug in $0$ in the $i$th place. Write $M'$ for the submatrix of $M$ given by deleting the first column; by assumption, it has full rank $n$. Viewing $M'$ as associated to a map
        \[
        \mathbb{G}_m^n \times \Box^n \to \mathbb{G}_m^n \times \Box^{n+1}
        \]
        by its natural action on the toric part, and by inclusion $\Box^n\hookrightarrow \Box^{n+1}$ in the last $n$ coordinates, we see that the locus \eqref{eq:minus1loc} is then the finite pushforward
        \begin{equation}\label{eq:minus1cyc}
        M'_* \Gamma(1-z_1,\ldots, 1-z_n).
        \end{equation}
        The intersection of \eqref{eq:finalcycle} with a face corresponding to $I$ is then the intersection of \eqref{eq:minus1cyc} with the face corresponding to the labelled set $I\setminus \{1\}$. But \eqref{eq:minus1cyc} is a finite pushforward of a cycle meeting all faces properly, so we conclude that \eqref{eq:finalcycle} meets the face corresponding to $I$ properly as well. Thus, any face with at least one label of $0$ always intersects \eqref{eq:finalcycle} properly; it therefore suffices to check the intersection with faces labelled only with $\infty$s. 

        We claim that if $(m_1,\ldots, m_{n+1})$ is acyclic, this intersection is always empty, and thus trivially proper.\footnote{When $(m_1,\ldots, m_{n+1})$ fails to be acyclic, the corresponding stellar simplicial relation should not hold, and thus \eqref{eq:finalcycle} must not meet all faces properly. The simplest example of what happens in this case: if $n=1$, $m_1=1$, and $m_2=-1$, then the locus \eqref{eq:finalcycle} is the closure of $(z_1z_2^{-1},1-z_1,1-z_2)$. Its intersection with the unique codimension-two face labelled with two $\infty$s should therefore be codimension $3$, i.e. zero-dimensional. However, the points of the curve parameterized by $(g,\infty,\infty)$ are in the closure for all $g\in \mathbb{G}_m$, since the point corresponding to each fixed $g$ is in the closure (as $t\to \infty$) of the curve $z_1=gt$, $z_2=t$ with free parameter $t$.} Note that the property of having empty intersection with the $\infty$-labelled faces is invariant under left multiplication of $M$ by elements of $\mathrm{GL}_n(\mathbb{Q})\cap M_n(\mathbb{Z})$. Via left multiplication by such a matrix, we can always turn the first $n$ columns of the matrix into scalar multiples of the standard basis. Thus, it suffices to check matrices of the form
        \[
        \begin{pmatrix}
        x_1 &  & & & y_1 \\
            & x_2    &  & & y_2 \\
            &      & \ldots &  & \ldots \\
            &        & & x_n   & y_n
        \end{pmatrix}
        \]
        where each of the $x_i$ and $y_i$ must be nonzero by the assumption that every $n$-by-$n$ submatrix of $M$ is full rank, and $\sgn(x_i)=\sgn(y_i)$ for at least one $i$ by the acyclicity assumption.

        In this case, suppose without loss of generality that $I=\{1,2,\ldots, k\}$ for some $k\le n$, with $\ell(j)=\infty$ for each $j\in I$. Assume now for the sake of contradiction that \eqref{eq:finalcycle} intersects the face corresponding to $I$ nontrivially; in particular, suppose it contains a point with $z_1=z_2=\ldots=z_k=\infty$ but 
        \[
        p_i=z_i^{x_i}z_{n+1}^{y_i} \in \mathbb{G}_m(\overline{k})
        \]
        for $\sgn(x_i)=\sgn(y_i)$. If $i\in \{1,2,\ldots, k\}$, this is impossible, since at this point we must have $z_{n+1}=0\Rightarrow 1-z_{n+1}=1$, which means our point fails to lie in the algebraic cube $\Box^{n+1}$. Otherwise, $x_j$ and $y_j$ have opposite signs for $j=1,2,\ldots, k$, meaning that $z_{n+1}=\infty$ as well at this point. But then
        \[
        p_i=z_i^{x_i}z_{n+1}^{y_i} \in \mathbb{G}_m(\overline{k})
        \]
        implies that $z_i=0\Rightarrow 1-z_{i}=1$, again a contradiction. We conclude the intersection with the face corresponding to $I$ is in fact empty, as desired.

        It remains only to show that the fundamental class of $S^{n-1}$ is sent to the generator
        \[
        \{-z_1,\ldots, -z_n\}.
        \]
        Indeed, we can decompose the fundamental class as a sum of simplicial orthants 
        \[
        [S^{n-1}]=\sum_{I\in \{\pm 1\}^n} [\Delta(I)]
        \]
        where $\Delta(I)$ is the simplex corresponding to $\sigma(I) (I_1\cdot e_1,\ldots, I_n\cdot e_n)$ where $\sigma(I)$ is an arbitrary even permutation if $(I_1\cdot e_1,\ldots, I_n\cdot e_n)$ is positively oriented, is an arbitrary odd permutation otherwise. Under $f_n$, this is sent to the sum
        \[
        \sum_{I\in \{\pm 1\}^n} \sigma \{ 1-z_1^{I_1},\ldots, 1-z_n^{I_n}\} = \left\{ \frac{1-z_1}{1-z_1^{-1}} ,\ldots \frac{1-z_n}{1-z_n^{-1}} \right\} = \{ -z_1,\ldots, -z_n\}.
        \]
    \end{proof}
    With the realization map in hand, the proof of Theorem \ref{thm:thm1} is complete. In particular, we obtain the following classes: 
    \[
    \Theta(n) =  \rho_*\Theta^{S^{n-1}}(n) \in H^{n-1}(\mathrm{GL}_n(\mathbb{Q}), (K_n^M(k(\mathbb{G}_m^n))/\{-z_1,\ldots, -z_n\})^{(0)}).
    \]
    and, if $\Gamma\subset \mathrm{GL}_n(\mathbb{Q})$ is $S$-arithmetic for a co-nonempty set of primes $S$, for any transgression $\phi$ of $\varepsilon_n$ a class
    \[
    \Theta_\phi(n) =  \rho_*\Theta_\phi^{S^{n-1}}(n) \in H^{n-1}(\Gamma, K_n^M(k(\mathbb{G}_m^n))^{(0)})[d_{n,S}^{-1}].
    \]
    For any $\Delta$-extension $E$, the cocycle representative $\theta_E(n):=\rho_*\theta_E^{S^{n-1}}(n)$ for the former class is given by
    \[
    (\gamma_1,\ldots, \gamma_{n-1}) \mapsto \begin{pmatrix}c_1&\ldots & c_k\end{pmatrix}_* \{1-z_1,1-z_2,\ldots, 1-z_n\} \in  K_n^M(k(\mathbb{G}_m))/\{-z_1,\ldots, -z_n\}
    \]
    with $c_i=\gamma_i\ldots \gamma_1 e_1$ whenever these columns are independent; for any fixed $\Delta$-extension $E$, one can equally in principle work out the image of any tuple with non-independent such $c_i$, though we do not currently see a systematic way to do this. Similarly, the latter class is represented by $\theta_{E,\phi}(n):=\rho_*\theta_{E,\phi}^{S^{n-1}}(n)$ and under the same assumptions sends 
    \begin{equation} \label{eq:thecocycle}
    (\gamma_1,\ldots, \gamma_{n-1}) \mapsto \begin{pmatrix}c_1&\ldots & c_k\end{pmatrix}_* \{1-z_1,1-z_2,\ldots, 1-z_n\} - \phi(\underline{\gamma}) \{-z_1,\ldots, -z_n\}
    \end{equation}
    in $K_n^M(k(\mathbb{G}_m))[d_{n,S}^{-1}]$. 
    
    \subsubsection{Modular symbols from the Steinberg quotient} \label{section:steinmod}

    We now describe the realization of the parabolic, Steinberg module-valued, cocycle of section \ref{section:steinberg}. This amounts to determining the image of the kernel of \eqref{eq:stquo} under the map $f$, which by Lemma \ref{lem:stquo}, is generated by the $\mathrm{GL}_2(\mathbb{Q})$-orbit of 
    \[
    \{1-z_1,1-z_2,\ldots, 1-z_n\} - \{1-z_1^{-1},1-z_2,\ldots, 1-z_n\} = \{-z_1,1-z_2,\ldots, 1-z_n\}.
    \]
    We thus see it suffices to quotient out by the degree-$n$ part of the ideal $I$ of the Milnor $K$-theory ring generated by the symbols $-z_i\in K_1^M(k(\mathbb{G}_m))$, whereupon we obtain a $\mathrm{GL}_2(\mathbb{Q})$-equivariant map
    \[
    \Theta^{MS}(n):\St(n) \to (K_n^M(k(\mathbb{G}_m))/I)^{(0)}
    \]
    sending
    \[
    [\ell_1]\wedge \ldots \wedge [\ell_n]\mapsto \begin{pmatrix}\ell_1&\ldots & \ell_k\end{pmatrix}_* [\{1-z_1,1-z_2,\ldots, 1-z_n\}] \in  K_n^M(k(\mathbb{G}_m))/I.
    \]
    From this modular symbol, one can also deduce an explicit parabolic cocycle 
    \[H^{n-1}(\mathrm{GL}_2(\mathbb{Q}), (K_n^M(k(\mathbb{G}_m))/I)^{(0)}\] 
    represented by
    \[
    (\gamma_1,\ldots, \gamma_{n-1}) \mapsto \begin{pmatrix}c_1&\ldots & c_k\end{pmatrix}_* [\{1-z_1,1-z_2,\ldots, 1-z_n\}] \in  K_n^M(k(\mathbb{G}_m))/I
    \]
    with $c_i=\gamma_i\ldots \gamma_1 e_1$. When we wish to use the Steinberg quotient, however, we will generally work directly with the modular symbol, as it retains more information than the associated cohomology class or cocycle.

    \begin{rem} \label{rem:eqpoly}
        As suggested in \cite[\S5]{SV}, the cocycles of the form $\Theta(n)$ come from equivariant motivic polylogarithms for the action of $\GL_n(\Z)$ on the group scheme $\G_m^n$. The argument of Sharifi-Venkatesh essentially proves this for $n=2$ by realizing their chain complex in a Gersten complex computing motivic cohomology; our realization map could similarly be extended to map from the whole spherical chain complex to a Gersten complex as well. However, for general $n$, the Gersten complex does not necessarily compute motivic cohomology, so we would need instead a symbol complex with fewer relations, in order to map to the Bloch cycle complex (which \emph{does} compute motivic cohomology); such a complex can in fact be constructed using matroids. We omit these arguments from this article due to their considerable length and technical overhead, and irrelevance to our main results; however, we will use the matroid approach and prove the relationship with polylogarithms in the sequel to this article, in the elliptic setting (which degenerates at the cusps to the setting of the present article). In that setting, having the formalism of equivariant motivic polylogarithms is useful for comparison reasons.

        However, for the \emph{regulator} of the motivic cocycle, we have included the proof the comparison with an equivariant polylogarithm class in de Rham/coherent cohomology, in Appendix \ref{appendix:a}, as this requires less technical overhead. However, the flavor of the argument would be exactly the same in motivic cohomology.
    \end{rem}

    \subsection{Specialization at torsion sections and Sharifi maps}
    
    For briefness, we have only been working over the generic point until now, but to construct and analyze the properties of the maps in Theorem \ref{thm:sharifi}, we will need to consider specializations of our cocycles to torsion subschemes: let $\Gamma\subset \mathrm{GL}_n(\mathbb{Q})$ be any subgroup. The proof of Theorem \ref{thm:gerssym} applies identically to show that the realization map factors, as a $\Gamma$-map, through $H^n(U_\Gamma, \Z(n))^{(0)}$, where $U_\Gamma$ is defined by
    \begin{equation} \label{eq:spreadgers}
        U_\Gamma := \varinjlim_H \mathbb{G}_m^n-H
    \end{equation}
    and the direct limit ranges over finite subarrangements of the hyperplane arrangement which is the $\Gamma$-orbit of 
    \[
    \mathbb{G}_m^{n-1}\subset \mathbb{G}_m^{n}
    \]
    embedded as the kernel of $1-z_n$. For $n\ge 3$, the norm residue isomorphism theorem and cohomological dimension arguments show that all pulled back cocycles by torsion points are trivial; thus, we work only in the case $n=2$. (For $n=1$, we get classical cyclotomic $N$-units, for pullback by $N$-torsion.) 

    Suppose now that $\Gamma=\Gamma_0(N)$ fixes the line generated by a torsion section $s_N:\mathrm{Spec }\, \Q(\mu_N)\to \mathbb{G}_m^2$ which is $(1,\zeta_N)$ in the coordinates $z_1,z_2$. 
    
    While $\Gamma_0(N)$ does not fix $s_N$, since it fixes the line generated by $x$, we can define a homomorphism
    \[
        \sigma: \Gamma_0(N) \to (\Z/N\Z)^\times
    \]
    by the rule $\gamma s_N = \sigma(\gamma) s_N$, with kernel the index-$\varphi(N)$ subgroup $\Gamma_1(N)$ fixing $s_N$. We can then define an action of $\Gamma_0(N)$ on $\Q(\mu_N)$ by 
    \[
        \gamma \mapsto ([\sigma(\gamma^{-1})]:\zeta_N \mapsto \zeta_N^{\sigma(\gamma^{-1})}),
    \]
    which yields also pullback maps on motivic cohomology. Now, the pullback 
    \[
    s_N^*:H^2(U_\Gamma, \Z(2))^{(0)} \to H^n(\Q(\mu_N), \mathbb{Z}(n))
    \]
    is a priori only $\Gamma_1(N)$-equivariant, but with our newly defined action of $\Gamma_0(N)$ on the right-hand side, one can check (cf. \cite[\S4]{SV}) that it is actually $\Gamma_0(N)$-equivariant. We thus get a corresponding specialization of our Eisenstein cocycle
    \[
    s_N^* \Theta(2) \in H^{1}(\Gamma_0(N), H^2(\Q(\mu_N), \mathbb{Z}(2)))
    \]
    
    In fact, the localization sequence \eqref{prop:loc} for 
    \[
    \bigoplus_{\mathfrak{p} \nmid N} \Z[\mu_N]/\mathfrak{p} \hookrightarrow \Z[\mu_N, N^{-1}]
    \]
    yields a left-exact sequence
    \[
    H^2(\Z[\mu_n, N^{-1}], \Z(2)) \hookrightarrow K^M_2(\Q(\mu_N))\xrightarrow{\partial} \bigoplus_{\mathfrak{p} \nmid N} K^M_1(\Z[\mu_N]/\mathfrak{p})
    \]
    where the last arrow is are the tame symbols which vanish on integral-at-$\mathfrak{p}$ elements. The injectivity is because the Milnor $K$-theory of finite fields vanishes above degree $1$, as one can always express $1$ as the sum of quadratic residues. Taking the restriction of the action via $\sigma$ of $\Gamma_0(N)$ to the cohomology of $\Z[\zeta_N,N^{-1}]$, we see that our pullbacks are actually valued in the integral-away-from-$N$ submodule, i.e. we actually have a cocycle
    \[
    s_N^*\Theta(2) \in H^{1}(\Gamma_0(N),H^n(\Z[\zeta_N,N^{-1}], \Z[\tfrac{1}{2}](2)))
    \]
    This argument will be applied implicitly also to everything that follows, as well in the next subsection.
    
    One can equally make a ``Steinberg'' version of this construction: if we let $\St(2)^\circ$ be the $\Gamma_0(N)$-invariant submodule of $\St(2)$ generated by lines not reducing to $[0:1]$ modulo $N$ (i.e. cusps not in the $\Gamma_0(N)$-orbit of $[0:1]$), then the argument of section \ref{section:steinmod} affords us a map
    \[
    \Theta^{MS}(2)_{(N)}: \St(2)^\circ \to \varinjlim_H \,H^2(\mathbb{G}_m^2-H,\mathbb{Z}(2))/I_{(N)}
    \]
    where $I_{(N)}$ is module of relations spanned by symbols in the orbit of 
    \begin{equation} \label{eq:idef}
    \{-z_1,1-z_2\}
    \end{equation}
    under matrices of the form $L=\begin{pmatrix}\ell_1 & \ell_2\end{pmatrix}$, for $\ell_1,\ell_2$ cusps not in the $\Gamma_0(N)$-orbit of $[0:1]$, as above. This affords us a modular symbol
    \[
    s_N^*\Theta^{MS}(2)_{(N)}: [\ell_1]\wedge [\ell_2]\mapsto s_N^* \begin{pmatrix} \ell_1& \ell_2\end{pmatrix}_*\{1-z_1,1-z_2\} \in  H^2(\Z[\mu_N,N^{-1}], \Z[\tfrac{1}{2}](2))/\text{extra relations from }I_{(N)}
    \]
    After pullback, the module of relations $I_{(N)}$ consists of elements of the form 
        \[
        \{-\zeta_{N}^i, -\zeta_{N}^j\}, \{-\zeta_{N}^i, 1-\zeta_{N}^j\}.
        \]
    The former relations $\{-\zeta_{N}^i, -\zeta_{N}^j\}=0$ are all true up to $2$-torsion by alternation of the Steinberg symbol. The latter relations vanish upon taking the projection onto the plus part
    \[
    (\bullet )_+:  \{x,y\} \mapsto\frac{1}{2}(\{x,y\} + \{\overline{x}, \overline{y}\}),
    \]
    since, with $2$ inverted, we have
    \[
    2 \{-\zeta_{N}^i, 1-\zeta_{N}^j\}_+ = \{\zeta_{N}^{-i}, 1-\zeta_{N}^{-j}\} + \{\zeta_{N}^i, 1-\zeta_{N}^j\} = \{\zeta_N^i,- \zeta_N^j\} = 0.
    \]
    If we restrict now to the subgroup $\Gamma_1(N)$ fixing $s_N:=(1,\zeta_N)$, we thus obtain a specialization
    \[
    (\Pi_N^\circ)_+:=(s_N^* \Theta^{MS}(2)_{(N)})_+:\St(2)^\circ \to H^2(\Z[\zeta_N, N^{-1}],\Z[\tfrac{1}{2}](2))_+
    \]
    which is a $\Gamma_1(N)$-invariant modular symbol with the trivial action on the target. However, if $N=p^k$, the relations $\{-\zeta_{N}^i, 1-\zeta_{N}^j\}$ are zero before projection: the cyclotomic distribution relation 
    \[
        1-\zeta_{p^k}^{pt} = \prod_{i=0}^{p-1} 1-\zeta_{p^k}^{t+p^{k-1}i}
    \]
    means we can assume $(j,p)=1$, in which case 
    \[
    \{\zeta_{N}^i, 1-\zeta_{N}^j\}= k\{\zeta_{N}^j, 1-\zeta_{N}^j\}=0
    \]
    by the Steinberg relation $\{x,1-x\}=0$, where here $kj\equiv i \pmod{p^k}$. Thus in this case, we have a map
    \[
        \Pi_N^\circ:=(s_N^* \Theta^{MS}(2)_{(N)}):\St(2)^\circ \to H^2(\Z[\zeta_N, N^{-1}],\Z[\tfrac{1}{2}](2))
    \]
    In the remainder of this section, for brevity of notation we will write everything without the $+$ projection, with the understanding that we always mean the $+$ part except for when $N=p^s$.

    Consider now a \emph{unimodular} symbol $[\ell_1]\wedge [\ell_2]$, i.e. so that one for which the associated matrix formed from integer generators of these lines satisfies
    \[
    \det \begin{pmatrix}\ell_1 & \ell_2\end{pmatrix} \in \GL_2(\Z).
    \]
    The map of \cite[\S4]{SV} is defined in terms of these, writing all symbols as sums thereof via their ``connecting sequences'' (which we avoid due to using pushforwards). To compare our map with theirs, we have that if $\ell_1,\ell_2\ne [0:1]$ (meaning that that $s_N$ is not in the polar locus and the pullback is well-defined),
    \[
    s_N^* \Theta^{MS}(2)_{(N)}([\ell_1]\wedge [\ell_2])) = \left(\begin{pmatrix}\ell_1 & \ell_2\end{pmatrix}^{-1} s_N\right)^* \{1-z_1,1-z_2\}= \{ 1- \zeta_N^{-v}, 1-\zeta_N^u\}  \in H^2(\Z[\zeta_N, N^{-1}],\Z[\tfrac{1}{2}](2))
    \]
    where $(u,v)$ is the top row of the matrix $\begin{pmatrix}\ell_1 & \ell_2\end{pmatrix}$. On symbols disjoint from $[0:1]$, therefore, up to sign and convention of basis, our specialization coincides with the Eisenstein cocycle defined in \cite[\S4]{SV}, as they agree on unimodular generators. (In loc. cit., $(u,v)$ is the bottom row of the analogous matrix; this corresponds to swapping the role of the standard basis lines $e_1$ and $e_2$.)

    \begin{rem}
        The above calculation shows that the relations of the form $\{-\zeta_{N}^i, 1-\zeta_{N}^j\}$, where $\zeta_N^i$ and $\zeta_N^j$ have orders divisible by distinct primes, also do not appear if we omit symbols in the Steinberg module containing lines $[a:b]$ with $a$ divisible by primes dividing $N$; thus, the full (non-plus part) $\Pi_N^\circ$ can also be constructed (and shown to be Hecke equivariant, etc., as below) if we are willing to restrict our set of symbols: this corresponds to restricting the set of cusps to those not in the $\Gamma_0(p)$-orbit of $\infty$ for any prime $p|N$ (in the language of the following section; see below). This remark explains how the calculations of \cite[\S4]{SV} result in a parabolic cocycle while only inverting $2$ in the coefficients, as they only work with homology of the closed curve.
    \end{rem}

    \subsubsection{Hecke operators and modular symbols}

    In this section, we define the Hecke operators and prove Theorem \ref{thm:sharifi}, for 
    \[
        (\Pi^\circ_N)_+ := (s_N^*\Theta^{MS}(2)_{(N)})_+: \St(2)^\circ_{\Gamma_1(N)} \to H^2(\Z[\mu_N,N^{-1}], \Z[\tfrac{1}{2}])_+
    \]
    though as discussed previously, we will omit the $+$ signs for ease of notation, with the understanding that this is only \emph{actually} allowed if $N=p^s$ (or by omitting cusps, as in the remark at the end of last section). Let us explain the notation (also used in Theorem \ref{thm:sharifi}). Write $X_1(N)$ for the compactified modular curve of level $\Gamma_1(N)$, and $C_1(N)$ for its set of cusps. Then there is a natural identification between coinvariants of the Steinberg module and the Borel-Moore homology:
    \[
    \St(2)_{\Gamma_1(N)} \xrightarrow{\sim} H_1(X_1(N), C_1(N), \Z)
    \]
    sending $[\ell_1]\wedge [\ell_2]$ to the image of the geodesic path usually denoted $\{\ell_1,\ell_2\}$ between the cusps corresponding to $\ell_1,\ell_2\in \P^1(\Q)$ under the uniformization of $Y_1(N):=X_1(N)-C_1(N)$ by the complex upper half-plane; similarly, the co-invariants of Steinberg symbols $\St(2)^\circ_{\Gamma_1(N)}$ disjoint from $\Gamma_1(N)[0:1]$ can be identified with the homology relative to the \emph{restricted} cusps $C_1(N)^\circ$ disjoint from $\Gamma_0(N)\cdot \infty$ \cite{AR}. Hence, modular symbols (respectively, modular symbols restricted away from $\infty$, modular symbols restricted to a single cusp) valued in a trivial $\Gamma_1(N)$-module can be identified with maps from the homology of $(X_1(N),C_1(N))$ (respectively, $(X_1(N),C_1(N)^\circ)$). More generally, we can take any $\Gamma_1(N)$-invariant set of cusps, and obtain as coinvariants the homology relative to only those cusps.

    Now we define Hecke operators on such symbols. A useful notion will be the \emph{adjugate} of an invertible matrix
    \[
    \mathrm{adj}(M) := (\det M)M^{-1}.
    \]
    Note that the adjugate of an integer matrix is an integer matrix, and that the scalar $\det M = M\cdot \mathrm{adj}(M)$ acts trivially on the Steinberg module, and hence on the image of any modular symbol equivariant for a group containing these matrices: in other words, $M$ and its adjugate act as inverses on such elements.
    
    Now, for any double coset 
    \[
    \Gamma_1(N)\alpha \Gamma_1(N) \in \Gamma_1(N)\backslash \GL_2(\Q)/\Gamma_1(N)
    \]
    with left coset decomposition
    \[
    \Gamma_1(N)\alpha \Gamma_1(N) = \bigcup_i \alpha_i \Gamma_1(N),
    \]
    we can define an associated operator $T_\alpha$ on a modular symbol $c: \St(2)_{\Gamma_1(N)} \to A$ (for an abelian group $A$) by
    \[
        (T_\alpha c)([\ell_1,\ell_2]) = \sum_i c(\alpha_i^{-1} [\ell_1,\ell_2]).
    \]
    One can check this is well-defined independent of the choice of left coset representatives, and corresponds to the classical Hecke action by correspondences on (co)homology of modular curves (and thereby on modular forms, etc.).

    First, if $\alpha_d\in \Gamma_0(N)$ with $\sigma(\alpha)=d$ (i.e. the lower right entry is $d$ modulo $N$), then from previous discussion, the double coset $T_\alpha$ depends only on $d$; we write $T_\alpha=\langle d\rangle$ and call it a diamond operator.

    Next, for any prime $p \nmid N$, we have the double coset operator 
    \[
    T_p:=\Gamma_1(N) \begin{pmatrix} p & \\ & 1 \end{pmatrix} \Gamma_1(N)
    \]
    and its dual
    \[
    T_p^*:=\Gamma_1(N) \begin{pmatrix} 1 & \\ & p \end{pmatrix} \Gamma_1(N)
    \]
    One computes the relation $T_p^* = \langle p \rangle T_{p}$. We will consider the following set of coset representatives for $T_p^*$: by \cite[Theorem 4.3.7]{SV}, there exists a set of representatives $\alpha_{i,p}$, $0\le i \le p$, given by 
    \[
    \alpha_{i,p} = \begin{pmatrix}1 & \\ i & p \end{pmatrix}
    \]
    for $i<p$, and 
    \[
    \alpha_{p,p} = \begin{pmatrix}p & \\  & 1 \end{pmatrix} \alpha_p.
    \]
    Here, $\alpha_{p}$ is the representative for the diamond operator $\langle p\rangle$. Note that that considered as maps $\Q^2/\Z^2 \to \Q^2/\Z^2$, the kernel of these $p+1$ matrices are the $p+1$ subgroups of order $p$, and that $\text{adj}(\alpha_i)$ fixes $(0,1)$ modulo $N$.
    
    For primes $p|N$, we similarly have 
    \[
    U_p:=\Gamma_1(N) \begin{pmatrix} p & \\ & 1 \end{pmatrix} \Gamma_1(N).
    \]
    We will consider the left coset decomposition
    \[
    U_p = \bigcup_{i=0}^{p-1} \begin{pmatrix}p & Ni \\ & 1 \end{pmatrix} \Gamma_1(N).
    \]
    This operator also has a dual, which we ignore (since our symbol will have no simple corresponding equivariance property).
    \begin{proof}[Proof of Theorem \ref{thm:sharifi}]
    From previous results, we find that 
    \[
    \langle d \rangle  s_N^*\Theta^{MS}(2)_{(N)}([\ell_1,\ell_2]) = s_N^* \alpha_*^{-1} \Theta^{MS}(2)_{(N)}([\ell_1,\ell_2]) = [d]^* s_N^*\Theta^{MS}(2)_{(N)}([\ell_1,\ell_2])
    \]
    since $\alpha s_N= [d] s_N$. 
    
    We conclude that $\langle d \rangle \Pi_N^\circ = [d]^* \Pi_N^\circ$, as desired.

    For the operator $T_p^*$, $p\nmid N$, we have
    \begin{align}
        T_p^* s_N^*\Theta^{MS}(2)_{(N)}([\ell_1,\ell_2]) &= \sum_{i=0}^p s_N^*\Theta^{MS}(2)_{(N)}(\alpha_{i,p}^{-1}[\ell_1,\ell_2]) \\
        & = \sum_{i=0}^p s_N^*(\mathrm{adj}(\alpha_{i,p}))_*\Theta^{MS}(2)_{(N)}([\ell_1,\ell_2]) 
        \\ 
        & = \sum_{i=0}^p s_N^*(\mathrm{adj}(\alpha_{i,p}))^*(\mathrm{adj}(\alpha_{i,p}))_*\Theta^{MS}(2)_{(N)}([\ell_1,\ell_2]) 
    \end{align}
    Now observe that for matrix $M:\G_m^2\to \G_m^2$, the correspondence $M^*M_*$ is precisely consists of $(x, x\cdot  \ker M)$ where $\cdot \ker M$ means translation by this subgroup. Thus, since the kernel of the various $\alpha_{i,p}$ are precisely the $p+1$ lines of order $p$ torsion, we have the equality of correspondences
    \[
    \sum_{i=0}^p \alpha_{i,p}^*(\alpha_{i,p})_* = (x, px + x\cdot \ker\,[p]) \subset \mathbb{G}_m^2
    \]
    since the union of all $p$-torsion lines is precisely the $p$-torsion, except the identity is counted once per line. These correspondences also all preserve the open set $U_{\Gamma_1(N)}$, and hence can be applied to the values of the spread-out cocycle $\Theta^{MS}(2)_{(N)}$. Since $\Theta^{MS}(2)_{(N)}([\ell_1,\ell_2])$ is $[p]_*$-invariant, we finally obtain 
    \begin{align}
        T_p^* s_N^*\Theta^{MS}(2)_{(N)}([\ell_1,\ell_2]) &= s_N^*(p+[p]^*)\Theta^{MS}(2)_{(N)}([\ell_1,\ell_2]) \\
        & = ps_N^* \langle p \rangle s_N^*\Theta^{MS}(2)_{(N)}([\ell_1,\ell_2]) 
    \end{align}
    so that $T_p^*\Pi_n^\circ = (p+\langle p \rangle)\Pi_n^\circ$, as desired.

    Finally, for the operator $U_p$, $p|N$, we write:
    \begin{align}
        U_p s_N^*\Theta^{MS}(2)_{(N)}([\ell_1,\ell_2]) &= \sum_{i=0}^{p-1} s_N^*\begin{pmatrix}1 & -Ni \\ & p\end{pmatrix}_*\Theta^{MS}(2)_{(N)}([\ell_1,\ell_2]) \\
        & = \sum_{i=0}^{p-1} s_N^*\begin{pmatrix}1 & -\frac{N}{p}i \\ & 1\end{pmatrix}_*\begin{pmatrix}1& \\ & p \end{pmatrix}_*\Theta^{MS}(2)_{(N)}([\ell_1,\ell_2]) \\
        & = \sum_{i=0}^{p-1} (\zeta_p^i, \zeta_N)^*\begin{pmatrix}1& \\ & p \end{pmatrix}_*\Theta^{MS}(2)_{(N)}([\ell_1,\ell_2]) \\
        & = s_N^* \begin{pmatrix}p& \\ & p \end{pmatrix}_*\Theta^{MS}(2)_{(N)}([\ell_1,\ell_2]) \\
        & = s_N^* \Theta^{MS}(2)_{(N)}([\ell_1,\ell_2]) 
    \end{align}
    from which $U_p=1$ on $\Pi_N^\circ$ follows.
    \end{proof}

    We also furnish a new proof of the $N$-integrality of $\Pi_N^\circ$ (or $(\Pi_N^\circ)_+$, depending on $N$) when restricted to the homology of the compact curve (compare \cite[Lemma 4.2.7]{SV}, \cite[Lemma 3.3.11]{FK}). For this, we make the observation that the homology of the closed modular curve $X_1(N)$ is a submodule 
    \[
     H_1(X_1(N)) \hookrightarrow H_1(X_1(N),C_1(N))
    \]
    described in terms of the Steinberg module as being generated by symbols of the form $\gamma_1[1:0]\wedge \gamma_2[1:0]$ for $\gamma_1,\gamma_2\in \Gamma_1(N)$, i.e. geodesic paths between cusps in the orbit of the zero cusp (or, equivalently, any given fixed cusp).
    
    \begin{thm}
        The restriction of $\Pi_N^\circ$ (respectively $(\Pi_N^\circ)_+$, when $N$ has distinct prime divisors) to $H_1(X_1(N))$ takes values in the submodule
        \[
        H^2(\Z[\mu_N],\Z[\tfrac{1}{2}](2)) \hookrightarrow H^2(\Z[\mu_N, N^{-1}],\Z[\tfrac{1}{2}](2)) 
        \]
        (respectively with $+$ parts).
    \end{thm}
    \begin{proof}
        Again, we will write the proof without $+$ projections, for brevity. By Proposition \ref{prop:loc}, the $H^2(\Z[\mu_N],\Z[\tfrac{1}{2}](2))$ is the submodule of $H^2(\Q(\mu_N),\Z[\tfrac{1}{2}](2))$ on which the various tame symbols
        \[
        H^2(\Q(\mu_N),\Z[\tfrac{1}{2}](2)) \to H^1(\Z[\mu_N]/\mathfrak{p}, \Z[\tfrac{1}{2}](2))
        \]
        vanish, for \emph{all} primes $\mathfrak{p}$. Let 
        \[
        L=\begin{pmatrix} \ell_1 & \ell_2 \end{pmatrix}
        \]
        be the pushforward matrix associated to a Steinberg symbol in $\St(2)^\circ$. We have the following pullback/pushforward functoriality of tame symbols \cite{Lev2}:
        \begin{equation}
        \begin{tikzcd}
            H^2(U_{\Gamma_1(N)},\Z[\tfrac{1}{2}](2)) \arrow[d,"L_*"] \arrow[r, "\partial"] &\bigoplus\limits_{D} H^1(D, \Z[\tfrac{1}{2}](2))\arrow[d,"L_*"] \\
            H^2(U_{\Gamma_1(N)},\Z[\tfrac{1}{2}](2)) \arrow[d,"s_N^*"] \arrow[r, "\partial"] &\bigoplus\limits_{D} H^1(D, \Z[\tfrac{1}{2}](2))\arrow[d,"s_N^*"] \\
            H^2(\Q(\zeta_N),\Z[\tfrac{1}{2}](2)) \arrow[r, "\partial"] &\bigoplus\limits_{\mathfrak{p}} H^1(\Z[\zeta_N]/\mathfrak{p}, \Z[\tfrac{1}{2}](2))
        \end{tikzcd}
        \end{equation}
        where the direct sums range over irreducible components of the codimension-$1$ locus $\mathbb{G}_m-U_{\Gamma_1(N)}$ and closed points of $\mathrm{Spec}\,\Z[\zeta_N]$ respectively, indexed by convention and familiarity by ``$D$'' for divisor (in $\mathbb{G}_m^2$) for the geometrically-flavored $U_{\Gamma_1(N)}$, and by the associated prime ideal $\mathfrak{p}$ for the cyclotomic number ring. In the lower right vertical arrow, $s_N^*$ is zero by convention if $s_N$ fails to properly intersect $D$. 

        Now, the value $\Pi_N^\circ([\ell_1]\wedge [\ell_2])$ is the image in the bottom left group of $\{1-z_1,1-z_2\}$ in the top left group, whose tame symbol is
        \[
        \{1-z_2\}_{z_1=0}-\{1-z_1\}_{z_2=0}.
        \]
        Let the coordinates of (integral generators of) the lines $\ell_1,\ell_2$ be considered as maps $(\ell_1),(\ell_2):\G_m\mapsto \G_m^2$. Then commutativity of the above diagram implies that 
        \begin{align}
        \partial \Pi_N^\circ([\ell_1]\wedge [\ell_2]) & = s_N^*L_* \{1-z_2\}_{z_1=0}-\{1-z_1\}_{z_2=0} \\
        & = (1,\zeta_N)^* [(\ell_1)_*\{1-z\} - (\ell_2)_*\{1-z\}]
        \end{align}
        which vanishes when $\ell_1=\gamma\ell_2$ for $\gamma \in \Gamma_1(N)$, as in that case $(\ell_1) = \gamma \circ (\ell_2)$ as maps, and we have 
        \[
        (1,\zeta_N)^*\gamma_* = (\gamma^{-1}(1,\zeta_N))^*=(1,\zeta_N)^*
        \]
        as maps.
    \end{proof}

    \begin{rem}
        A certain level compatibility for these Eisenstein cocycles was proven relative to more restrictive sets of cusps in \cite[Theorem 1.1]{LcW}, and asked whether a version existed relative to larger sets of cusps. We expect that the methods of that article applied in our formalism furnish a proof of this extended level compatibility. We do not go into the details as the main purpose of the authors of that article was to prove Hecke equivariance results, which we have established by other means.
    \end{rem}
    \subsubsection{Distributions over torsion sections}

    Besides pulling back by torsion sections in a ``good locus,'' it is actually possible to obtain (motivic or de Rham) cocycles valued in distributions over all torsion sections; analogous cocycles have been constructed (in related but distinct settings) in \cite{BPPS} and \cite{RX}. In this section, we construct such distributions in the motivic setting, but only for $n=2$.

    Write $\D(\mathbb{Z}_p^n,M)$ for the module of $M$-valued locally constant distributions on $\mathbb{Z}_p^n$, i.e. the dual of the module of $M$-valued locally constant and compactly supported functions $\cS(\mathbb{Z}_p^n,M)$ (``Schwartz'' functions). This latter module is generated by $M$-valued indicator functions on subsets of the form $a+L$, where $L\subset \mathbb{Z}_p^n$ is a full-rank $\mathbb{Z}_p$-lattice. If $M$ is a left $\mathrm{GL}_n(\mathbb{Z})$-module, a distribution $\mu\in\D(\mathbb{Z}^n_p,M)$ takes a left $\mathrm{GL}_n(\mathbb{Z})$-action via
    \[
    (\gamma \mu)(\varphi) := \gamma \cdot\mu( \varphi\circ \gamma).
    \]
    where $\gamma$ acts via its natural left action on $\mathbb{Z}^n$, and thus by the above right pullback action of postcomposition with test functions $\varphi\in \mathcal{S}(\Z^n)$. We also write 
    \[
    \D(\mathbb{Z}^n_p,M)^{(0)} \subset \D(\mathbb{Z}^n_p,M)
    \]
    for the scalar-invariant distributions, i.e. those for which the set $U$ has the same measure as the set $pU$. We write also the periodic function $\exp(x):=e^{2\pi i x}$ from $\Q_p/\Z_p$ to the $p$-power roots of unity, and $e_1,e_2$ for the lines corresponding to the standard basis of $\Z^2$. Then we have:

    \begin{prop}
        There is a $\GL_2(\Z)$-equivariant modular symbol
        \[
        \Theta_{\D_p}(2): \St(2)\to \D(\Z_p^2, H^2(\Z[\zeta_{p^\infty}, p^{-1}], \Z[\tfrac{1}{2}](2)))
        \]
        which sends
        \[
        [e_1]\wedge [e_2] \mapsto \mu_e
        \]
        where $\mu_e$ is the distribution
        \[
        1_{(\alpha, \beta)+p^k\Z_p^2}\mapsto  \{1-\exp(\alpha/p^k), 1-\exp(\beta/p^k)\} 
        \]
        for $\alpha,\beta$ both not divisible by $p^k$, and otherwise sends
        \[
        1_{(0, \beta)+p^k\Z_p^2} \mapsto \{p^{-k}, 1-\exp(\beta/p^k)\},
        \]
        symmetrically 
        \[
        1_{(\alpha,0)+p^k\Z_p^2} \mapsto \{1-\exp(\alpha/p^k), p^{-k}\},
        \]      
        and 
        \[
        1_{(0, 0)+p^k\Z_p^2} \mapsto \{p^{-k}, p^{-k}\} = 0
        \]
        Here, the section $(\alpha/p^k,\beta/p^k)$ means the pair of roots of unity $(\exp(2\pi i \alpha/p^k),\exp(2\pi i \beta/p^k))$
    \end{prop}
    \begin{proof}
        The distribution relations for $\mu_e$ follow from the cyclotomic distribution relations
        \[
            \prod_{p x' = x} 1-\zeta_{p^{k+1}}^{x'} = 1-\zeta_{p^k}^x
        \]
        for $x\in \Z/p^k\Z$, $x'\in \Z/p^{k+1}\Z$, and 
        \[
            \prod_{i=0}^{p-1} 1-\zeta_p^i = p.
        \]
        Moving this around by the $\GL_2(\Z)$-action gives the distribution relations on arbitrary symbols.
        
        By the standard theory of modular symbols, it remains therefore only to verify that $\mu_e$ satisfies the two Manin relations
        \[
        \mu_e + \begin{pmatrix}1&\\ &1 \end{pmatrix}\mu_e = 0,\, \mu_e + \begin{pmatrix}&1\\ 1&1 \end{pmatrix} \mu_e + \begin{pmatrix}1&1\\1&\end{pmatrix}\mu_e=0
        \]
        The former is obvious by the fact that Steinberg symbols are alternating with $\Z[\frac{1}{2}]$-coefficients. The latter can be checked after evaluating at $1_{(\alpha, \beta)+p^k\Z_p^2}$: in the case that $\alpha,\beta\ne 0$ and $\alpha \ne \beta$, none of the associated pullbacks have coordinate zero, so follows from pulling back the corresponding relation for $\Theta^{MS}(\D_p)$ by $(\alpha,\beta)$. 
        
        Otherwise, if $\alpha=0$, the relation is
        \[
        \{p^{-k}, 1-\exp(\beta/p^k)\} + \{1-\exp(\beta/p^k), p^{-k}\}+\{p^{-k},p^{-k}\}
        \]
        which is plainly zero by alternation (with $2$ inverted). The cases $\beta=0$, $\alpha=\beta$, are similar visibly zero (and in fact follow by $\GL_2(\Q)$-equivariance).
    \end{proof}

    We remark that by restricting to open sets corresponding to $p^k$-torsion points with respect to the fixed lattice $\Z^2$, we can also obtain finite-level versions valued in the groups $K_2^M(\Z[1/p^k,p^{-1}])\otimes \Z[\tfrac{1}{2}]$. Specializing further at a $\Gamma_1(p^k)$-fixed point, the results of this section can be used to define a map from homology relative \emph{all} cusps of $X_1(p^k)$ to $K$-theory; however, it no longer has nice Hecke equivariance properties. Thus, it is unclear whether it is of interest to the Sharifi conjectures, though their dynamics under the $U_p,U_p^*$ operators may be interesting. 
    
    In particular, observe that the specialized values at symbols bordering $\infty$ will be of the form $\{ p^{-k}, 1-\zeta_{p^k}^\bullet\}$. Previous literature considered the extension of $\Pi_N^\circ$ to cusps over $\infty$ by simply extending by zero; we have no concrete evidence that our extension is more useful, except that the values above are suggested by certain distribution relations.

    \section{The de Rham cocycles and applications} \label{section:derhamtheta}

    In this section, we pass from our motivic/$K$-theory-valued class to one valued in differential forms, from which we will be able to extract actual numbers (or distributions) related to $L$-values of totally real fields.

    There is a regulator map \cite[\S 2.1.5]{LW}
    \[
    (d\log)^{\wedge n}: H^n(X,\mathbb{Z}(n)) \to \Omega^n_X, \{u_1,\ldots, u_n\} \mapsto d\log u_1 \wedge \ldots \wedge d\log u_n
    \]
    which exists for any scheme $X$, functorial for pullbacks and pushforwards, to turn our motivic-valued cocycles into differential form-valued cocycles. We deduce from our motivic cocycle
    \[
    (d\log)^{\wedge n}_* \Theta(n) \in H^{n-1}(\GL_n(\Q), (\Omega^n_{k(\G_m^n)})^{(0)}/\langle (z_1\ldots z_n)^{-1} dz_1\wedge \ldots \wedge dz_n\rangle)
    \]
    which can be represented by a homogeneous cocycle which sends
    \[
    (\gamma_0,\ldots, \gamma_{n-1}) \mapsto \begin{pmatrix} \gamma_0 e_1 & \ldots \gamma_{n-1}e_1\end{pmatrix}_* \frac{(-1)^n}{(1-z_1)\ldots (1-z_n)}dz_1\wedge \ldots \wedge dz_n
    \]
    whenever the first columns of $\gamma_0,\ldots, \gamma_{n-1}$ are independent (or an analogous condition if we replace $e_1$ with any rational ray). If we write $z_i=\exp(2\pi i t_i)$ for $1\le i \le n$, then we have
    \[
    dz_1\wedge \ldots \wedge dz_n = z_1z_2\ldots z_n dt_1\wedge \ldots \wedge dt_n.
    \]
    The $n$-form $dt_1\wedge \ldots \wedge dt_n$ transforms by the character $\det$ under pullback by $\GL_n(\Q)$, so 
    \[
    \omega \mapsto \frac{\omega}{dt_1\wedge \ldots \wedge dt_n}
    \]
    furnishes a $\GL_n(\Q)$-equivariant isomorphism
    \[
    \iota_t: \Omega^n_{k(\G_m^n)}/\langle dz_1\wedge \ldots \wedge dz_n\rangle \to (\mathcal{M}_{\G_m^n}/\langle 1\rangle )(-\det)
    \]
    by the push-pull formula
    \[
    \gamma_*(f\, dt_1\wedge \ldots \wedge dt_n) = \frac{1}{\det \gamma} \gamma_*(f\, \gamma^* dt_1\wedge \ldots \wedge dt_n)=\frac{1}{\det \gamma} \gamma_*(f)\,  dt_1\wedge \ldots \wedge dt_n
    \]
    Here, $\mathcal{M}_{\G_m^n}(-\det)$ denotes the meromorphic functions with their pushforward action twisted by the character $\det^{-1}$. We then define our differential Eisenstein cocycle as
    \[
        \Theta^{dR}(n) := (\iota_t)_*\circ (d\log)^{\wedge n}_* \Theta(n) \in H^{n-1}(\GL_n(\Q), (\mathcal{M}_{\G_m^n}/\langle 1\rangle (-\det))^{(0)})
    \]
    which can be represented by a homogeneous cocycle sending
    \[
    (\gamma_0,\ldots, \gamma_{n-1}) \mapsto (\det M)^{-1} M_* \frac{z_1z_2\ldots z_n}{(z_1-1)\ldots (z_n-1)}
    \]
    for \[
    M=\begin{pmatrix} \gamma_0 e_1 & \ldots \gamma_{n-1}e_1\end{pmatrix}
    \]
    whenever the first columns of $\gamma_0,\ldots, \gamma_{n-1}$ are independent. Here, the superscript $(0)$ means trace-fixed \emph{with} the determinant twist; i.e. $[a]_*f = a^nf$ for $a\in \mathbb{N}$.
    
    In fact, if we restrict to $\GL_n(\Z[S^{-1}])$ for some set of primes $S$, Corollary \ref{cor:sullivan} and \eqref{eq:eulerlift} tell us that associated to a lift $\phi$ of the Euler class we obtain
    \[
    \Theta^{dR,\phi}(n) \in d_{n,S}^{-1} \cdot H^{n-1}(\GL_n(\Z[S^{-1}]), \mathcal{M}_{\G_m^n}(-\det)^{(0)}). 
    \]
    For our application to $L$-values of totally real fields, we will be interested in the case $S=\emptyset$, for which $d_n$ is the denominator of the halved Bernoulli number $\frac{B_n}{2}$. For this reason, we will also restrict to the subgroup $\GL_n(\Z)$ in all that follows, though many of the same statements also hold with essentially the same proofs for $S$-arithmetic groups. 

    It will be useful to introduce the shorthand
    \[
    \tilde{r}:=\iota_t\circ d\log^{\wedge n}\circ \tilde{\rho}:\tilde{C}(n) \to \mathcal{M}_{\G_m^n}(-\det)^{(0)}
    \]
    for the realization map from rational spherical chains to meromorphic functions.

    \subsubsection{Stabilizations}

    To obtain relations to totally real $L$-values, we will need certain combinations of torsion-section translates of our de Rham cocycles. In this section, we introduce the requisite notation, and make some basic observations about the resulting stabilizations when some degree-zero properties are satisfied.

    Over $\overline{\Q}$, the torsion of $\G_m^n$ is the $n$th power of the roots of unity $\mu_\infty^n$; given $x\in \mu_\infty^n$, we will write 
    \[
    t_x: \G_m^n\mapsto \G_m^n
    \]
    for the translation map which is multplication by the \emph{inverse} $x^{-1}$. More generally, via the complex uniformization $\G_m^n \cong \C^n/\Z^n$, we can identify torsion sections with elements of $\Q^n/\Z^n$. Then to any Schwartz function $\varphi \in \cS(\Q^n/\Z^n)$ with period $\Z^n$ (in other words, finitely supported functions on $\Q^n/\Z^n$), we define an associated pullback operator on $\mathcal{M}_{\G_m^n}(-\det)$ by 
    \[
    t_{\varphi}^* := \sum_{x\in \Q^n/\Z^n} \varphi(x) \cdot t_x^*.
    \]
    The group $\GL_n(\Z)$ acts on $\Q^n/\Z^n$ by the standard left action, and for this action, we have the equivariance property
    \[
    \gamma_* t^*_\varphi  f = t^*_{\varphi \circ \gamma^{-1}} \gamma_* f
    \]
    for any $f\in \mathcal{M}_{\G_m^n}$.
    
    We observe that if $\varphi$ is $\Gamma$-invariant, for an arithmetic subgroup $\Gamma\subset \GL_n(\Z)$, then $t_\varphi^*\Theta^{dR}(n)$ is a cohomology class for $\Gamma$; this applies to all variants of this construction we have seen as well, and to the cocycle representatives we have written down. In particular, associated to a $\Delta$-extension $E$ and an Euler class lift $\phi$, we have a cocycle representative
    \[
        t_\varphi^*\theta^{dR}_{E,\phi}(n)
    \]
    If $\varphi$ sums to zero, $t_\varphi^*$ annihilates the Euler class, so $\phi$ can be taken to be zero and omitted from the notation. Furthermore, we recall that the ambiguity in the choice of a $\Delta$-extension is only on degenerate ``wedge''-like simplices $\Delta^E(r_1,\ldots, r_n)$ with a linear dependence among the rays. For such a simplex, suppose its rays span a proper subspace $V\subset \Q^n$. Notice that the realization of the standard wedge $\Q e_1 \oplus \ldots \oplus \Q e_k \oplus \Q^+ e_{k+1}\oplus \ldots \Q^+ e_{n}$ is
    \[
        \frac{z_{k+1}\ldots z_{k}}{(z_{k+1}-1)\ldots(z_n-1)},
    \]
    which is unaffected by translation by elements in $V$, meaning its $\varphi$-weighted sum over any $V$-coset is just the total mass of $\varphi$ on that coset. By $\GL_n(\Q)$-equivariance, this implies that in general, if $\varphi$ sums to zero considered as a function on $(\Q^n/\Z^n)/V(\Q/\Z)$, then $t_\varphi^* \tilde{r}_* \Delta^E(r_1,\ldots, r_n) = 0$ as well regardless of the choice of $\Delta$-extension.

    For instance, if $\varphi$ sums to zero on \emph{every} subspace spanned by $\Gamma$-orbits of $e_1$, then $t_\varphi^*\theta^{dR,E}(n)$ is also independent of $E$ (which can thus be omitted from the notation in this case), and simply vanishes on \emph{all} linearly dependent tuples of rays arising from tuples of matrices in $\Gamma$. This phemomenon, wherein realizations of ``wedge'' classes disappear upon stabilization, will be important in our application to Shintani domains/$L$-values of totally real fields. It also arises in the following comparison:

    \begin{ex} \label{ex:bcg}
        The analytic ``cocycle multiplicatif'' of \cite{BCGV} is a homogeneous $(n-1)$-cocycle 
        \[
        \mathbf{S}_{\mathrm{mult}}^*[\varphi_f]:\Gamma^n\to \mathcal{M}_{\G_m^n},
        \]
        for the \emph{right} pullback action on the coefficients, where $\varphi\in \cS(\Q^n/\Z^n)$ is $\Gamma$-invariant for some arithmetic subgroup $\Gamma\subset \SL_n(\Z)$, and sums to zero on any subspace containing the line corresponding to the first standard basis element $e_1$. 
        
        In \cite[Proposition 8.8]{BCGV}, the value of $\mathbf{S}_{\mathrm{mult}}^*[\varphi]$ on a tuple $(\gamma_0,\ldots, \gamma_{n-1})$ is computed as
        \[
            \sum_{v\in \Q^n/\Z^n}\varphi(v)\sum_{\substack{\xi\in \Q^n/\Z^n\\ h \xi \equiv v \pmod{\Z^n}}} 
            \frac{d\ell_1\wedge \ldots \wedge d\ell_n}{(e^{2\pi i (\ell_1-\xi_1)}-1)\ldots (e^{2\pi i (\ell_n-\xi_n)}-1)}
        \]
        whenever 
        \[
        h:=\begin{pmatrix} \gamma_0^{-1}e_1 & \ldots & \gamma_{n-1}^{-1}e_1\end{pmatrix}
        \]
        has linearly independent columns and $h^* \ell_j = e_j^\vee$. We find that upon restriction of our cocycle to $\Gamma$, we have
        \begin{equation} \label{eq:bcgcomp}
            \mathbf{S}_{\mathrm{mult}}^*[\varphi](\gamma_0^{-1},\ldots, \gamma_{n-1}^{-1})= t_{\varphi_f}^* \theta^{dR}(n)(\gamma_0,\ldots, \gamma_{n-1})
        \end{equation}
        where here we use 
        \[
        \frac{1}{x-1}= \frac{x}{x-1} -1
        \]
        and then the degree-zero property of $\varphi$ to see that the contributions of the constants $-1$ cancel out. As remarked above, our cocycle representative $t_{\varphi}^* \theta^{dR}(n)(\gamma_0,\ldots, \gamma_{n-1})$ equals zero independently of a choice of $\Delta$-extension when $h$ is not full rank. Note that the presence of inverses makes sense, since our cocycle is a \emph{left} cocycle while theirs is a right cocycle. This reflects their systematic use of pullback actions rather than pushforwards; since we are working with subgroups of $\SL_n(\Z)$, the pushforward coincides with the pullback of the inverse. In fact, this comparison holds in more generality without need for such stringent degree-zero assumptions: \cite[Theorem 1.7]{BCGV} identifies $\mathbf{S}_{\mathrm{mult}}^*[\varphi_f]$ with the image under an edge map of a certain \emph{equivariant polylogarithm} class for $\G_m^n$, up to some controlled ambiguity; we prove the same for our cocycle in Appendix \ref{appendix:a}. 
    \end{ex}

    \subsubsection{$p$-power distributions} \label{section:fourier}

    We now discuss the \emph{specializations} of elements of $\mathcal{M}_{\G_m^n}(-\det)^{(0)}$ at torsion points of $\G_m^n$, and in particular specializations to distributions over all $p$-power torsion at once (which will furnish the link to $p$-adic $L$-functions). A general meromorphic function of course cannot be specialized at arbitrary torsion points, since there may be poles, so to obtain a nice formalism, we will need to stabilize.
    
    For simplicity, we will from this point on fix one prime $p$ and focus on specializations at $p$-power torsion points. Let $A_p\hookrightarrow \mathcal{M}_{\mathbb{G}_m^n}$ be the meromorphic functions on $\G_m$ which are regular on the open $p$-adic disc around $1$. We will write $A_p^{(n)}$ for the functions that $[p]_*f = p^nf$ for all integers $p\ne 0$; one checks easily that $A_p \cap \mathcal{M}_{\mathbb{G}_m^n}(-\det)^{(0)}\subset A_p^{(n)}$. 

    \begin{ex}
    As an example, for any torsion point $x$ of order $c>1$ prime to $p$ with all $n$ coordinates nonzero, 
    \[
    t_x^* \frac{z_1\ldots z_n}{(1-z_1)\ldots (1-z_n)}\in A_p.
    \]
    More generally, if $M\in M_n(\Z)$ is a matrix, and every torsion section in $M^{-1}x$ has all coordinates nonzero, then
    \[
    t_x^* M_*\frac{z_1\ldots z_n}{(1-z_1)\ldots (1-z_n)}\in A_p.
    \]
    \end{ex}

    Elements in $A_p$ can be pulled back by all $p$-power torsion by definition; moreover, we have the following formalism in terms of Schwartz functions on $\Q_p^n$:

    \begin{prop} \label{prop:specialization}
        For $f\in A_p^{(n)}$, the rule
        \[
        1((x_1,\ldots, x_n) + p^k\Z_p^n) \mapsto p^{-kn}(e^{2\pi i x_1/p^k},\ldots, e^{2\pi i x_n/p^k})^* f,
        \]
        extended linearly to all Schwartz functions, gives a $\GL_n(\Z)$-equivariant map $\kappa: A_p^{(n)}\to \D^{(n)}(\Q_p^n,\Q_p)$. Here, the $\GL_n(\Z)$ action on distributions is as in the motivic setting; i.e. $(\gamma\cdot \mu)(\varphi) := \mu(\varphi\circ \gamma)$, and superscript $(n)$ on distributions means that $\mu(\varphi\circ [p]) = p^n\mu(\varphi)$.
    \end{prop}
    \begin{proof}
        The well-definedness of the stated formula amounts to the distribution relations
        \[
        p^n x^*f = \sum_{[p]x'=x} (x')^*
        \]
        for any $p$-power torsion section $x$, which follow by the $\det^{-1}$-twisted $[p]_*$-fixedness of $f$. The $\GL_n(\Z)$-equivariance follows from the computation that
        \[
        \gamma\cdot f := (\det\,\gamma)^{-1} \cdot \gamma_* f
        \]
        is sent by $\kappa$ to
        \begin{align}
            1((x_1,\ldots, x_n)+ p^k\Z_p^n) & \mapsto (\det\,\gamma)^{-1} p^{-kn} (e^{2\pi i x_1/p^k},\ldots, e^{2\pi i x_n/p^k})^* \gamma_* f \\
            & = (\sgn\,\gamma) p^{-kn} [\gamma^{-1}(e^{2\pi i x_1/p^k},\ldots, e^{2\pi i x_n/p^k})]^* f 
        \end{align}
        which is precisely 
        \[
        (\sgn\,\gamma)
        \kappa(f)(1((x_1,\ldots, x_n)+ p^k\Z_p^n) \cdot \gamma).
        \]
        The scaling property is obvious from the formula.
    \end{proof}

    The standard inner product $\langle-,-\rangle$ on $\Q^n$ induces a choice of Fourier transform 
    \[
    \mathscr{F}: \cS(\Q_p^n, \Q_p)\to \cS(\Q_p^n,\Q_p(\mu_{p^\infty})), f\mapsto \widehat{f}(y):= \int_{\Q_p^n}f(x) \omega_p(\langle x, y\rangle)\, dh(x)
    \]
    where $\omega_p:\Q_p/\Z_p\to \mu_{p^\infty}$ is the additive character sending $1/p\mapsto e^{-2\pi i/p}$, and $h(x)$ is the Haar measure for which $h(\Z_p^n)=1$. If we write $f|_\gamma(x):=f(\gamma x)$, then the Fourier transform satisfies the equivariance property (see \cite[Appendix A]{LP})
    \[
    \widehat{f|_\gamma}(x)= |\det \gamma|\cdot \widehat{f}|_{(\gamma^{-1})^T}(x). 
    \]
    for any $\gamma\in \GL_n(\Q)$. If we extend coefficients to $\Q_p(\mu_p^{\infty})$ on the domain, the Fourier transform is an involution.

    We can similarly define a Fourier transform on distributions 
    \[
    \mathscr{F}:\D(\Q_p^n, \Q_p(\mu_{p^\infty}))\to \D(\Q_p^n, \Q_p(\mu_{p^\infty})), \mu\mapsto \widehat{\mu}(f):= \mu(\widehat{f})
    \]
    which similarly satisfies $\widehat{(\gamma \cdot \mu)}(f) = |\det \gamma|^{-1}\cdot ((\gamma^{-1})^T\cdot \widehat{\mu})(f)$, and is itself an involution. 
    
    For the remainder of the article, to keep track of the group action, we will put a hat over the $\D$ (as in $\hat{\D}$) to indicate the group action for which $\mathscr{F}:\D \to \hat{\D}$ is $\GL_n(\Q)$-equivariant. 

    \begin{rem} \label{rem:stardef}
    Restricting attention to equivariance under $\GL_n(\Z)$, if we write $\star: \GL_n(\Z)\to \GL_n(\Z)$ for the inverse transpose automorphism $\gamma\mapsto (\gamma^{-1})^T$, then the above is equivalent to saying the Fourier transform is a map 
    \[
    \D(\Q_p^n, \Q_p(\mu_{p^\infty}))\to (\star^*)\D(\Q_p^n, \Q_p(\mu_{p^\infty})),
    \]
    The $\star$ notation will be useful later in accounting for other dualities.
    \end{rem}
    
    One can compute (e.g., see \cite{Stevens}) that the Fourier transform on distributions restricts to an isomorphism
    \[
    \D^{(n)}(\Q_p^n, \Q_p(\mu_{p^\infty})) \mapsto \D^{(0)}(\Q_p^n, \Q_p(\mu_{p^\infty})).
    \]
    
    Finally, we need the following important observation, due to Katz:
    \begin{prop}
        The map    
        \[
        \widehat{\kappa}:=\mathscr{F}\circ \kappa: A_p^{(n)} \xrightarrow{\sim} \hat{\D}^{(0)}(\Q_p^n, \Q_p(\mu_{p^\infty}))
        \]
        actually maps $A_p^{(n)}$ isomorphically onto the subspace $\hat{\D}^{(0)}(\Q_p^n, \Z_p)$. 
    \end{prop}
    \begin{proof}
        Cartier duality for $\G_m^n$ precisely says that a map with this formula in fact maps $A_p$ isomorphically onto $\hat{\D}^{(0)}(\Q_p^n, \Q_p(\mu_{p^\infty}))$; see \cite[Theorem 1]{Katz}. Then take the appropriate trace-isotypic parts of both sides.
    \end{proof}

    Note that Katz's statement involves \emph{measures}, i.e. the continuous dual of continuous functions, but since locally constant functions on $\Z_p^n$ are dense in continuous ones, there is a canonical isomorphism 
    \[
        \D(\Z_p^n, \Z_p) \cong \mathrm{Measures}(\Z_p^n, \Z_p)
    \]
    given by approximating by Riemann sums; note that it is important here that we take $p$-complete coefficients. We are mostly interested in $p$-adic integration to obtain $p$-adic $L$-functions, so for convenience of exposition, we will use $p$-adic coefficients in all distributions going forward, though many of our classes (for example, ${}_c\Theta^{dR}_{\D_p}(n)$ below) are actually $\Z$-valued when considered only as distributions (i.e. against locally constant test functions).\footnote{Using the $\Z$-valued distributions would be necessary to imitate the ``multiplicative'' construction of \cite{RX2}, as mentioned in the introduction, for example.}

    \begin{rem} \label{rem:imdone}
        We also note that $\widehat{\kappa}$ satisfies an easily-computable (e.g. by the formula given in \cite[Theorem 1]{Katz}) functoriality for pullbacks $\gamma^*$ for $\gamma\in \GL_n(\Z_{(p)})$, given by
        \[
        \widehat{\kappa}(\gamma^* f)(\varphi) = \widehat{\kappa}(f)(\varphi\circ \gamma^{-1}).
        \]
    \end{rem}
    
    \subsection{Shintani generating functions and conical duality} \label{section:cones}

    We review the method of Shintani for obtaining zeta values of a totally real field $F$ of absolute degree $n$ \cite{Shin}, though we formulate things more in the style of \cite{Katz} (with some modernizations of notation); this method is how we will prove interpolation properties connecting our cocycles to $L$-values.
    
    Let $U:=\mathcal{O}_F^\times$, and let $U^+$ be the totally positive units, i.e. units positive at every real place. We will also write $F^+$, $\mathcal{O}_F^+$, etc. for analogous constructions.
    
    A \emph{Shintani decomposition} of $F\otimes \R$ is a collection $\mathscr{S}$ of relatively open\footnote{Meaning, excluding its lower-dimensional conical faces.} rational simplicial cones (of various dimensions), each bounded by linearly independent $F$-rational totally positive rays, such that the $U^+$-orbit of these cones in $\mathscr{S}$ yield a disjoint partition of $(F\otimes \R)^+$: in other words, the union of the cones in such a decomposition constitutes a fundamental (``Shintani'') domain for $U^+$ acting on this orthant. 
    
    If $I\subset F$ is any fractional ideal, viewed as a lattice inside $F\otimes \R$, we will be interested in coordinatizations 
    \[
    \alpha: F\otimes \R \cong \R^n
    \] 
    which restricts to an identification of lattices $\alpha: I\xrightarrow{\sim} \Z^n$.
    Given a cone (as above) $C$ with bounding rays generated by integral points $x_1,x_2,\ldots, x_k$, write
    \[
    R_C: \Z^n \cap \left\{c_1x_1+\ldots + c_kx_k, \forall i, c_i\in \Q, 0<c_i\le 1\right\}
    \]
    for the lattice points in the ``unit cube'' spanned by the $x_i$. We also write $z_{x}$ for the monomial $z_1^{x_1}\ldots z_n^{x_n}$, for any $x\in I\cong \Z^n$.
    
    Then if $\psi:I\to \overline{\Q}^\times$ is any finite order additive character (which we will identify via $\alpha$ with a character of $\Z^n$), identified with a character of $\Z^n$, associated to this data we have a rational ``Shintani generating function''
        \begin{equation} \label{eq:shinfunc}
        f_{\mathscr{S}}(\psi)(z_1,\ldots, z_n):= \sum_{C\in \mathscr{S}} f_C(\psi)(z_1,\ldots, z_n) := \sum_{C\in \mathscr{S}} \sum_{y\in R_C} \frac{\psi(y) z_y}{\prod_{x}  (1-\psi(x)z_x)}
        \end{equation}
    where the product in the denominator is over integral generators for the rays bounding the cone $C$. Note that here, $\psi$ can be identified with a torsion point in $\G_m^n$; namely, the section $(\zeta_1,\ldots, \zeta_n)$ if $\psi(x_1,\ldots, x_n)=\zeta_1^{x_1}\ldots \zeta_n^{x_n}$. We will also write simply $f_C$ for $f_C(1)$.
    
    The rational function $f_C(\psi)$ is a way to make sense of the naive infinite sum
    \begin{equation} \label{eq:naivegen}
        \sum_{x\in C} \psi(x) z_x
    \end{equation}
    whose meaning is otherwise ambiguous. Note that if $x\in \G_m^n[c]$ is a torsion point, then
    \[
    t_x^* f_C(\xi) = f_C(\psi_x\cdot \xi)
    \]
    for the character $\psi_x:\Z^n\to \overline{\Q}^\times$ given by $(i_1,\ldots, i_n)\mapsto (x_1^{-i_1},\ldots, x_n^{-i_n})$,
    i.e. we are implicitly identifying $\Z^n$ with the character lattice of $\G_m^n$.
    
    On these Shintani functions, we also define the differential operator sending
    \[
    D_{\mathbf{N}}: z_x \mapsto \mathbf{N}(x) z_x
    \]
    for any $x\in I$, where we write $\mathbf{N}$ for the field norm of $F$. The main result of Shintani's approach to $L$-functions is then the following:

    \begin{thm} \emph{(\cite[Prop. 1]{Shin}, formulation from \cite[Theorem 2]{Katz})} \label{thm:shin}
        One can define a meromorphic function by analytically continuing the Dirichlet series
        \[\zeta_{I}(\psi, s) := \sum_{[\lambda]\in I^+/U^+} \psi(\lambda)\mathbf{N}(\lambda)^{-s}.\]
        Then for all $k\ge 0$, 
        \[
        \zeta_I(\psi,-k)= D_{\mathbf{N}}^k(f_{\mathscr{S}}(\psi))(1,\ldots, 1)
        \]
        for any character $\psi$ nontrivial on each ray bounding any $C\in \mathscr{S}$.
    \end{thm}    
    The proof of this theorem is an analytic trick generalizing the proof the functional equation of the Riemann zeta function using an integral representation; what it actually shows is that 
    \[
    D_{\mathbf{N}}^k(f_{C}(\psi))(1,\ldots, 1)=\sum_{\lambda\in C}\psi(\lambda)\mathbf{N}(\lambda)^{-s},
    \]
    so that the fact that the union of $\mathscr{S}$ is a fundamental domain for $I^+/U^+$ implies the result. 

    We also note the following relation to $p$-adic integration of the differential operator $D_\mathbf{N}^k$, noted by Katz \cite[Theorem 1]{Katz}, which will be useful later:

    \begin{prop} \label{prop:katz}
        If $\psi:I^+/U^+\to \overline{\Q_p}^\times$ is any additive character, we identify $\psi$ with a function $\tilde{\psi}$ on $\Z_p$ via $\psi=\tilde{\psi}\circ \mathbf{N}\circ \alpha^{-1}$. Then for any $k\ge 0$, and any $f\in A_p$, we have
        \[
        \int_{I_p} \psi(t) \mathbf{N}(t)^k \, d\widehat{\kappa}(f)(t) = (D^k_{\mathbf{N}}f)(\psi)
        \]
        where we here identify $\psi$, as before, with a $p$-power torsion point of $\G_m^n$, and $I_p:= I\otimes \Z_p$.
    \end{prop}

    Note that replacing $\psi$ in all of the above by \emph{any} locally constant function in the $p$-adic topology, we obtain an analogous result by taking the associated linear combination of additive characters corresponding to the Fourier transform of that function. We will use this more general version without comment.
    
    Observe that the functions $f_C(\psi)$ lie in the image of $\Theta^{dR}(n)$ when $\psi=1$; for general $\psi$, they are linear combinations of torsion translates of functions in the image. The data of a top-dimensional rational simplicial cone in $\R^n$ is equivalent to a the data of its simplex of intersection with $S^{n-1}$, i.e. the generators of $\tilde{C}(n)$. One may naively then expect that for an acyclic simplex \emph{qua} cone $C$, our de Rham realization of $[C]$ as a spherical chain\footnote{To be strict, here we should write $[\overline{C \cap S^{n-1}}]$ since $C$ is a relatively open cone and we want a closed spherical chain; however, since the open/closed distinction does not exist within the chain complex, and the cone/simplex identification is very simple, for convenience we simply write $[C]$.} is the Shintani generating function for $C$; i.e., that $
    \widetilde{r}([C]) = f_C(1)\in \mathcal{M}_{\G_m^n}$.
    
    However, this is \emph{not} the case: rather, if $C=\Delta(r_1,\ldots, r_n)$ is an acyclic, positively oriented simplex, then 
    \[
        f_C(1) = \mathrm{adj} \begin{pmatrix} r_1 & \ldots & r_n \end{pmatrix}^T_* \frac{z_1\ldots z_n}{(1-z_1)\ldots (1-z_n)},
    \]    
    where $\mathrm{adj}\,\gamma:= |\det \gamma| \gamma^{-1}$ is the \emph{adjugate} matrix of $\gamma$, while we recall that
    \[
        \widetilde{r}([C]) = (\mathrm{adj}\, \begin{pmatrix} r_1 & \ldots & r_n \end{pmatrix}^{-1})_* \frac{z_1\ldots z_n}{(1-z_1)\ldots (1-z_n)},
    \]
    and we compute that this expression in fact equal to the Shintani generating function $f_{C^\vee}(1)$ of the \emph{dual} cone $C^\vee$, whose bounding rays are $r_1^\vee,\ldots, r_n^\vee$, defined by
    \[
    \langle r_i^\vee,r_j\rangle = \begin{cases}
        \R_+, &i=j\\
        0, &i\ne j
        \end{cases}.
    \]
    More intrinsically, the dual $C^\vee$ of \emph{any} cone $C$ can be defined as the locus pairing positively with $C$ under $\langle-,-\rangle$; for simplicial cones, one can check this coincides with the definition in terms of bounding rays.

    \begin{rem}\label{rem:conicalduality}
    This ``conical duality'' between the two constructions is observed also in the following phenomenon: as we have seen previously, the realization of a wedge, as defined in Lemma \ref{lem:stquo}, is generally nonzero under $\widetilde{r}$, while it is easy to compute that $f_C=0$ for any wedge $C$, since 
    \[
    f_{\R}(z) = f_{\R_+}(z)+f_{\R_{\le 0}}(z) = \frac{z}{1-z} + \frac{1}{1-z^{-1}} =0 ,
    \]
    cf. \cite{SH}). Conversely, by construction, $\widetilde{r}$ ignores any kind of degenerate lower-dimensional simplex, while $f_C$ makes sense and is generally nonzero for lower-dimensional cones $C$. Indeed, the roles played by ``wedge'' and ``lower-dimensional cone'' are swapped; they are precisely the conical duals to each other. For a concrete example, observe that 
    \[
    \R_+\oplus \R\subset \R^2
    \]
    has realization $z_1/(1-z_1)$, which is the same as the Shintani generating function for its conical dual $\R_+\oplus \{0\} \subset \R$. Note also that under this duality, the fundamental class $[S^{n-1}]$ corresponds to the degenerate cone $\{0\}$; the realization $\widetilde{r}$, respectively the association of the Shintani generating function, send these to the constant function $1$. 
    \end{rem}
    
    \begin{rem}
        A consequence of the previous remark is in the \emph{additivity} of $\widetilde{r}$, as opposed to the Shintani generating functions: when we have a decomposition of top-dimensional simplices/cones 
        \[
        [C_1]+[C_2] = [C_3],
        \]
        we have proven previously that $\widetilde{r}([C_1])+\widetilde{r}([C_2]) = \widetilde{r}([C_3])$. However, it is not \emph{quite} the case that 
        \[
        f_{C_1}+f_{C_2}=f_{C_3}.
        \]
        Indeed, from the ``naive'' generating function definition \eqref{eq:naivegen}, one sees that this \emph{almost} is obvious from definition, except that the terms coming from lower-dimensional cones in the intersection $\overline{C_1}\cap \overline{C}_2$ will appear on the right-hand side, but not on the left-hand side: thus, this ``obvious'' additivity needs to be corrected by the intervention of lower-dimensional conical faces. This phenomenon poses a delicate problem in the construction of ``true'' Shintani domains; see discussion further below. Meanwhile, the additivity coming from $\widetilde{r}$ is less obvious on the level of expanding out infinite monomial sums (though it can be deduced by canceling out generating functions of wedges), but holds on the nose. In the terminology of \cite{GP}, the latter is ``$N$-additivity,'' as opposed to the naive ``$M$-additivity'' needing correction from lower-dimensional faces. The duality between these types of additivity is well-known in convexity theory and linear programming; see, for example, \cite[Chapter 1, Problem 3]{Bv}. 
            \end{rem}

        The above two remarks can be viewed more structurally in the following way: let $\mathcal{K}_\Q$ be the module of functions on $\R^n$ generated by indicator functions of relatively open rational simplicial cones, and let and $\mathcal{L}_\Q$ be those generated by indicator functions of (relatively open) wedges, as in \cite{SH} and \cite{CDG}. Notice then that the Shintani generating function can be viewed as an association\footnote{This is just a restricted version of the Solomon-Hu pairing defined in \cite{SH}.}
        \[
        f_\bullet:\mathcal{K}_\Q/\mathcal{L}_\Q \to \mathcal{M}_{\G_m^n}, 1_{C}\mapsto f_{C}
        \]
        since the Shintani generating function of a wedge is zero; from this viewpoint, a Shintani decomposition $\mathscr{S}$ is simply the corresponding linear combination of indicator functions of its cones. This association satisfies the $\GL_n(\Q)$-equivariance property 
        \[
            f_{\psi \circ \gamma^{-1}} = (\mathrm{adj}\, \gamma^T)_* f_{\psi}
        \]
        for $\psi\in \mathcal{K}_\Q/\mathcal{L}_\Q$. We then have:

        \begin{prop} \label{prop:conicalduality}
        The association 
        \[
        [\Delta]\mapsto 1_{(\R_+\Delta)^\vee},
        \]
        on a positively-oriented simplex $\Delta$, defines a map 
        \[
        \delta:\widetilde{C}(n) \to \mathcal{K}_\Q/\mathcal{L}_\Q, [\Delta]\mapsto 1_{(\R_+\Delta)^\vee}
        \]
        satisfying the equivariance property
        \[
        \delta(\gamma[\Delta]) = \delta([\Delta]) \circ \gamma^T
        \]
        for $\gamma\in \GL_n(\Q)$. Furthermore, $\widetilde{r}(\bullet) = f_{\delta(\bullet)}(1)$.
        \end{prop}
        \begin{proof}
            This is a direct consequence of the convex cone duality of \cite[IV, Theorem 1.6]{Bv}, once one notes that the duals of lower-dimensional cones are precisely wedges (and identifies the ``polar'' as the negative of the dual). The $\GL_n$-equivariance property is an immediate consequence of the definition of dual cone, and the asserted equality of ``realization maps'' is then a direct consequence of the formulas.
        \end{proof}

    Now, associated to the coordinatization $\alpha$ of $F\otimes \R$ is an embedding
    \[
    \iota_\alpha: U \to \GL_n(\Z).
    \]
    We observe that the value of $\theta^{S^{n-1}}$ on a set of generators for $U^+$, when antisymmetrized, yields a ``fake fundamental domain'' so long as $e_1\in I^+$:
    \begin{prop} \label{prop:fakeshin}
        If $u_1,\ldots, u_{n-1}$ is set of generators of $U^+$, then the set of $(n-1)!$ open cones
        \[
        \R^+\Delta^\circ(v, u_{\tau(1)}v, \ldots,u_{\tau(1)}\ldots u_{\tau(n-1)} v),
        \]
        for any $v\in I^+$, as $\tau$ ranges over $S_{n-1}$, contains a unique $U^+$-translate of any $x\in (F\otimes \R)^+/\R_+^\times$ not in the orbit of their boundary faces. Equivalently, the associated cone contains a unique $U^+$ translate of any $x\in (F\otimes \R)^+$ not in a boundary orbit.
    \end{prop}
    \begin{proof}
        We take the logarithm map from the proof of Dirichlet's unit theorem
        \[
            (F\otimes \R)^+\to \mathbb{R}^{n}/\langle (t,t,\ldots, t)\rangle, x\otimes r \mapsto [(\log (r\sigma_1(x)), \log (r\sigma_2(x)), \ldots \log (r\sigma_n(x)))]
        \]
        where $\sigma_1,\ldots, \sigma_n$ are the real places of $F$, and the notation indicates we quotient by the diagonal copy of $\R$. Recall that this maps $U^+$ onto a full rank lattice inside this vector space. The statement of the proposition is then simply the standard decomposition of the unit $n$-cube for this lattice (translated by the image of $v$) into $(n-1)!$ simplices. Note that an \emph{actual} Shintani decomposition would have to contain certain lower-dimensional cones (the corresponding boundary faces) to be a true fundamental domain for $U^+$. 
    \end{proof}
    
    Fix any ray $v$ in $(F\otimes \R)^+\overset{\alpha}\subset \R^n$, and choose an ordering of basis $u_1,\ldots, u_{n-1}$ of $U^+$ such that $\Delta^\circ(v, u_{1}v, \ldots,u_{1}\ldots u_{n-1} v)$ is positively oriented with respect to the standard orientation of $\R^n$ (given by the ordered standard basis $e_1, \ldots, e_n$). Note that $\det(v, u_{1}v, \ldots,u_{1}\ldots u_{n-1} v)$ is nowhere vanishing in $(F\otimes \R)^+$; thus, by continuity, we may define:

    \begin{defn} \label{defn:fundclass}
        We define a choice of fundamental class $c_{U^+}^\alpha\in H_{n-1}(U^+,\Z)$ corresponding to an orientation such that $\Delta^\circ(v, u_{1}v, \ldots,u_{1}\ldots u_{n-1} v)$ is positively oriented for \emph{any} $v\in (F\otimes \R)^+$. (Note that the dependence on $\alpha$ is only up to sign.)
    \end{defn}

    \begin{rem} \label{rem:orthflip}
        More generally, there are $2^n$ orthants in $F\otimes \R$ corresponding to possible combinations of signs at real places of $F$; their disjoint union is the complement of the zero locus of $v\mapsto \det(v, u_{1}v, \ldots,u_{1}\ldots u_{n-1} v)$, so the orientation of $\Delta^\circ(v, u_{1}v, \ldots,u_{1}\ldots u_{n-1} v)$ is constant on each of them. This will be $+1$ on all the orthants with an even number of negative signs, and $-1$ otherwise: this is clear if one simultaneously diagonalizes $u_1,\ldots, u_{n-1}$, and notices that the matrices $\mathrm{diag}(\pm 1, \ldots, \pm 1)$ (which is orientation-preserving exactly when the number of $-1$s is even) in the resulting eigenbasis transitively permute the simplices along with their corresponding orthants.
    \end{rem}
    
    Now take a cocycle like $\theta^{S^{n-1}}(n)|_{U^+}$, but lifting to some general $\Delta(v)\in S^{n-1}(\Q)$ instead of $[\Delta(e_1)]$ from $1\in \Chains(n)_0$, representing the class $\Theta^{S^{n-1}}(n)|_{U^+}\in H^{n-1}(U^+,C(n))$. Then the cap product
    \begin{equation} \label{eq:fakeshin}
    \theta^{S^{n-1}}(n)\frown c_{U^+}^\alpha = \sum_{\sigma\in S_{n-1}} (-1)^{\sgn\,\sigma}[\Delta(v, u_{\tau(1)}v, \ldots,u_{\tau(1)}\ldots u_{\tau(n-1)} v)] \in H_0(U^+, C(n))
    \end{equation}
    is a formal sum of simplices whose $\R^+$-span is a Shintani domain, up to lower-dimensional cones.

    Previous cohomological approaches, such as \cite{Hill} or \cite{CDG}, used auxiliary data (a lexicographic order, respectively a ``Colmez'' perturbation vector) to find the lower-dimensional cones which correct the top-dimensional term corresponding to the ``fake Shintani domain'' $
    f_{C}(\xi)$, for 
    \[C=\R_+\Delta^\circ (v, u_{\tau(1)}v, \ldots,u_{\tau(n-1)}\ldots u_{\tau(1)} v),\]
    \and thus obtain totally real $L$-values; for example, \cite[Theorem 2.1]{CDG} shows that there exists a formal linear combination $\varepsilon$ of indicator functions of lower-dimensional faces of these various cones such that 
    \[
    D=\epsilon + \sum_{\tau\in S_{n-1}} 1(\R_+\Delta^\circ (v, u_{\tau(1)}v, \ldots,u_{\tau(n-1)}\ldots u_{\tau(1)} v) \in \mathcal{K}_{\R^n}
    \]
    is a ``signed Shintani domain,'' which is to say formally satisfies the requirements of Theorem \ref{thm:shin} if we view $f_\bullet$ as a function on $\mathcal{K}_\Q/\mathcal{L}_\Q$. However, \emph{our} cocycle has values which look like
    \[
    \widetilde{r}([\Delta(v, u_{\tau(1)}v, \ldots,u_{\tau(n-1)}\ldots u_{\tau(1)} v)]),
    \]
    which as we have seen, are actually \emph{not} the generating functions considered in those articles. Thus, even ignoring the issue of lower-dimensional faces, it is not immediately clear how Shintani's result Theorem \ref{thm:shin} can be used to relate our cocycle to $L$-values of totally real fields; we are a ``conical duality'' away.

    The solution lies in the following observation: the target $\mathcal{K}_\Q/\mathcal{L}_\Q$ of the duality morphism $\delta$ is very similar in nature to $\widetilde{\mathrm{C}}(n)$; as remarked previously, the difference is that the latter ``cares about'' wedges, while the former ``cares about'' lower-dimensional cones. If we could define the realization map $\widetilde{r}$ on the \emph{former} module instead of the latter, then $\widetilde{r}\circ \delta$ \emph{would} carry the class of our ``fake fundamental domain'' to the ``main'' (top-dimensional) term contributing to a Shintani function for a totally real field $F$. But $\widetilde{r}$ \emph{does} care about wedges, and ignores lower-dimensional cones, so this strategy does not quite make sense, as written.

    The work-around is that if we introduce certain stabilizations of our symbol complexes by torsion sections, then we can remove all the above obstructions: the classes of wedges and lower-dimensional cones alike are annihilated, so that $\delta$ becomes a kind of auto-duality, whereupon a realization extending $\widetilde{r}$ makes sense on the target. This stabilization will also remove the contribution of all lower-dimensional terms in associated Shintani generating functions, effectively making our fake fundamental domain as good as a real one. The resulting stabilized cocycle can then be directly related to a (stabilized) Shintani function for a totally real field $F$, and therefore to $L$-values.
    
    \subsection{Stabilized duality and $p$-adic $L$-functions} \label{section:smooth}

    We introduce stabilizations at auxiliary integers, necessary to realize the previously-stated strategy for constructing and proving interpolation properties for our $p$-adic $L$-functions. Because $\delta$ is a duality map, there will be two dual notions of stabilization, which we will call \emph{avoiding} and \emph{smoothing}. Broadly speaking, ``avoiding'' corresponds to ensuring the poles of the values of $\Theta^{dR}(n)$ avoid zero, while ``smoothing'' enforces a degree-zero condition killing off lower-dimensional conical faces.

    \begin{rem}
        The Euler factors that appear corresponding to these two types of stabilization are precisely those corresponding to the two kinds of sets of places often labelled $S$ and $T$ in literature on Stark-type conjectures; e.g., in \cite{DK}.
    \end{rem}

    \subsubsection{$c$-avoiding}
    Let $c>1$ be an integer prime to $p$; though it is not necessary, for later convenience we may as well take it to be prime. We first describe a ``$c$-avoiding'' modification of our symbol complex, as well as the resulting stabilization of $\Theta^{dR}(n)$; this latter class has the upside is that it is valued in functions holomorphic in a $p$-adic neighborhood of the identity.
    
    Write
    \[
    {}_c\widetilde{\Chains'}(n)_i=\bigoplus_{x\in \mu_{c}^n-\{1\}} {}_x\widetilde{\Chains}(n)_i
    \]
    where the definition of ${}_x\widetilde{\Chains}(n)_\bullet [-1]$ needs to be explained: roughly, it is the augmented simplicial chain complex for the ind-triangulation of $S^{n-1}$ consisting of simplices whose $< (n-1)$-dimensional faces do not pass through any rational point reducing to $x$ (considered in $(\Z/c)^n$ via the complex uniformization). We will call a triangulation of $S^{n-1}$ satisfying this property ``$x$-avoiding''.

    The \emph{existence} of $x$-avoiding triangulations is clear from the density of points reducing to any given ray in $(\Z/c)^n$ in any open set of $S^{n-1}$; this is a very basic form of weak approximation, or just the Chinese remainder theorem. The existence of a direct limit over \emph{all} such triangulations (analogous to \eqref{eq:ind}) is slightly less obvious, requiring a slight refinement in the proof of Proposition \ref{prop:refine}: we need two $x$-avoiding triangulations $T_1$ and $T_2$ to have an $x$-avoiding common refinement. In that argument, when superimposing two $x$-avoiding triangulations, all the resulting polyhedral faces of dimension $<(n-1)$ will automatically already avoid $x$ modulo $c$ by assumption. Here, however, we again use weak approximation to see that we can always pick the point of barycentric subdivision so that all the resulting new faces also avoid $x$, rather than picking an arbitrary rational point.
    
    Thus, as in \eqref{eq:ind}, we can define ${}_x\widetilde{\Chains}(n)_\bullet[-1]$ as the direct limit over \emph{all} chain complexes of $x$-avoiding triangulations, as before; it is naturally a subcomplex of $\widetilde{\Chains}(n)_\bullet$, and the inclusion is a quasi-isomorphism. 
    
    Clearly, pushforward by $\gamma$ carries $x$-avoiding triangulations to $\gamma x$-avoiding triangulations. Then the complex ${}_c\widetilde{\Chains'}(n)_i$ carries a natural $\GL_n(\Z)$ action, which permutes the summands
    \[
        {}_x\widetilde{\Chains}(n)_i \to {}_{\gamma x}\widetilde{\Chains}(n)_i
    \]
    given simply by the pushforward $\gamma_*:S^{n-1}\to S^{n-1}$ on the level of chains. Note that as plain modules, we can identify
    \[
         {}_c\widetilde{\Chains'}(n)_0 \cong \Z \{\mu_{c}^n - \{1\}\}.
    \]
    Similarly, in top degree, we have
    \[
        H_{n}({}_c\widetilde{\Chains}(n)_\bullet) \cong \Z \{\mu_{c}^n-\{1\}\}(\mathrm{sgn})
    \]
    with the same action twisted by the sign character, since there is one fundamental class of the sphere for each $x\in \mu_{c}^n$. If we write $H^0\subset \Z \{\mu_{c}^n-\{1\}\}$ for the submodule of degree-zero elements in $\Z \{\mu_{c}^n\}(\mathrm{sgn})$, we can define
    \begin{align}
        {}_c\widetilde{\Chains}(n)_\bullet := {}_c\widetilde{\Chains'}(n)_\bullet/H^0 \\
        {}_c\Chains(n)_\bullet :={}_c\widetilde{\Chains'}(n)_\bullet/H_{n}({}_c\widetilde{\Chains}(n)_\bullet),
    \end{align}
    and the latter complex is exact. The same argument in the proof of Theorem \ref{thm:gerssym} (but skipping the motivic realization to pass directly to the $d\log$ regulator) shows that there is a $\GL_n(\Z)$-equivariant realization map 
    \[
    {}_c\tilde{r}: {}_c\widetilde{\Chains}(n)_n \to \mathcal{M}_{\G_m^n}(-\det)^{(0_{/c})}
    \]
    given on acyclic simplicial generators by
    \begin{equation} \label{eq:stabilizedsum}
        [\Delta(r_1,\ldots, r_n)]_x \mapsto \frac{1}{\det M}t_x^*M_* \frac{(-1)^nz_1\ldots z_n}{(1-z_1)\ldots (1-z_n)}
    \end{equation}
    where 
    \[
        M= \begin{pmatrix} r_1& \ldots & r_n\end{pmatrix}
    \]
    and the superscript $(0_{/c})$ means that the action of $[a]_*$ for positive integers $a$ with $(a,c)=1$ depends only on the class of $a$ modulo $c$. This realization quotients down to a map 
    \[
        {}_cr: {}_c\Chains(n)_n \to \mathcal{M}_{\G_m^n}(\sgn)^{(0_{/c})}/\mathrm{constants}.
    \]
    By the $x$-avoiding condition, the image of ${}_cr$ lands in the intersection of $A_p^{(n_{/c})}$ for \emph{all} $p$ prime to $c$. Fixing some such $p$, we obtain by lifting the fixed class
    \[
    \mu_c^n - \{1\} \in {}_c\widetilde{\Chains}(n)_0
    \]
    a cocycle
    \[
    {}_c\Theta^{dR}(n)\in H^{n-1}(\GL_n(\Z), A_p^{(n_{/c})}(\sgn)/\Z_p)
    \]
    which, via Cartier duality (as discussed previously), yields a cocycle
    \[
    {}_c\Theta^{dR}_{\D_p}(n)\in H^{n-1}(\GL_n(\Z), 
    \hat{\D}(\Z_p^n, \Z_p)^{(0_{/c})}(\sgn)/\langle \delta_0\rangle)
    \]
    where $\delta_0$ is the atomic (``Dirac delta'') measure at zero. 

    \begin{rem} \label{rem:finitetofixed}
    In fact, since the lifted element $\mu_c^n - \{1\}$ is fixed by $[a]_*$ for all $[a]_*$ prime to $c$, we can project to make this valued in $[a]_*$-fixed distributions for all $a$ prime to $c$, at the cost of inverting $|(\Z/c)^\times|=\phi(c)$. When we pass to $p$-adic coefficients later, by choosing $c$ so that $(p,\phi(c))=1$, we hence can consider our cocycle as taking values in the $(0)$-part. In fact, we will only use fixedness under $[p]_*$, and one only needs to pick $c$ such that $(p, \mathrm{ord}_{(\Z/c)^\times}(p))$ for this.
    \end{rem}

    \begin{rem} \label{rem:smoothsharks}
    There is a simple relation between ${}_c\Theta^{dR}(n)$, viewed as valued in $\mathcal{M}_{\G_m^n}(\sgn)^{(0_{/c})}$, and $\Theta^{dR}(n)$: we can view both ``spherical chain complexes'' $\GL_n(\Z)$-equivariantly inside $\bigoplus_{x\in \mu_c^n} \Chains(n)_\bullet$, via the obvious inclusion for ${}_c\Chains(n)$, and inclusion at the identity section for $\Chains(n)$. The class $\Theta^{dR}(n)$ is obtained by lifting $\{1\}\in \Chains(n)_0$, while ${}_c\Theta^{dR}(n)$ is obtained by lifting $([c]^*-1)\{1\}$; then by functoriality under $[c]^*$ of the lifting process in group cohomology, we find
    \[
    (c^n[c]^*-1)\Theta^{dR}(n)={}_c\Theta^{dR}(n)
    \]
    after checking that under ${}_cr\circ [c]^* = c^n[c]^* \circ {}_cr$, which is a short calculation.\footnote{The extra scalar $c^n$ can be thought of as coming from the $\det$-twist in the contraction from $n$-forms to functions.}
    \end{rem}
    
    \subsubsection{Restriction to a nonsplit torus}

    We now restrict to $U\subset \GL_n(\Z)$, corresponding to some totally real field $F$ of degree $n$ as previously. As before, we pick a fractional ideal $I$, but instead of identifying it with $\Z^n$ as in the original Shintani method, we will fix instead an isomorphism
    \[
        \alpha: I\xrightarrow{\sim} (\Z^n)^\vee
    \]
    to the \emph{dual}, which is equivariant for $U\subset \GL_n(\Z)$ for the usual unit multiplication on the source and the standard dual representation on the target. Since $x\mapsto \langle x, -\rangle$ identifies $\Z^n$ and $(\Z^n)^\vee$ as $\GL_n(\Z)$-modules, we can also view $\alpha$ as being an isomorphism $I\to (\star^*)\Z^n$. This dualizing of the usual coordinatization is necessary to make things consistent with our treatment of conical duality. The definition of the orientation $c_{U^+}^\alpha\in H^{n-1}(U^+,\Z)$ is as in Definition \ref{defn:fundclass}, with respect to the standard orientation on $(\R^n)^\vee$; i.e., corresponding to the ordering $e_1^\vee,\ldots, e_n^\vee$.
    
    We write $\iota_\alpha: U\hookrightarrow \GL_n(\Z)$ for the inclusion; remark that any scaling $\alpha \circ [\times a]$ for $a\in F$ is $U$-equivariant for the same inclusion $\iota_\alpha$. Conversely, $\iota_\alpha$ can be recovered from the data of $\alpha$ by writing out the matrix action of $U$ on $I$, then taking the inverse transpose, so each $\iota_\alpha$ is uniquely associated to an ``isogeny class'' of coordinatizations $\alpha$.

    Then $\alpha$, treated as a map $I\to (\star^*)\Z^n$, induces a $U$-equivariant identification
    \begin{equation} \label{eq:distrocoord}
        \hat{\D}(\Q_p^n,\Z_p)\xrightarrow{\sim} \D(F_p,\Z_p), \mu\mapsto (\varphi\mapsto \mu(\varphi \circ \alpha^{-1}))
    \end{equation}
    via which we view ${}_c\Theta^{dR}_{\D_p}(n)|_{U^+}$ as living inside $H^{n-1}(U^+, \D(F_p, \Z_p)^{(0_{/c})}/\langle \delta_0\rangle))$.

    \subsubsection{Stabilizing at ideals of $F$}
    If we \emph{first} restrict to the image of $F^\times \xrightarrow{\iota_\alpha} \GL_n(\Q)$, then we may form the $F^\times$-equivariant complex ${}_\c\Chains(n)_\bullet$ for \emph{any} ideal $\c$ of $F$ (which we will choose prime to $p$), entirely analogously to the preceding construction, but only over $\c$-torsion points of $\G_m^n$. Then the cycle $\G_m^n[\c]-\{1\}\in {}_\c\Chains(n)_\bullet$ is $U$-fixed, and we have a realization map
    \[
        {}_\c r: {}_\c\Chains(n)_n \to \mathcal{M}_{\G_m^n}(\sgn)^{(0_{/\c})}/\mathrm{constants}.
    \]
    defined exactly analogously to ${}_cr$ via a sum of translates \eqref{eq:stabilizedsum}, from which we deduce classes we notate by
    \begin{equation} \label{eq:driotaalpha}
        {}_\c\Theta^{dR}(\iota_\alpha)\in H^{n-1}(U, A_p(\sgn)^{(0_{/\c})}/\Z_p), {}_\c\Theta^{dR}_{\D_p}(\iota_\alpha)=(\alpha^{-1})_*\widehat{\kappa}_* {}_\c\Theta^{dR}(\iota_\alpha)\in H^{n-1}(U, \D(F_p, \Z_p)^{(0_{/\c})}/\langle \delta_0\rangle)
    \end{equation}
    satisfying $((\mathbf{N}\c)[\c]^*-1) \Theta^{dR}(n)
    _U= {}_\c\Theta^{dR}(\iota_\alpha)$. Here, $[\c]:\G_m^n\to \G_m^n$ is the $U$-equivariant map dual to the inclusion of lattices $\c I\hookrightarrow I$ (via the identification $\alpha$).
    
    \subsubsection{$p$-adic $L$-functions} \label{section:padic}
    
    We now define ${}_\c\zeta_p^{I,\alpha}$ to be the image of ${}_\c\Theta^{dR}_{\D_p}(\iota_\alpha)$ under the composition
    \begin{equation} \label{eq:getzeta}
        H^{n-1}(U^+, \D(F_p, \Z_p)^{(0_{/c})}/\langle \delta_0\rangle)) \xrightarrow{\frown c_{U^+}^\alpha} H_0(U^+, \D(F_p, \Z_p)^{(0_{/\c})}/\langle \delta_0\rangle))\to \D(F_p/U^+, \Z_p)^{(0_{/\c})}/\langle \delta_0\rangle 
    \end{equation}
    Here, the last map is the natural map
    \[
        H_0(U^+, \D(F_p, \Z_p)^{(0_{/\c})}/\langle \delta_0\rangle)) \to H_0(U^+, \D(F_p, \Z_p)/\langle \delta_0\rangle))^{(0_{/\c})} \cong \D(F_p/U^+, \Z_p)/\langle \delta_0\rangle.
    \]
    Explicitly, if $\varphi$ is a compactly supported continuous function on $F_p/U^+$ with $\varphi(0)=0$, which we view as a function on $(\Q_p^n)^\vee$ via 
    \[
    (\Q_p^n)^\vee\xrightarrow{\alpha^{-1}}F_p\xrightarrow{\varphi} \Q_p,
    \]
    we have
    \begin{equation} \label{eq:getzeta2}
    \int_{F_p/U^+}\varphi(s)\, d{}_\c\zeta_p^{I}(s) = \int_{(\Q_p^n)^\vee} \varphi(\alpha^{-1}(t))\, d({}_\c\Theta^{dR}_{\D_p}(n)|_{U^+} \frown c_{U^+}^\alpha). 
    \end{equation}
    For $\c=(c)$ an ideal of $\Z$, these zeta elements are specializations of the ``global'' cocycle ${}_c\Theta^{dR}(n)$; otherwise, it is only a specialization of ${}_\c\Theta^{dR}(\iota_\alpha)$.

    If one replaces the identification $\alpha:I\to (\mathbb{Z}^n)^\vee$ with some scalar multiple
        \[
            \alpha \circ [t^{-1}]: (t)I \to (\mathbb{Z}^n)^\vee
        \]
    for any $t\in F^\times$, as we noted, this is equivariant for the same $\iota_{\alpha}:U \hookrightarrow \GL_n(\Z)$. As with $c^{\alpha}_{U^+}$, the dependence of the $p$-adic $L$-function on $\alpha$ is quite weak; rescaling by $t$ results in the same information up to a flip of sign:

    \begin{prop} \label{prop:signkey}
    We have the relation
        \begin{equation} \label{eq:scaleideal}
        {}_c\zeta^{I,\alpha}_p = (\sgn\, \mathbf{N}t)[t]_*{}_c\zeta^{(t)I, \alpha \circ [t^{-1}]}_p.
        \end{equation}
        for the map
        \[
            [t]_*: \D(F_p/U^+, \Z_p)^{(0_{/c})}\to \D(F_p/U^+, \Z_p)^{(0_{/\c})}, \mu\mapsto (\varphi \mapsto \mu(\varphi \circ [t])).
        \]
    \end{prop}
    \begin{proof}
        From the functoriality under $[t]_*$ of the sequence of maps \eqref{eq:getzeta}, this amounts to showing that $c_{U^+}^{\alpha\circ [t^{-1}]} = (\sgn\, \mathbf{N}t) c_{U^+}^\alpha$, which follows from Remark \ref{rem:orthflip}. 
    \end{proof}

    This provides a strong sense in which the information in our $p$-adic $L$-functions only depends on the narrow ideal class of $I$, and in fact only on the wide ideal class ``up to sign''. This will be used later in Section \ref{section:2adic} to build $p$-adic $L$-functions for ray class characters, and prove $2$-adic congruences.
    
    We have thus defined certain ``$p$-adic $L$-function'' elements, but have not justified the name; we need an interpolation theorem:

    \begin{thm} \label{thm:interp}
        Let $\psi:I^+/U^+\to \Z_p^\times$ be any locally constant function, with $\psi(0)=0$. Then we have the specialization property
        \[
        \int_{F_p/U^+} \psi(t)\mathbf{N}(t)^k\, d{}_\c\zeta_p^{I,\alpha}(t)=\zeta_{I}(\psi, -k)-\mathbf{N}\c^{k}\zeta_{\c I}(\psi, -k).
        \]
    \end{thm}

    Note that Theorem \ref{thm:distros} from the introduction is an immediate corollary of the above interpolation statements. The remainder of this subsection (until Section \ref{section:2adic}) will be dedicated to proving this interpolation theorem, by finding a duality reconciling the (stabilized) values of $\Theta^{dR}(n)$ to the (stabilized) classical Shintani method for $F$. This will be an intricate technical argument, which may be safely skipped by a reader willing to accept it on faith.

    \begin{rem} \label{rem:comparison}
        Since $U^+$ preserves the angular domain $(F\otimes \R)^+\subset \R^n$ (and all the other orthants coming from combinations of signs at places of $F$), the orientation obstruction always vanishes on $U^+$; i.e., the Euler cocycle $\varepsilon_n|_{U^+}$ is \emph{identically} zero (regardless of the $\Delta$-extension used to define our cocycle representative), and we can thus canonically lift ${}_c\Theta^{dR}(\iota_\alpha)$ over the orientation obstruction to obtain natural-seeming lifts of ${}_c\zeta_p^{I,\alpha}$ in $\D(F_p/U^+, \Z_p)$. This raises the question of whether specialization at the trivial character $\psi = 1$ of these lifts contains any interesting information, or are related to the corresponding complex $L$-values; in light of Section \ref{section:cones}, this seems tantamount to ``ignoring lower-dimensional cones'' in the Shintani method, so it would be somewhat surprising if the values were correct. We would be interested in computational verification and comparison of these values to the true $L$-values (known by Stark's conjecture).
    \end{rem}
    
    \subsubsection{$c$-smoothing}
    We now describe the ``$c$-smoothed'' modification of $\widetilde{\Chains}(n)$: in fact, this modification will set all wedge classes to zero, so it is actually rather a ``Orlik-Solomon''-flavored complex (as in Section \ref{section:steinberg}), moreso than a ``spherical chains'' flavored one. As with $c$-avoiding, after describing the case of smoothing at an integer resulting in a $\GL_n(\Z)$-cocycle, we will also indicate a generalization to smoothing at an arbitrary ideal $\c\subset F$; this yields only a $U$-cocycle when $\c$ is not a rational ideal. Again, we emphasize smoothing at integers more for expository purposes, as working with $\GL_n$-actions makes the underlying linear algebraic structure conceptually clearer.
    
    Write $\widecheck{\mu_{c}^n}$ for the Pontryagin dual of the $c$-torsion in $\G_m^n$; then we define the $c$-smoothed complex
    \[
    {}^c\OS(n)_i=\bigoplus_{\xi\in \widecheck{\mu_{c}^n}-\{1\}} {}^\xi\OS(n)_i
    \]
    where ${}^\xi\OS(n)_\bullet$ is defined as the Orlik-Solomon complex for the lines $\ell\in \P^{n-1}(\Q)=\P^{n-1}(\Z)$ for which $\xi|_{\ell \pmod c}$ is not the trivial character. Each summand is exact as a complex by the usual Orlik-Solomon formalism, so the whole complex is as well. In top degree, we will also $\St(n)$ instead of $\OS(n)_n$, continuing earlier ``Steinberg representation'' notation. We define the $\GL_n(\Z)$-action on this complex by
    \[
        {}^\xi\OS(n)_i\to {}^{ \xi\circ \gamma^{-1}}\OS(n)_i, ([\ell_1]\wedge \ldots \wedge [\ell_i])_\xi \mapsto [([\gamma\ell_1]\wedge \ldots \wedge [(\gamma\ell_i])_{\xi\circ \gamma^{-1}}.
    \]

    \begin{prop}
    There is a $\GL_n(\Z)$-equivariant realization map
    \[
    {}^cr: {}^c\St(n) \to (\mathcal{M}_{\G_m^n}(\sgn))^{(0_{/c})}
    \]
    given by 
    \begin{equation} \label{eq:smoothformula}
    ([\ell_1]\wedge \ldots \wedge [\ell_n])_\xi \mapsto \sum_{x\in \mu_c^n} \xi(x)t_x^* (\det M)^{-1} M_* \frac{(-1)^nz_1\ldots z_n}{(1-z_1)\ldots (1-z_n)}
    \end{equation}
    where $M=\begin{pmatrix}\ell_1&\ldots &\ell_n\end{pmatrix}$; here, the notation as usual means we pick any integral generator of the line (or in fact, any integral point, by the usual trace-fixedness argument).
    \end{prop}
    \begin{proof}
        We saw previously, in Section \ref{section:steinberg}, that the Orlik-Solomon relations are generated by the spherical chain relations together with the ``wedge'' relations, where the latter amount to checking that the formula \eqref{eq:smoothformula} is independent of replacing a column in $M$ with its negative: our previous trace-fixedness arguments show that the formula is independent of scaling the generators by $[a]$ for $a\in \mathbb{N}$, but not for scaling by $[-1]$.
        
        The proof that the stellar subdivision relations hold in the image of ${}^cr$ is identical to the argument for Theorem \ref{thm:thm1}, so it suffices to check latter wedge relations. Indeed, if we identify $\ell_1,\ldots, e_n$ with a set of integral generators, let
        \[
        M' := \begin{pmatrix}-\ell_1&\ldots &\ell_n\end{pmatrix}
        \]
        be the matrix for the same set of integral generators, with one sign flipped. Then we compute that the difference between the expression \eqref{eq:smoothformula} for $M$ and $M'$ is 
        \[
            \sum_{x\in \mu_c^n} \xi(x)t_x^* (\det M)^{-1} M_* \left(\frac{(-1)^nz_1\ldots z_n}{(1-z_1)\ldots (1-z_n)} + \frac{(-1)^nz_1^{-1}\ldots z_n}{(1-z_1^{-1})\ldots (1-z_n)} \right)
            \]
        \[ 
        = (\det M)^{-1} \sum_{x\in \mu_c^n} \xi(x)t_x^*  M_* \frac{(-1)^{n-1}z_2\ldots z_n}{(1-z_2)\ldots (1-z_n)}.
        \]
        If we split this sum over $x$ into cosets for the order-$c$ subgroup $(\ell_1)_* \mu_c\subset \mu_c^n$, then one computes that
        \[
            t_x^* M_* \frac{(-1)^{n-1}z_2\ldots z_n}{(1-z_2)\ldots (1-z_n)}
        \]
        is constant on each coset, and $\xi$ has sum zero over each coset since it is a nontrivial character on $(\ell_1)_* \mu_c$. The result follows.
    \end{proof}

    Inside ${}^c\OS(n)_0\cong \Z\{\widecheck{\mu_{c}^n}-\{1\}\}$, we have the $\GL_n(\Z)$-fixed element 
    \[
    -\sum_{\xi \in \widecheck{\mu_{c}^n}} \xi,\]
    which, if viewed as a function on $\mu_c^n$ and then as a formal linear combination of elements thereof in the usual way (where the values of the functions become the coefficients), corresponds to the torsion cycle $([c]^*-c^n)\{1\}$.
    
    As usual, lifting this element and taking the realization ${}^cr$ then affords us a cocycle 
    \[
        {}^c\Theta^{dR}(n)\in H^{n-1}(\GL_n(\Z),\mathcal{M}_{\G_m^n}(\sgn)^{(0_{/c})})
    \]
    which, by the same argument used in Remark \ref{rem:smoothsharks}, satisfies $(c^n[c]^*-c^n)\Theta^{dR}(n) = {}^c\Theta^{dR}(n)$.
    
    Now if we first restrict to $U\hookrightarrow{\iota_\alpha} \GL_n(\Z)$, we can likewise define a $\c$-smoothed complex with $U$-action
    \[
    {}^\c\OS(n)_i=\bigoplus_{\xi\in \widecheck{\G_m^n[\c]}-\{1\}} {}^\xi\OS(n)_i
    \]
    along with a $U$-equivariant realization map ${}^\c r: {}^\c\St(\iota_\alpha) \to (\mathcal{M}_{\G_m^n}(\sgn))^{(0_{/\c})}$. Then the $U$-fixed element 
    \[
    \sum_{\xi \in \widecheck{\G_m^n[\c]}-\{1\}} - \xi
    \]
    affords us a cocycle ${}^\c\Theta^{dR}(\iota_\alpha)\in H^{n-1}(U, (\mathcal{M}_{\G_m^n}(\sgn))^{(0_{/\c})})$ with $\mathbf{N}\c([\c]^*-1)\Theta^{dR}(n)|_U = {}^\c\Theta^{dR}(\iota_\alpha)$, generalizing the case of $\c=(c)$.

    \subsubsection{Stabilized duality}

    We now let $b$ and $c$ be distinct rational primes, which are also distinct from $p$. With identical arguments, we can combine the stabilizations of the preceding two sections to be simultaneously $b$-avoiding and $c$-smooth, resulting in an exact homological complex 
    \[
    {}_b^c\OS(n)_\bullet = \bigoplus_{\xi \in \widecheck{\mu_c^n}-\{1\}} \bigoplus_{x \in \mu_b^n-\{1\}} {}^\xi_x\OS(n)_\bullet
    \]
    where the summands consisting of symbols whose $\le(n-1)$-dimensional spans avoid $x$, and whose lines modulo $d$ are not contained in the kernel of the character $\xi$. The $\GL_n(\Z)$-invariant class 
    \[
    - \sum_{x\in \mu_b^n-\{1\}}\sum_{\xi\in \widecheck{\mu_c^n}-\{1\}} 1_{\xi, x}
    \]
    yields a class ${}_b^c\Theta^{St}(n)\in H^{n-1}(\GL_n(\Z),{}^c_b\St(n))$. The $\GL_n(\Z)$-equivariant realization map
    \[
    {}^c_br: {}^c_b\St(n)\to A_p(\sgn)^{(n_{/bc})}
    \]
    results in
    \[
    {}_b^c\Theta^{dR}(n) :={}^c_br_* {}_b^c\Theta^{St}(n)\in H^{n-1}(\GL_n(\Z), A_p(\sgn)^{(n_{/bc})})
    \]
    such that 
    \[
    (c^n[c]^*-c^n)(b^n[b]^*-1)\Theta^{dR}(n) = (c^n[c]^*-c^n){}_b\Theta^{dR}(n)= (b^n[b]^*-1) {}^c\Theta^{dR}(n)= {}^c_b\Theta^{dR}(n).
    \]
    Analogously, for $\b$ and $\c$ distinct primes of $F$ not dividing $p$, we have a class 
    \[
        {}_\b^\c\Theta^{dR}(\iota_\alpha)= {}^\c_\b r_*{}_\b^\c\Theta^{St}(\iota_\alpha) \in H^{n-1}(U, A_p(\sgn)^{(n_{/\b\c})})
    \]
    such that 
        \[
        (\mathbf{N}\c[\c]^*-\mathbf{N}\c)(\mathbf{N}b[\b]^*-1)\Theta^{dR}(n)|U = (\mathbf{N}\c[\c]^*-\mathbf{N}\c){}_\b\Theta^{dR}(\iota_\alpha)= (\mathbf{N}\b[\b]^*-1) {}^\c\Theta^{dR}(n)= {}^\c_\b\Theta^{dR}(n).
    \]
    We are now ready to define the stabilized conical duality map and prove its properties. We recall that by definition, $(\star^*){}^b_c\St(n)$ carries the action
    \[
    ([\ell_1]\wedge \ldots \wedge [\ell_n])_{\xi_x, y} \mapsto ([(\gamma^T)^{-1}\ell_1] \wedge \ldots \wedge [(\gamma^T)^{-1} \ell_n])_{\xi_x \circ \gamma^T, (\gamma^T)^{-1} y}.
    \]
    Specializing this, we also write $(\star^*){}^\b_\c\St(\iota_\alpha)$ for the $\star U$-module with the same underlying space as ${}^\b_\c\St(\iota_\alpha)$, on which $\gamma \in \star U$ acts by the action of $\star \gamma \in U$ on ${}^\b_\c\St(n)$.

    \begin{prop} \label{prop:c-conicalduality}
        There is a $\GL_n(\Z)$-equivariant map
        \[
        {}^c_b\delta: {}_b^c\St(n) \to (\star^*){}^b_c\St(n)
        \]
        given by sending 
        \[
        (([\ell_1] \wedge \ldots \wedge [\ell_n]))_{\xi_x,y} \mapsto ([\ell_1^\vee] \wedge \ldots \wedge [\ell_n^\vee])_{\xi_y,x}
        \]
        where $\xi_x$, for $x\in \mu_c^n$, is the character
        \[
        \xi_x:\mu_c^n \to \C^\times ,t\mapsto \exp(2\pi ic\langle x, t\rangle)
        \]
        where we identify $\mu_c^n \cong \frac{1}{c}\Z^n/\Z^n$ to define the inner product of $x$ and $t$, and similarly for the relationship between $y$ and $\xi_y$, replacing $c$ by $b$ everywhere. 
        
        The same formula yields a $U$-equivariant map 
        \[
        {}^\c_\b\delta: {}_\b^\c\St(\iota_\alpha) \to (\star^*){}^b_c\St(\star \iota_\alpha)
        \]
        when $\b,\c$ are ideals in $F$. 
    \end{prop}
    \begin{proof}
        The proof in the $U$-equivariant case is formally identical to the general linear case, so we treat just the latter. We by noting that if the formula is well-defined, then $\GL_n(\Z)$- (respectively $U$-)equivariance is clear. 
        
        We now observe that for each fixed factor labeled by $x,y$, ignoring the modulo $c$ and $d$ conditions, the map ${}_c\delta$ can be identified as the quotient of the duality map
        \[
        \delta: \widetilde{\Chains}(n)_n \xrightarrow{\delta} \mathcal{K}_\Q/\mathcal{L}_\Q
        \]
        by the wedge classes on the source, and their duals the indicators of lower-dimensional cones on the target. Indeed, we have already seen previously that the Steinberg module is the quotient of $\widetilde{\Chains}(n)_n$ by wedges, and the latter quotient
        \[
        \mathcal{K}_\Q/\mathcal{L}_\Q \to \St(n)
        \]
        is given on top-dimensional generators as
        \[
        1_{\R^+\Delta^\circ(r_1,\ldots, r_n)}\mapsto [\ell_1] \wedge \ldots \wedge [\ell_n].\]
        This is well-defined by the earlier discussion of ``$M$-additivity'' (along with the evident fact that there can be no relations between a top-dimensional cone indicator function and any combination of lower-dimensional ones). One then needs only to note that the ``local'' conditions with respect to the $c$-torsion point $x$ and the $d$-torsion point $y$, defining the submodules ${}^{\xi_x}_y\St(n)$, ${}^{\xi_y}_x\St(n)$, are dual to each other under $\delta$: indeed, if $[\ell_1]\wedge \ldots \wedge [\ell_n]$ avoids $x$, then $\xi_x$ is nontrivial on each of $\ell_1^\vee,\ldots, \ell_n^\vee$: the former condition says none of the $(n-1)$-hyperplanes spanned by the $r_\bullet$ pass through $x$, which is the same as saying that $x$ is not in the kernel of the any of the linear forms $\langle \ell_\bullet^\vee, -\rangle$, which is to say that $\xi_x$ is not trivial on any of the $\ell_\bullet^\vee$. The $y$ conditions are symmetric.
    \end{proof}

    We consider now the classes
    \[
    {}^c_b\delta_*{}^c_b\Theta^{St}(n)\in H^{n-1}(\GL_n(\Z), (\star^*){}^b_c\St(n)),
    {}^\c_\b\delta_*{}^\c_\b\Theta^{St}(\iota_\alpha)\in H^{n-1}(\GL_n(\Z), (\star^*){}^b_c\St(n))
    \] whose realizations under ${}^b_cr_*$ we notate
    \begin{equation} \label{eq:cdreal}
         {}^b_c\Theta^{Shin}(n) \in H^{n-1}(\GL_n(\Z),(\star^*)A_p(\sgn)^{(n_{/bc})}),\;\;\; {}^\b_\c\Theta^{Shin}(\iota_\alpha) \in H^{n-1}(U,(\star^*)A_p(\sgn)^{(n_{/\b\c})}), 
    \end{equation}
    To prove Theorem \ref{thm:interp}, we will first show that ${}^c_b\delta$ ``acts trivially'' on cohomology: we note that $\star$ induces a kind of tautological action on $\GL_n(\Z)$-cohomology
    \[
    \star^*:H^i(\GL_n(\Z),M) \to H^i(\GL_n(\Z),(\star^*)M)
    \]
    by acting on cochains as $c(\gamma_1,\ldots, \gamma_i)\mapsto (c\circ \star^i)(\gamma_1,\ldots, \gamma_i)$; the $U$-equivariant version of this statement is that we have a diagram
    \[
        \star^*:H^i(U,M) \to H^i(\star U,(\star^*)M)
    \]
    where $\star U$ means the same abstract group $U$, but embedded into $\GL_n(\Z)$ via $\star \iota_\alpha$. Then the main part of the result we need is a considerable technical digression which we defer to the appendix:
    
    \begin{prop} \label{prop:titsduality}
        The class ${}^b_c\Theta^{Shin}(n)$ is equal to $(\star^*){}^b_c\Theta^{dR}(n)$, and correspondingly, ${}^\b_\c\Theta^{Shin}(\iota_\alpha)$ is equal to $(\star^*){}^b_c\Theta^{dR}(\star \iota_\alpha)$
    \end{prop}
    \begin{proof}
        Using a duality of (restricted) Tits buildings giving rise to (restricted) Steinberg representations in their top reduced homology, we prove in Appendix \ref{appendix:b} that ${}^c_b\delta_*{}^c_b\Theta^{St}(n)$ is equal to ${}^b_c\Theta^{St}(n)$, which implies the proposition by taking realizations. The statement for $\b$ and $\c$ follows similarly from the equality of ${}^\c_\b\delta_*{}^\c_\b\Theta^{St}(\iota_\alpha)$ and ${}^\b_\c\Theta^{St}(\star \iota_\alpha)$, proved in the same way.
    \end{proof}

    It now only remains to compute the values of $(\star^*){}^\b_\c\Theta^{Shin}(\star \iota_\alpha)$ using the Shintani method. The notation in the following proposition is as in Section \ref{section:padic}. 

    \begin{prop} \label{prop:stabshin}
        Let $\b$ and $\c$ be ideals as previously. We identify $(\star^*)\widehat{\kappa}_*{}^\b_\c\Theta^{Shin}(\star \iota_\alpha)$ with a class in $H^{n-1}(U^+, \D(F_p/U^+, \Z_p)^{(0_{/\c})})$ via the map \eqref{eq:distrocoord}, which is now $\star \iota_\alpha$-equivariant (because of the extra $\star$-twist on the coefficients of $\Theta^{Shin}(n)$).
        
        Then the image of this class under the chain of arrows \eqref{eq:getzeta}, when evaluated against the function
        \[
            F_p/U^+\to \overline{\Q_p}^\times, t\mapsto  \psi(t)\mathbf{N}(t)^k
        \]
        for any locally constant function $\psi: I_p/U^+ \to \overline{\Q_p}^\times$, is
        \[
        \mathbf{N}b(\zeta_I(\psi_\b, -k) - \mathbf{N}\c\zeta_{\c I}(\psi_\b, -k))
        \]
        where $\psi_\b$ is the function $t\mapsto \psi(t) (1-1_{\b I}(t))$. 
    \end{prop}
    \begin{proof}           
        We recall the definition of $(\star^*){}^\b_\c\Theta^{Shin}(\star \iota_\alpha)$ is
        \begin{equation} \label{eq:defshin}
            (\star^*)({}^\b_{\c}r)_*({}^\c_{\b}\delta)_* {}^\c_{\b}\Theta^{St}(\star \iota_\alpha).
        \end{equation}
       Let $U(\b,\c)^+\subset U^+$ be the congruence subgroup reducing to the identity modulo $\b$ and $\c$. This is a finite-index subgroup of the rank-$(n-1)$ free abelian group $U^+$; let $c_{\b,\c}^\alpha$ be the positively-oriented fundamental class of $U(\b,\c)^+$, so that $c_{\b,\c}^\alpha$ pushes forward to $[U^+:U(\b,\c)^+]c_{U^+}^\alpha$ under the subgroup inclusion. Then we compute \eqref{eq:defshin} as
       \begin{equation}
           \frac{1}{[U^+:U(\b,\c)^+]}\mathrm{cores}^{U^+}_{U(\b,\c)^+}\, \mathrm{res}^{U^+}_{U(\b,\c)^+} (\star^*)({}^\b_{\c}r)_*({}^\c_{\b}\delta)_* {}^\c_{\b}\Theta^{St}(\star \iota_\alpha).
       \end{equation}
        We spell out the inside symbol $\Xi:=\mathrm{res}^{U^+}_{U(\b,\c)^+} (\star^*)({}^\b_{\c}r)_*({}^\c_{\b}\delta)_* {}^\c_{\b}\Theta^{St}(\star \iota_\alpha)$: on a basis $u_1,\ldots, u_{n-1}$ of $U(\b,\c)^+\xrightarrow{\iota_\alpha} \GL_n(\Z)$, we obtain  
        \[
        \Xi(u_1,\ldots, u_{n-1})= {}^\b_{\c}r_*\left(-\sum_{\xi\in \widecheck{\G_m^n[\c]}-\{1\}} \sum_{y\in \G_m^n[\b]-\{1\}} [(u_1\ldots u_{n-1} \ell_{\xi, y})^\vee, \ldots, (u_1 \ell_{\xi, y})^\vee, \ell_{\xi,y}^\vee]_{\xi, y}\right)\]
        \[
        = -\sum_{\xi\in \widecheck{\G_m^n[\c]}-\{1\}} \sum_{y\in \G_m^n[\b]-\{1\}} \sum_{\chi\in \widecheck{\G_m^n[\b]}} \chi(y) f_{C(\xi,y)}(\xi^{-1}\cdot \chi)\in A_p
        \]
        where $\ell_{\xi,y}$ are any choices of rays in $(F\otimes \Q)^+$ which modulo $\b$ and $\c$ do not contain $y$, and are not contained in $\ker \xi$, and $C(\xi,y)$ is the cone bounded by the rays $u_1\ldots u_{n-1} \ell_{\xi, y}, \ldots, \ell_{\xi,y}$. From \eqref{eq:fakeshin}, Theorem \ref{thm:shin}, and Proposition \ref{prop:katz}, we then find that the realization under $\widetilde{\kappa}$ of $\Xi$ satisfies 
        \begin{align}
            \int_{(\Z_p)^\vee} \psi(\alpha^{-1}(t)) \mathbf{N}(\alpha^{-1}(t))^k\, d\widehat{\kappa}(\Xi \frown c_{\b,\c}^\alpha)(t) & =  - \sum_{\lambda \in I_p^+/U(\b,\c)^+} \sum_{\xi\in \widecheck{\G_m^n[\c]}-\{1\}} \sum_{y\in \G_m^n[\b]-\{1\}} \sum_{\chi\in \widecheck{\G_m^n[\b]}} \chi(y) \chi(\lambda)\xi^{-1}(\lambda)\mathbf{N}\lambda^k \\
            &= [U^+:U(\b,\c)^+] \mathbf{N}\b(\zeta_I(\psi_\b, -k) - \mathbf{N}\c\zeta_{\c I}(\psi_\b, -k))
        \end{align}
        where here we use that any lower-dimensional conical face of $C(\xi, y)$ has vanishing Shintani generating function when smoothed along the character $\xi$, by the condition that $\ell_{\xi,y}\not\in \ker \xi$, as in the discussion before Example \ref{ex:bcg}. The result then follows from the commutative diagram in group (co)homology
        \begin{equation}
            \begin{tikzcd}
            H^{n-1}(U(\b,\c)^+, \D(F_p, \Z_p)^{(0_{/c})}/\langle \delta_0\rangle)) \arrow[d,"\mathrm{cores}"] \arrow[r,"\frown c_{\b,\c}^\alpha"] & H_0(U(\b,\c)^+, \D(F_p, \Z_p)^{(0_{/\c})}/\langle \delta_0\rangle)) \arrow[r] \arrow[d,"\mathrm{cores}"] & \D(F_p/U(\b,\c)^+, \Z_p)^{(0_{/\c})}/\langle \delta_0\rangle) \arrow[d,"\mathrm{projection}"] \\
            H^{n-1}(U^+, \D(F_p, \Z_p)^{(0_{/c})}/\langle \delta_0\rangle)) \arrow[r,"\frown c_{U^+}^\alpha"] & H_0(U^+, \D(F_p, \Z_p)^{(0_{/\c})}/\langle \delta_0\rangle)) \arrow[r] & \D(F_p/U^+, \Z_p)^{(0_{/\c})}/\langle \delta_0\rangle)
            \end{tikzcd}
        \end{equation}
        
    \end{proof}
    
    \begin{proof}[Proof of Theorem \ref{thm:interp}]
        We have, for any ideal $\b$ in $F$ prime to $p$ and $(c)$, that
        \[
        \mathbf{N}\b([b]^*-1){}_\c\Theta^{dR}(\iota_\alpha)={}^\b_\c\Theta^{dR}(\iota_\alpha) = (\star^*){}^\b_\c\Theta^{Shin}(\star \iota_\alpha).
        \]
        It is immediate from the preceding proposition that the value of this cocycle at $u_1,\ldots, u_{n-1}$, after applying \eqref{eq:getzeta}, satisfies
        \begin{equation} \label{eq:bstabilized}
            \int_{(\Z_p^n)^\vee} \psi(\alpha^{-1}(t)) \mathbf{N}(\alpha^{-1}(t))^k \, d\widehat{\kappa}({}^\b_\c\Theta^{dR}(\iota_\alpha)\frown c_{U^+}^\alpha)(t) = \mathbf{N}\b(\zeta_{I}(\psi_\b, -k)-\mathbf{N}c\cdot \zeta_{\c I}(\psi_\b,-k))
        \end{equation}
        for any locally constant $\psi$ as previously. Thus by the functoriality noted in Remark \ref{rem:imdone}, the distribution 
        \[
        \mu_\alpha:=\widehat{\kappa}_*({}_\c\Theta^{dR}(\iota_\alpha)\frown c_{U^+}^\alpha) \in \D(F_p/U^+, \overline{\Q_p})^{(0_{/c})}
        \]
        satisfies 
        \[
        \mu_\alpha((\psi \cdot \mathbf{N}^k)\circ [\b]) - \mu_\alpha(\psi \cdot \mathbf{N}^k) =\zeta_{I}(\psi_\b, -k)-\mathbf{N}c \cdot \zeta_{\c I}(\psi_\b,-k)
        \]
        for any $\mathfrak{b}$ prime to $p$ and $\mathfrak{c}$,\footnote{Here, the inverse from the remark disappears because of the $\star$-twist coming from $\alpha: I\to (\Z^n)^\vee$.} any locally constant $\psi$, and any integer $k\ge 0$.
        
        We wish to prove that this implies $\mu_\alpha(\psi) =\zeta_I(\psi, -k) - \mathbf{N}c \cdot \zeta_{\c I}(\psi, -k)$ for any $\varphi: I_p/U^+\to \overline{\Q_p}$. Indeed, any Schwartz function on $I_p$ differs by a combination of indicator functions $1_{p^kI_p}$ (for $k\in \Z$) from a linear combination of differences of indicator functions of the form $1_{a_1+p^kI_p} - 1_{a_2+p^kI_p}$ for $a_1,a_2\in p^{k-1}I_p/p^kI_p- \{0\}$. But the various $[\b]$ transitively permute the indicator functions $1_{x+p^kI_p}$ as $x$ ranges over $p^{k-1}I_p/p^kI_p- \{0\}$, so for any combination of such functions $1_{a_1+p^kI_p} - 1_{a_2+p^kI_p}$, the interpolation property holds. In fact, such differences generate the space of locally constant functions for which $\psi(0)=0$, so by continuity we are done.
    \end{proof}

    As noted previously, Theorem \ref{thm:distros} from the introduction is an immediate corollary. 
    
    \begin{rem}
        We note that working with the $\GL_n(\Z)$-version of the duality statement was not logically necessary for the proof of Theorem \ref{thm:interp}: only the $U$-restricted duality, with the more general stabilizations at ideals $\b,\c\subset F$, is strictly required. However, we find the duality statement (and the geometry behind its proof in Appendix \ref{appendix:b}) much clearer to understand from the $\GL_n(\Z)$-equivariant perspective, because of the essential intervention of its involution $\star$ (which simply ``switches the embedding'' when acting on $U$, in a potentially confusing way if one is not keeping track of the ambient group). Thus, we emphasized its role in the exposition.
    \end{rem}

    \begin{rem} \label{rem:grossstark}
        A variant of the computation of Proposition \ref{prop:stabshin} can also be used to show that ${}^\b_\c\Theta^{Shin}(\iota_\alpha)$ is also a linear combination of $\b$-torsion translates of the cocycle of \cite{CDG}, specialized to constant polynomials and restricted to $U^+$. Just as with the cocycle of \cite{CDG}, it then should result from \cite{DKSW} and \cite{DH} that there is
        \begin{equation} \label{eq:starkint}
        \text{some linear combination of integrals like }\int_{\Z_p^n-p\Z_p^n} \log_p(\alpha^{-1}(t))\, d({}_c\Theta^{dR}(\iota_\alpha) \frown c_{U^+}^\alpha) = \log_p(u_\psi) 
        \end{equation}
        where $u_\psi$ is a Gross-Stark unit for the extension of associated to $\psi$ a narrow Hilbert class character, and the linear combination will be a sum over various coordinatizations $\alpha$ corresponding to different ideal classes $[I]$, with coefficients corresponding to $\psi$, as in the below Section \ref{section:2adic}). (By adding tame level, one should also obtain Gross-Stark units in abelian extensions of $F$ ramified at finite places.) Note that the integral is well-defined as a $p$-adic number, not just up to $U^+$-multiplication, because our cocycle is invariant under $[p]_*$, meaning that $\Z_p^n-p\Z_p^n$ has mass zero and thus the integrals of constants vanish (see Remark \ref{rem:finitetofixed}). It may be of interest to explore if our formulation of cocycles has any benefits in understanding this approach to explicit class field theory, but we do not pursue this further here. It is worth noting that we do spell out this relationship in the joint work \cite{RX2}, in a very related setting.
    \end{rem}

    \subsection{Extra $2$-adic congruences of Deligne-Ribet} \label{section:2adic}

    As an application of our formalism, we recover the exceptional $2$-adic congruences for totally odd functions proven in \cite[Theorems 8.11, 8.12]{DR}, though as before with a $\psi(0)=0$ assumption. Along the way, we discuss the interpolation property for our elements ${}_\c\zeta_p^{I,\alpha}$ when $\mathfrak{c}$ is nontrivial in the narrow class group, packaging them into $p$-adic $L$-functions for narrow Hilbert class characters.

    To put ourselves in the appropriate framework, we recall the formalism of $L$-functions associated to the narrow Hilbert class group of $F$: following \cite{Katz}, given an integral ideal $\mathfrak{N}\subset F$, we write $M_{\mathfrak{N}}$ for the monoid (under multiplication) of integral ideals of $\mathcal{F}$, modulo the equivalence relation $I\sim J$ whenever $IJ^{-1}$ is generated by a totally positive element in $1+ \mathfrak{N} J^{-1}$. We consider the inverse limit
    \[
        M_{p^\infty} := \varprojlim_k M_{p^k}.
    \]
    As noted previously, we will omit consideration of the more general construction with tame level $M_{\mathfrak{N}(p^\infty)}$ considered in \cite{Katz}; however, the same results should follow from our arguments below.
    
    Let $\mathrm{Cl}^+(F)$ denote the narrow Hilbert class group of $F$, of size $h\in \mathbb{N}$, and choose integral representatives $I_1,\ldots, I_h$ for its classes. We have the identification
    \[
        f:\bigsqcup_{i=1}^h \left(\varprojlim I_i/p^k I_i\right)/U^+  \xrightarrow{\sim} M_{p^\infty} ,  (x_k)_k \mapsto ((x_k)\cdot I_i^{-1})_k,\,\mathrm{for}\,(x_k)_k\in \varprojlim I_i/p^k I_i
    \]
    from loc. cit. The norm function on ideals extends to a continuous map $\mathbf{N}:M_{p^\infty}\to \Z_p$, and we have the relation
    \[
    (\mathbf{N}\circ f)(x)= \mathbf{N}I_i^{-1} \cdot \mathbf{N}x
    \]
    for $x\in I_i\otimes \Z_p$. Given a locally constant function $\psi: M_{p^\infty}\to \overline{\Q_p}$, we have the associated complex $L$-function analytically continuing the Dirichlet series
    \[
        \zeta_F(\psi,s) := \sum_{J} \frac{\psi([J])}{(\mathbf{N}J)^{s}}
    \]
    where the sum ranges over all integral ideals of $F$. From $\psi$, for each class $[I_i]$, we also get an associated locally constant function
    \[
    \psi_i:  (I_i)_p \to \overline{\Q_p}, t\mapsto \psi(f(t)).
    \]
    for which we can form $\zeta_{I_i}(\psi_i, s)$ as previously; then we have, from the identification coefficient-wise of the associated Dirichlet series, that
    \[
        \zeta_F(\psi,s) := \sum_{i=1}^h (\mathbf{N}I_i)^s \zeta_{I_i}(\psi_i, s).
    \]
    Accordingly, we define $p$-adic $L$-elements
    \[
        \D(M_{p^\infty}, \Z_p)/\langle \delta_{0_i}, 1\le i \le h\rangle \ni {}_\c\zeta_p^F := \sum_{i=1}^h f_*{}_\c\zeta_p^{I,\alpha},
    \]
    where $\delta_{0_i}$ is the atomic measure at $0_i$, the image of $0$ in the coordinate corresponding to $I_i$ under $f$. the We deduce the interpolation property, for any $\psi:M_{p^\infty}$ such that $\psi(0_i)=0$ for $i=1,\ldots, h$ (i.e. $\psi_i(0)=0$),
    \begin{align} \label{eq:hilbinterp}
        \int_{M_{p^\infty}} \psi(J) \mathbf{N}J^k\,d{}_\c\zeta_p^F(J) & = \sum_{i=1}^h \mathbf{N}I_i^{-k}\int_{F_p/U^+} \psi_i(t) \mathbf{N}t^k\,d{}_\c\zeta_p^F(t)\\
        & = \sum_{i=1}^h \mathbf{N}I_i^{-k} (\zeta_{I_i}(\psi_i, -k) - \mathbf{N}\c \zeta_{\c I_i}(\psi_i, -k)) \\
        & = \zeta_F(\psi, -k) - \mathbf{N}\c^{k+1} \zeta_F(\psi_{\c},-k) = (1-\psi(\c)\mathbf{N}\c^{1+k})\zeta_F(\psi, -k)
    \end{align}
    where $\psi_{\c}:M_{p^\infty} \to \overline{\Q_p}$ is the function $J\mapsto \psi(J\c)$. 

    Recall the notion of even/odd functions on $M_{p^\infty}$ from \cite[2.24]{DR}: let $G_{p^k}$ be the $p^k$-ray class group of $F$, and $G_{p^\infty}:=\varprojlim G_{p^k}$; it acts on $M_{p^\infty}$ by level-wise multiplication. Write $\Sigma$ for the subgroup generated by towers of principal ideals generated by elements $1\pmod{p^k}$. Writing $S_\infty$ for the set of infinite places of $F$, we have a canonical isomorphism $\Sigma \cong \{\pm 1\}^{S_\infty}$, the generator at the place $\nu$ given by a tower of principal ideals which are negative at precisely $\nu$ and positive at all the other places. We also have a norm map $\mathbf{N}:\Sigma\to \{\pm 1\}$ by taking the product over all places.
    
    We say the a function $\varphi:M_{p^\infty} \to \overline{\Q_p}$ has \emph{sign} $\sgn(\varphi)\in \{\pm 1\}^{S_\infty}$ if $\varphi \circ \sigma = \sgn(\varphi)\cdot \varphi$ for all $\sigma\in \Sigma$; if $\sgn(\varphi)$ consists of all $-1$s, we call $\varphi$ totally odd, and if $\sgn(\varphi)$ is all $1$s, we call it totally even.

    We now are ready to prove the analogue of \cite[(8.11--12)]{DR}:

    \begin{thm}[Theorem \ref{thm:2adic}] \label{thm:2adicbody}
        Let $\psi:M_{p^\infty} \to \overline{\Q_p}$ be a totally odd continuous function with $\psi(0_i)=0$ for $1\le i \le h$. Then 
        \[
        \int_{M_{(2^\infty)}} \psi(J)\, d{}_\c\zeta_2^F(J) \equiv 0 \pmod{2^n}.
        \]
        (Note that from our assumption on $\psi$, the invariant $\delta(\psi)$ of \cite[(8.11)]{DR} is always zero.)
    \end{thm}
    \begin{proof}
        By continuity, we may assume $\psi$ is locally constant; thus,  we can represent the elements in $\Sigma$ by actual principal ideals at some fixed level $2^k$ for some sufficiently large integer $k$, and we may consider integrals of totally odd functions $\psi$ on $M_{(2^k)}$ against the image of ${}_\c\zeta_2$ considered in $\D(M_{(2^k)}, \Z_2)/\langle\delta_{0,\bullet}\rangle$. From \cite[\S2]{DR}, the space $M_{(2^k)}$ in fact is a finite disjoint union of ray class groups, via the isomorphism
        \[
            \mathcal{D}:\bigsqcup_{\mathfrak{d} | (2^k)} G_{(2^k)\mathfrak{d}^{-1}} \xrightarrow{\sim} M_{2^k} , [J]_{(2^k)\mathfrak{d}^{-1}}\mapsto [\mathfrak{d} \cdot J]_{(2^k)}
        \]
        which affords an $\Sigma$-equivariant identification
        \[
            \mathcal{D}_*: \bigoplus_{\mathfrak{d} | (2^k), \mathfrak{d}\ne (2^k)} \mathbf{D}(G_{(2^k)\mathfrak{d}^{-1}}, \Z_2)\xrightarrow{\sim} \D(M_{(2^k)}, \Z_p)/\langle\delta_{0,\bullet}\rangle
        \]
        We have the relation
        \[
            \int_{M_{(2^k)}} \psi\, d{}_\c\zeta_2 =  \cdot \int_{G_{(2^k)\d^{-1}}} \mathcal{D}^*\psi\, d(\mathcal{D}_*)^{-1}{}_\c\zeta_2 
        \]
        for any function $\psi$ supported on $\mathcal{D}(G_{(2^k)\d^{-1}})\subset M_{(2^k)}$. Note that if $\psi$ is totally odd, then $\mathcal{D}^*\psi$ is totally odd on $G_{(2^{k})\d^{-1}}$. Thus, it suffices to prove the theorem for integrating totally odd ray class characters against $\mathcal{D}_*^{-1} {}_\c \zeta_2 \in \D(G_{(2^k)\d^{-1}}, \Z_2)$. This can be checked at any higher conductor ray class group, so without loss of generality, we consider $k$ sufficiently large so that $\Sigma$ acts freely on $G_{(2^k)\d^{-1}}$.

        By Proposition \ref{prop:signkey}, for $\sigma\in \Sigma \cong \{\pm 1\}^{S_\infty}$, we have $[\sigma]_* {}_\c\zeta_2 = \mathbf{N}\sigma \cdot {}_\c\zeta_2$. This means ${}_\c\zeta_2$ lies in the invariant submodule $(\D(G_{2^k\d^{-1}}, \Z_2)(\mathbf{N}))^{\Sigma}$, since multiplication by a generator (in $G_{2^k\d^{-1}}$) of $\sigma$ is orientation preserving exactly when $\mathbf{N}\sigma =+1$ by definition; the same is then true for $\mathcal{D}_*^{-1} {}_\c \zeta_2$ by $\Sigma$-equivariance of $\mathcal{D}$. We then have the exact sequence in Tate cohomology
        \[
        \D(G_{2^k\d^{-1}}, \Z_2)(\mathbf{N}) \xrightarrow{\mathrm{N}_\Sigma} (\D(G_{2^k\d^{-1}}, \Z_2)(\mathbf{N}))^\Sigma \to H^0_T(\Sigma, \D(G_{2^k\d^{-1}}, \Z_2)(\mathbf{N}) )=0.
        \]
        Here, 
        \[
        \mathrm{N}_\Sigma=\sum_{\sigma \in \Sigma} \sigma
        \]
        is the norm map for $\Sigma$, and the vanishing is by Shapiro's lemma for Tate cohomology, since $\Sigma$ acts freely on $\D(G_{2^k\d^{-1}}, \Z_2)(\mathbf{N})$. We thus have proven that the image of ${}_\c\zeta_2$ in $\D(G_{2^k\d^{-1}}, \Z_2)(\mathbf{N})$ is divisible by $\sum_{\sigma \in \Sigma} \sigma$, and the resulting $2$-adic divisibility for pairing it against totally odd $\psi$ then follows by \cite[Lemma 5.3]{Gro}.
    \end{proof}

    \begin{rem}
    Since our functions do not even interpolate the $L$-values of totally odd Hilbert class characters at the $s=0$ specialization (i.e. the $\psi(0)=0$ excluded above), it seems unlikely we can prove these divisibilities without some new idea. It would be of interest to see whether the cocycle of \cite{CDG}, using ``true'' Shintani decompositions and no smoothing, could recover this extra data. Also interesting to note is that in the original approach of \cite[(8.6--10)]{DR}, it is also precisely these cases (in the notation there, $\Delta_c(0, \varepsilon)$ for $\varepsilon$ totally odd) which require special extra attention, due to technicalities in the definition of weight-$1$ Hilbert-Eisenstein series.
    \end{rem}
    
    \appendix

    \section{Comparison with a cyclotomic equivariant polylogarithm} \label{appendix:a}

    In this appendix, working over the complex numbers, we argue that our de Rham \emph{left} cohomology class 
    \[
        \Theta^{dR}(n)\in H^{n-1}(\SL_n(\Z),\mathcal{M}_{\G_m^n}/\langle d\log z_1 \wedge \ldots d\log z_n\rangle)
    \]
    coincides with the \emph{right} cohomology class 
    \[
    S_{mult}[\{1\}]\in H^{n-1}(\SL_n(\Z),\mathcal{M}_{\G_m^n})
    \]
    of \cite[Théorème 1.7]{BCGV}, under the anti-involution $\gamma\mapsto \gamma^{-1}$ of $\SL_n(\Z)$. This latter class is only well-defined up to the ambiguity $H^{n-1}(\SL_n(\Z),\C)$, so what we really mean by this is that $S_{mult}[\{1\}]$ coincides with the $H^{n-1}(\SL_n(\Z),\C)$-torsor of lifts of $\Theta^{dR}(n)$ defined from transgressions of the Euler class.
    
    To briefly recall the definition of $S_{mult}[\{1\}]$, it is obtained from an equivariant (``polylogarithm'') cohomology class 
    \[
        \cZ_n \in H^{2n-1}_\Gamma(\G_m^n-\{1\}, \C)^{(0)}
    \]
    which is characterized (up to ambiguity $H^{2n-1}_\Gamma(\G_m^n, \C)^{(0)}$) by having image a positive generator of $\{1\}\in H^{2n}_\Gamma(\{1\})$ under the residue map, along with its trace-invariance. Then if $H\subset \G_m$ is a suitable set of hyperplanes through zero, the image of $\cZ_n$ under the edge map
    \[
        H^{2n-1}_\Gamma(\G_m^n-H, \C)^{(0)} \to H^{n-1}(\Gamma, H^{n}(\G_m^n-H)^{(0)})
    \]
    together with a ``formality'' isomorphism
    \[
        H^{n}(\G_m^n-H)^{(0)} \cong (\Omega^n_{\G_m^n-H})^{(0)}
    \]
    yields $S_{mult}[{0}]$, up to the earlier-specified ambiguity.\footnote{Using instead the approach of Kings-Sprang \cite{KS} and starting with a polylogarithm class in \emph{coherent} cohomology instead, the formality isomorphism can be avoided.}

    Let 
        \[
        \mathcal{D}^0\to \ldots \to \ldots \mathcal{D}^{2n}
        \]
    be the distributional de Rham complex computing the de Rham cohomology of $\G_m^n$; then the equivariant cohomology can be computed by the double complex of inhomogeneous cochains $C^\bullet(\SL_n(\Z), \mathcal{D}^\bullet)$; see \cite{RX2} or the author's thesis \cite{X}. 

    Fix a $\Delta$-extension $E$ and an Euler transgression $\phi\in C^{n-1}(\SL_n(\Z), \C)$; then one computes by the Poincaré-Lelong formula that the map
    \[
        [\Delta^E(r_1,\ldots, r_k)]\mapsto \begin{pmatrix}r_1 & \ldots & r_k\end{pmatrix}_* (d\log)^{\wedge k} \{1-z_1,\ldots, 1-z_k\}
    \]
    is a $\SL_n(\Z)$-equivariant map $\widetilde{\Chains}(n)_{\bullet}\to \mathcal{D}^{2n-\bullet}$. Then the sum of the lifts from Lemma \ref{lem:groupcohom}
    \[
        \ell_1+\ell_2+\ldots + \ell_n-\phi\cdot (d\log)^{\wedge n} \{-z_1,\ldots, -z_k\}
    \]
    is a trace-fixed class which has total coboundary $\{1\}$, and hence represents $\cZ_n$. Its image under the Hochschild-Serre edge map is $(d\log)^{\wedge k}\theta_{E,\phi}(n)$, valued in trace-invariant holomorphic forms, and which therefore is identified with $S_{mult}[\{1\}]$ under the formality isomorphism up to the specified ambiguity.

    By the same argument, using a stabilized equivariant class ${}_c\cZ_n$, the class ${}_c\Theta^{dR}(n)$ coincides with $S_{mult}[\mu_{c}^n-\{1\}]$, where now the hyperplanes $H$ can be taken to avoid lines through $\mu_{p^\infty}^n$ (and thus the resulting cohomology class extends holomorphically over all $p$-power torsion). (If one does this with $c$-smoothing rather than $c$-avoiding, one also can eliminate the ambiguity coming from the orientation obstruction.)

\section{Duality in restricted Tits buildings} \label{appendix:b}

In this appendix, we prove that ${}^c_b\delta_*{}^c_b\Theta^{St}(n)$ and $(\star^*){}^b_c\Theta^{St}(n)$ yield the same class in $H^{n-1}(\GL_n(\Z), (\star^*){}^b_c\St(n))$. To give the idea of the argument, we first treat the unstabilized version: observe that the unstabilized duality map $\delta$ of Proposition \ref{prop:conicalduality} descends to a $\GL_n(\Q)$-equivariant map
\[
\delta: \St(n)\to (\star^*)\St(n), [\ell_1]\wedge \ldots \wedge [\ell_n] \to [\ell_1^\vee]\wedge \ldots \wedge [\ell_n^\vee]
\]
by the same argument in Proposition \ref{prop:c-conicalduality}. It is indeed also the case that\footnote{See following footnote.} 
\begin{equation}  \label{eq:untwistedduality}
\delta_* \Theta^{St}(n) = - (\star^*) \Theta^{St}(n) \in H^{n-1}(\GL_n(\Q),(\star^*)\St(n)).
\end{equation}
To see this, let $T_n$ be the \emph{Tits building} for $\GL_n(\Q)$: the $(n-2)$-dimensional simplicial complex whose $i$-simplices are proper flags of length $i$ in $\Q^n$, i.e. chains of proper inclusions of proper nonzero subspaces of $\Q^n$
\[
V_1\subset \ldots \subset V_i
\]
The Solomon-Tits theorem says that $T_n$ has the homotopy type of a wedge of $(n-2)$-spheres. The group $\GL_n(\Q)$ acts on $T_n$ via its standard action on $\Q^n$, and the only nontrivial reduced homology group $\widetilde{H}_{n-2}(T_n)$ is then a model for the Steinberg representation $\St(n)$.

To describe the identification
\[ \St(n) \xrightarrow{\sim} \widetilde{H}_{n-2}(T_n),\]
we note that given the standard set on $[n]$ elements $\{1,2,\ldots, n\}$, we can form the standard $(n-1)$-simplex $\Delta^{n-1}$ whose $i$-faces correspond to the $(i+1)$-subsets of $[n]$, and inclusion of sets corresponds to the face relation. Then its first barycentric subdivision $\mathrm{sd}\,\Delta^{n-1}$ is also a contractible $(n-1)$-simplicial complex whose $i$-simplices are flags of non-empty subsets (i.e. chains of subsets of $[n]$ under proper inclusion).Then, as shown in \cite{AR}, the generator $[\ell_1]\wedge \ldots \wedge [\ell_n]$ (in our previous Orlik-Solomon-based notation) is identified with the image of the fundamental class (for the orientation corresponding to the usual order on $[n]$) under the simplex map
\begin{equation} \label{eq:simplexmap}
 \partial \mathrm{sd}\,\Delta^{n-1} \hookrightarrow T_n
\end{equation}
corresponding to the set $\ell_1,\ldots, \ell_n$, where we identify for example the flag of subsets $[1] \subset [1,3] \subset [1,3,4]$ with the flag of subspaces $\langle \ell_1\rangle \subset \langle \ell_1, \ell_3\rangle \subset \langle \ell_1, \ell_3, \ell_4 \rangle$.

Then the point is that using the lifting process in Lemma \ref{lem:groupcohom}, one straightforwardly computes that the class $\theta^{St}(n)$ arises by lifting the $\GL_n(\Q)$-fixed element $1\in \Z$ along the map $H^0(\GL_n(\Q),\Z)\to H^{n-1}(\GL_n(\Q), \widetilde{H}_{n-2}(T_n))$ coming from the exact doubly-augmented homology complex
\begin{equation} \label{eq:augtits}
C(T_n)_\bullet:=\widetilde{H}_{n-2}(T_n)\to C_{n-2}(T_n)\to \ldots \to C_{0}(T_n) \to \Z,
\end{equation}
in the same way as we previously did with the Orlik-Solomon complex. 

The Tits building admits an automorphism $\delta_T:T_n\to T_n$ sending a vertex corresponding to a subspace $V\subset \Q^n$ to its orthogonal complement $V^\perp$ under the standard inner product, and similarly on the simplices corresponding to flags (where $\perp$ reverses the inclusion chains). Further, one checks that this automorphism sends the top-degree homology class corresponding to the symbol $[\ell_1]\wedge \ldots \wedge [\ell_n]$ to $- [\ell_1^\vee] \wedge \ldots \wedge [\ell_n^\vee]$; hence, it induces the automorphism $\delta$ on $\St(n)$.

As with the duality map $\delta$, the map $\delta_T$ is not $\GL_n(\Q)$-equivariant, but it \emph{is} $\GL_n(\Q)$-equivariant for the $\star$-twisted action on the target, so induces an equivariant automorphism of homology complexes
\[
(\delta_T)_*: C(T_n)_\bullet \xrightarrow{\sim} (\star^*)C(T_n)_\bullet
\]
which evidently fixes the class of $1\in \Z$, and hence the corresponding derived classes in top homology. From this, the duality relation \eqref{eq:untwistedduality} follows immediately.\footnote{If we take de Rham realizations of $\St(n)$ directly using this duality, we obtain that $(\star^*)\Theta^{dR}(n)$ agrees with 
\[
(\gamma_1,\ldots, \gamma_{n-1})\mapsto \begin{pmatrix} e & \gamma_1 e & \ldots & \gamma_1 \ldots \gamma_{n-1}e\end{pmatrix} \frac{(-1)^nz_1\ldots z_n}{(1-z_1)\ldots (1-z_n)}
\]
if taken to be valued in $\mathcal{M}_{\G_m^n}(-\det)^{(0)}$ modulo the $\GL_n(\Q)$-orbit of classes of the form
\[
(d\log)^{\wedge n} \{-z_1, \ldots, -z_i, 1-z_{i+1},\ldots, 1-z_n\},
\]
i.e. the image of all wedges (or, dually and equivalently, the Shintani generating functions of all lower-dimensional cones). The latter class is precisely the ``naive Shintani cocycle'' which \cite{LP} proves is a cocycle modulo precisely these same relations.}

To attack the stabilized case, let $(b)=\b \cap \Z$, $(c)=\c \cap \Z$, and fix a $c$-torsion section $x$ and a $b$-torsion section $y$ such that neither lie on the line spanned by the reduction of the first basis vector $e_1$; let $\Gamma\subset \GL_n(\Z)$ (or $\Gamma\subset U$, for that respesctive case) be the congruence subgroup fixing $\xi_x$ and $y$, so that $\star\Gamma\subset \GL_n(\Z)$ (respectively, $\subset \star U$) is the subgroup fixing $x$ and $\xi_y$. Then write ${}^{\xi_x}_y\theta^{St}(n)$ for the $\Gamma$-cocycle
\[
    (\gamma_1,\ldots, \gamma_{n-1})\mapsto [\gamma_1 \ldots \gamma_{n-1} e_1] \wedge \ldots \wedge [\gamma_1 e_1]\wedge [e_1]\in {}^{\xi_x}_{y}\St(n)
\]
ad ${}^{\xi_x}_y\Theta^{St}(n)$ for its corresponding class in $H^{n-1}(\Gamma, {}^{\xi_x}_{y}\St(n))$; switching the roles of $c$ and $b$, we can analogously produce 
\[
[{}^{\xi_y}_x\theta^{St}(n)]={}^{\xi_y}_x\Theta^{St}(n)\in H^{n-1}(\star\Gamma, {}^{\xi_y}_{x}\St(n)).
\]
Then Shapiro's lemma reduces Proposition \ref{prop:titsduality} to proving that\footnote{The appearance of the extra sign here, as well as in the unstabilized version \eqref{eq:untwistedduality}, is explained by the fact that avoiding comes from lifting the sum over nontrivial torsion sections, smoothing comes from \emph{negative} the sum over the nontrivial characters.}
\begin{equation} \label{eq:xyreduction}
({}^c_b\delta)_* {}^{\xi_x}_y\Theta^{St}(n) = - (\star^*) {}^{\xi_y}_{x}\Theta^{St}(n)\in H^{n-1}(\Gamma, (\star^*){}^{\xi_y}_{x}\St(n))
\end{equation}
where we we are restricting ${}^c_d\delta$ to the summand ${}^{\xi_x}_{y}\St(n)$, which it precisely maps to the summand ${}^{\xi_y}_{x}\St(n)$. 

Let ${}^{\xi_x}_yT_n$ be the \emph{restricted} Tits building consisting of the induced subcomplex of $T_n$ on the vertices corresponding to subspaces which \emph{do not contain} the line $y$ modulo $d$, and \emph{are not contained} in the kernel of $\xi_x$ modulo c; let ${}^{\xi_y}_xT_n$ be similar, but switching the role of $x$ and $y$, $c$ and $b$. 

We write $C({}^{\xi_x}_yT_n)_\bullet$ and $C({}^{\xi_y}_xT_n)_\bullet$ for the respective doubly-augmented homology complexes of the form \eqref{eq:augtits}. Then the automorphism $\delta_T$ of $T_n$ restricts to a map ${}^{\xi_x}_yT_n\to {}^{\xi_y}_xT_n$, and hence a $\Gamma$-equivariant map 
\[
(\delta_T)_*:C({}^{\xi_x}_yT_n)_\bullet \to (\star^*)C({}^{\xi_y}_xT_n)_\bullet. 
\]
It remains only to prove that ${}^{\xi_x}_yT_n$ and ${}^{\xi_y}_xT_n$ have reduced homology concentrated only in top degree, such that that we have an identification ${}^{\xi_x}_y\St(n) \xrightarrow{\sim} \widetilde{H}_{n-2}({}^{\xi_x}_yT_n)$ fitting in a $\Gamma$-equivariant commutative diagram 
\begin{equation}
    \begin{tikzcd}
        {}^{\xi_x}_y\St(n) \arrow[r, "\sim"] \arrow[hookrightarrow]{d} & \widetilde{H}_{n-2}({}^{\xi_x}_yT_n) \arrow[d] \\
        \St(n) \arrow[r, "\sim"] & \widetilde{H}_{n-2}(T_n)
    \end{tikzcd}
\end{equation}
and similarly for ${}^{\xi_y}_xT_n$ and $\star\Gamma$. The conclusion would then follow by the same argument as the unstabilized case.

Since the two cases are symmetric in swapping $x$ and $y$, $c$ and $d$, we may as well focus only on the former case. Indeed, we note that the conditions defining ${}^{\xi_x}_y\St(n)$ precisely state that it is generated by symbols 
\[
[\ell_1]\wedge \ldots \wedge [\ell_n]
\]
for which no line is contained in the kernel of $\xi_x$ modulo $c$, and no $\le (n-1)$-subtuple span contains $y$ modulo $b$. Thus, the simplex map \eqref{eq:simplexmap} exists and makes sense when restricted to subspaces for the vertices of ${}^{\xi_x}_yT_n$, satisfying all the Steinberg relations by the same arguments in \cite{AR}. This definition for the top arrow evidently is $\Gamma$-equivariant and makes the diagram above commute, so the only thing that remains to prove is that it is an isomorphism. 

Let us write $\St(n,\F_c)$, $\St(n,\F_b)$ for the Steinberg representation for lines modulo $c$, $b$ respectively, and ${}^{\xi_x}\St(n,\F_c)$, ${}_y\St(n,\F_b)$ for the same with the corresponding local conditions from the definition of ${}^{\xi_x}_y\St(n)$.\footnote{These are (up to $\perp$-duality) the top-degree parts of the Orlik-Solomon complex for what is called the rank-$(n-1)$ \emph{affine geometry matroid} over $\F_c$ and $\F_b$, while the usual Steinberg module corresponds to the usual projective geometry matroid consisting of all lines.} Then by definition, we have a Cartesian square 
\begin{equation} \label{eq:dia1}
    \begin{tikzcd}
        {}^{\xi_x}_y\St(n) \arrow[r] \arrow[d] & \St(n) \arrow[d] \\
        {}^{\xi_x}\St(n,\F_c) \oplus {}_y\St(n,\F_b) \arrow[r] & \St(n,\F_c) \oplus \St(n,\F_b)
    \end{tikzcd}
\end{equation}

Likewise, let $T_n(\F_c),T_n(\F_b)$ be the Tits buildings built from subspaces of $\F_c^n$, $\F_b^n$, and ${}^{\xi_x}T_n(\F_c), {}_yT_n(\F_b)$ the induced subcomplexes on subspaces not contained in $\ker \xi_x$, respectively not containing $\F_by$. Then we have a Cartesian square of simplicial complexes
\begin{equation} \label{eq:dia2}
    \begin{tikzcd}
        {}^{\xi_x}_yT_n \arrow[r] \arrow[d] & T_n \arrow[d] \\
        {}^{\xi_x}T_n(\F_c) \oplus {}_yT_n(\F_b) \arrow[r] & T_n(\F_c) \oplus T_n(\F_b)
    \end{tikzcd}
\end{equation}
which leads to a Cartesian square also of the corresponding doubly-augmented homology complexes, and hence of the top-degree $\widetilde{H}_{n-2}$ groups, by the Künneth theorem. The left column of \eqref{eq:dia1} is isomorphic to the $\widetilde{H}_{n-2}$ of the left column of \eqref{eq:dia2} via the simplex maps \eqref{eq:simplexmap}; thus, if the same simplex map also induces isomorphisms
\begin{equation} \label{eq:localmaps}
{}^{\xi_x}\St(n,\F_c)\xrightarrow{\sim} \widetilde{H}_{n-2}({}^{\xi_x}T_n(\F_c)), {}_y\St(n,\F_b)\xrightarrow{\sim} \widetilde{H}_{n-2}({}_yT_n(\F_b))
\end{equation}
then we also have our desired identification ${}^{\xi_x}_y\St(n)\xrightarrow{\sim} \widetilde{H}_{n-2}({}^{\xi_x}_yT_n)$. The maps \eqref{eq:localmaps} are symmetric to one another under the automorphism $V\mapsto V^\perp$ (and replacing $c$ by $b$, etc.) so it suffices to prove the former. Indeed, this identification follows by the identification of the algebras $\mathscr{A}$ and $\mathscr{B}$ in \cite{OS}: in particular, see the remark on p. 3 relating the Orlik-Solomon complex ($\mathscr{A}$) to the construction of Steinberg modules from simplicial complexes in \cite{L} (agreeing with the construction of $\mathscr{B}$ in top degree); in the notation of loc. cit., our restricted Tits building is denoted $T(\F_c^n, \ker \xi_x)$, labelled case (b) on p. 551. 

\printbibliography

\end{document}